\documentclass[12pt,a4paper,reqno]{amsart}
\usepackage{amsfonts}
\usepackage{amsmath}
\usepackage{amssymb}
\usepackage{mathrsfs}
\usepackage{geometry}
\usepackage{graphicx}
\usepackage{subfigure}
\usepackage{hyperref}
\usepackage{microtype}
\usepackage{enumerate}
\usepackage{epsfig}
\usepackage{cancel}
\usepackage{color}
\usepackage{cite}
\usepackage{makecell}
\usepackage{multirow}
\usepackage{footmisc}
\usepackage{wrapfig}
\usepackage{array}
\usepackage{titletoc}
\usepackage{bm}
\usepackage{times}
\usepackage{tikz}
\usepackage{xcolor}
\usepackage{enumitem}

\usetikzlibrary{positioning}
\hypersetup{
colorlinks=true,
linkcolor=blue,
filecolor=blue,
urlcolor=blue,
citecolor=blue,}
\allowdisplaybreaks[3]

\newcommand{\pa}{\partial}
\newcommand{\e}{\varepsilon}
\newcommand{\om}{\omega}
\newcommand{\na}{\nabla}

\newtheorem{theorem}{Theorem}[section]

\newtheorem{lemma}[theorem]{Lemma}
\newtheorem{proposition}[theorem]{Proposition}

\newtheorem{Theorem}{Theorem}[section]

\newtheorem{Corollary}[Theorem]{Corollary}
\newtheorem{Remark}[Theorem]{Remark}

\renewcommand{\div}{\mathop{\rm div}\nolimits}
\numberwithin{equation}{section}
\theoremstyle{remark}

\begin{document}

\title[]{Asymptotic stability of the 2D Boussinesq equations around Couette flow in infinite channel}
\author[]{Qionglei Chen}
\address[]{Institute of Applied Physics and Computational Mathematics, Beijing 100088, China} 
\email{chen\_qionglei@iapcm.ac.cn}

\author[]{Jiakun Jin}

\address[]{Graduate School of China Academy of Engineering Physics, Beijing 100088, China}
\email{jinjiakun24@gscaep.ac.cn}

\author[]{Zhen Li}

\address[]{School of Mathematics, Hefei University of Technology, Hefei 230009, China}
\email{2026800053@hfut.edu.cn}

\author[]{Changxing Miao$^*$}
\address[]{School of Mathematics and Physics, University of Science and Technology Beijing, Beijing 100083, China}
\email{miao\_changxing@ustb.edu.cn}
\thanks{{$^{*}$}{Corresponding Author.} 
}

\begin{abstract}
In this paper, we investigate the stability threshold of the two-dimensional Boussinesq system around the Couette flow in an infinite channel $\mathbb{R}\times[-1,1]$ under Navier-slip boundary condition. To address nonlinear echo chains induced by temperature-vorticity coupling, we construct a two-level frequency-dependent temporal decomposition, with the two levels matched to the infinite and finite superposition principles for the equation, respectively. By fully exploiting the enhanced dissipation effect at each decomposition level, we prove that if the initial perturbation $(\omega^{in}, \theta^{in})\sim(\nu^{\frac{1}{3}},\nu^{\frac{2}{3}+})$, the Boussinesq system admits a globally asymptotically stable solution.  It should be emphasized that the method in this paper provides an effective approach for the hydrodynamic equations with coupling effects.
\end{abstract}
\maketitle

\section{Introduction}
The Boussinesq system plays a central role in the study of the laminar and turbulent regimes of Rayleigh--B\'enard convection~\cite{DR-04}. In this paper, we consider the two-dimensional incompressible Boussinesq equations in the domain $\mathbb{R}\times[-1,1]$,
\begin{align}\label{vrho-eq}
\begin{cases}
\pa_t v-\nu\Delta v+v\cdot\na v+\na p=\rho e_2, \\
\pa_t \rho-\mu\Delta \rho+v\cdot\na\rho=0, \\
\div v=0, \\
v|_{t=0}=v^{in},\quad \rho|_{t=0}=\rho^{in},
\end{cases}
\end{align}
where $v(t,x,y)$ is the velocity field, $\rho(t,x,y)$ is the temperature and $p$ is the pressure. The constants $\nu$ and $\mu$ denote the viscosity coefficient and the thermal diffusivity, respectively, and $e_2=(0,1)^T$ is the unit vector in the vertical direction.

We investigate the stability of the system \eqref{vrho-eq} around the Couette flow $(y,0)$ with constant background temperature $a\in\mathbb{R}$, under the Navier-slip boundary condition. To simplify, we focus on the case $\nu=\mu$.
Let
$$u=v-(y,0), \qquad \theta=\rho-a$$
be the velocity and temperature perturbations, respectively. Then the perturbed system reads
\begin{align*}
\begin{cases}
\pa_t u-\nu\Delta u+y\pa_x u+u\cdot\na u+\na p+(u^{(2)},0)^T=(\theta+a) e_2, \\
\pa_t \theta-\nu\Delta \theta+y\pa_x \theta+u\cdot\na\theta=0, \\
\pa_yu^{(1)}(t,x,\pm 1)=u^{(2)}(t,x,\pm 1)=\theta(t,x,\pm 1)=0, \\
u|_{t=0}=u^{in},\quad \theta|_{t=0}=\theta^{in}.
\end{cases}
\end{align*}
We introduce the vorticity $\omega=\pa_y u^{(1)}-\pa_x u^{(2)}$, and obtain the following coupled system
\begin{align}\label{ot-eq}
\begin{cases}
\pa_t \omega-\nu\Delta \omega+y\pa_x \omega+u\cdot\na \omega=-\pa_x\theta, \\
\pa_t \theta-\nu\Delta \theta+y\pa_x \theta+u\cdot\na\theta=0, \\
\omega(t,x,\pm 1)=\theta(t,x,\pm 1)=0, \\
\omega|_{t=0}=\omega^{in},\quad \theta|_{t=0}=\theta^{in}.
\end{cases}
\end{align}

The transition from laminar to turbulent flow has been a classical problem in fluid dynamics since the experiments of Reynolds~\cite{R-83}. Later,  Trefethen et al. \cite{TTRD-93} first proposed the transition threshold problem. The mathematical version was introduced by Bedrossian-Germain-Masmoudi~\cite{BGM-19} for Navier-Stokes equations. Accordingly, we present a characterization of analogous transition thresholds for  Boussinesq system. Given two norms $\|\cdot\|_{X_1}$ and $\|\cdot\|_{X_2}$, determine the transition threshold exponents $\gamma_1=\gamma_1(X_1,X_2)$ and $\gamma_2=\gamma_2(X_1,X_2)$ such that
\begin{equation*}
\begin{aligned}
&\|\omega^{in}\|_{X_1}\leq \nu^{\gamma_1} \;\text{and}\; \|\theta^{in}\|_{X_2}\leq \nu^{\gamma_2} \;\Longrightarrow\; \text{stability},\\
&\|\omega^{in}\|_{X_1}\gg \nu^{\gamma_1} \;\text{or}\; \|\theta^{in}\|_{X_2}\gg \nu^{\gamma_2} \;\Longrightarrow\; \text{instability}.
\end{aligned}
\end{equation*}

We now briefly review recent progress on system \eqref{ot-eq} in the Sobolev space. For $\mathbb{T}\times\mathbb{R}$, the stability
threshold $(\nu^{\frac{1}{2}},\nu^{\frac{11}{12}})$ was established in ~\cite{DWZ-21}. This result was later improved to
$(\nu^{\frac{1}{3}},\nu^{\frac{5}{6}})$ in ~\cite{ZZ-23},
and further sharpened to $(\nu^{\frac{1}{3}},\nu^{\frac{2}{3}})$ in ~\cite{NZ-25}. For $\mathbb{R}\times\mathbb{R}$, the threshold $(\nu^{\frac{3}{4}},\nu^{\frac{5}{4}})$ was obtained in \cite{WW-25}, then refined to $(\nu^{\frac{1}{3}+},\nu^{\frac{2}{3}+})$ in ~\cite{CWY-25}. Thresholds of $(\min\{\nu,\mu\}^{\frac{1}{2}},\min\{\nu,\mu\}^{\frac{11}{12}})$ for a finite channel with non-slip boundary conditions and $(\nu^{\frac{1}{2}},\nu^{\frac{5}{6}})$ for an infinite channel were derived in ~\cite{MZZ-23} and ~\cite{LWZ-25}, respectively. For the system \eqref{ot-eq} around Couette flow
\begin{align*}
\bar{v}=(y,0), \qquad \bar{\rho}=-\gamma^2y+1, \qquad \gamma^2\geq \tfrac{1}{4},
\end{align*} 
several stability thresholds have been established. On $\mathbb{T}\times\mathbb{R}$, the threshold $\nu^{\frac{1}{2}}$ was achieved in \cite{ZZ-23-2} for $\nu=\mu$, while the threshold $\min\{\nu,\mu\}^{\frac{1}{2}+}$ was given in \cite{LWZ-26} in the regime $\mu^3\leq\nu\leq\mu^{\frac13}$. If $\frac{\nu+\mu}{2\gamma\sqrt{\nu\mu}}<2-\e$, the threshold $\min\{\nu,\mu\}^{\frac{1}{2}}$ was obtained in \cite{RW-26-1} and then was improved to $\min\{\nu,\mu\}^{\frac{1}{3}}$  in \cite{RW-26-2} on the same domain. Work on $\mathbb{R}\times\mathbb{R}$ can be found in \cite{A-25}. For the studies related to instability are available in \cite{BBCD-23,Z-23} and so on.

Most recently, in the domain $\mathbb{R}\times[-1,1]$ under the Navier-slip boundary condition, Chen-Li-Miao~\cite{CLM-26} attained the threshold $\nu^{\frac{1}{3}}$ in Sobolev spaces for the Navier-Stokes equations. For the Boussinesq system \eqref{ot-eq}, the buoyancy coupling $\pa_x\theta$ introduces an additional energy transfer mechanism from temperature to vorticity. It is  natural to ask whether this coupling changes the transition threshold inherited from the Navier-Stokes equation. We show that the vorticity component preserves the threshold $\nu^{\frac{1}{3}}$, while the temperature component admits $\nu^{\frac{2}{3}+}$. 

 More precisely, the stability result is stated as follows.
\begin{theorem}\label{th-1.1}
  Let $\nu\in(0,1)$ and $\mathcal{M}=\mathcal{M}(D_x)$ be
  Fourier multiplier with  symbol given by
\begin{align*}
\mathcal{M}(k)=|k|^{\frac{1}{2}}+|k|.
\end{align*}
Assume that the initial data in \eqref{ot-eq} satisfy $\omega^{in}\in H^2_xH^2_y\cap L^1_xH^2_y$ and $\theta^{in}\in H^3_xH^2_y\cap \dot{W}_x^{\frac{1}{3},1}H^2_y$.  There exist harmless constants $c, C>0$ such that if
  \begin{align*}
  E_0=\|\omega^{in}\|_{H^2_xH^2_y}+\|\omega^{in}\|_{L^1_xH^2_y}\leq& c\nu^{\frac{1}{3}}, \quad F_0=\|\theta^{in}\|_{H^3_xH^2_y}+\||\pa_x|^{\frac{1}{3}}\theta^{in}\|_{L^1_xH^2_y}\leq c\nu^{\frac{2}{3}}|\ln\nu|^{-4},
  \end{align*}
  then the solution to \eqref{ot-eq} is global in time and satisfies
\begin{align*}
  \|\omega\|_{L^2}+\|\mathcal{M}\nabla\Delta^{-1}\omega\|_{L^2L^2}+\nu^{-\frac13}|\ln\nu|^{4}\|\theta\|_{L^2} \leq& C\nu^{\frac{1}{3}}.
\end{align*}

Furthermore, for any $s\in[0,\frac{1}{16}]$, we derive the following weak enhanced dissipation and inviscid damping estimates,
\begin{align*}
& \|(1+\nu^{\frac{1}{3}}t\mathcal{M}_1)^{s}\omega\|_{L^2}\leq C\nu^{\frac{1}{3}},\\
& \|(1+\nu^{\frac{1}{3}}t\mathcal{M}_1)^{s}\theta\|_{L^2}\leq C\nu^{\frac{2}{3}}|\ln\nu|^{-4},\\
& \|(1+\nu^{\frac{1}{3}}t\mathcal{M}_1)^{s}\mathcal{M}\nabla\Delta^{-1}\omega\|_{L^2L^2}\leq C\nu^{\frac{1}{3}},
\end{align*}
where $\mathcal{M}_1=\mathcal{M}_1(D_x)$ is a Fourier multiplier with symbol given by
\begin{align*}
\mathcal{M}_1(k)=|k|^{\frac{2}{3}}\chi+1-\chi,
\end{align*}
and $\chi\in C_c^\infty$ is a cut-off function defined by
\begin{align*}
\chi(k)=\begin{cases}
1, \quad \text{if} \,\, |k|\leq\frac{1}{2}, \\
0, \quad \text{if} \,\, |k|\geq 1.
\end{cases}
\end{align*}
\end{theorem}
\begin{Corollary}\label{col-1.1}
  For $\mathbb{T}\times[-1,1]$, let $\nu\in(0,1)$ and assume that the initial data in \eqref{ot-eq} satisfy $\omega^{in}\in H^2_xH^2_y$ and $\theta^{in}\in H^3_xH^2_y$.  There exist harmless constants $c, C>0$ such that if
  \begin{align*}
  \tilde{E}_0=\|\omega^{in}\|_{H^2_xH^2_y}\leq& c\nu^{\frac{1}{3}}, \quad \tilde{F}_0=\|\theta^{in}\|_{H^3_xH^2_y}\leq c\nu^{\frac{2}{3}},
  \end{align*}
  then the solution to \eqref{ot-eq} is global in time and satisfies
\begin{align*}
  \|\omega\|_{L^2}+\|\pa_x\nabla\Delta^{-1}\omega\|_{L^2L^2}+\nu^{-\frac13}\|\theta\|_{L^2} \leq& C\nu^{\frac{1}{3}}.
\end{align*}
\end{Corollary}
Indeed, for the proof of Corollary \ref{col-1.1}, the key observation is that for the linear system
\begin{align*}
\begin{cases}
\pa_t \omega_L-\nu t^2\pa_x^2\omega_L+y\pa_x\omega_L+\pa_x\theta_L=0, \\
\pa_t \theta_L-\nu t^2\pa_x^2\theta_L+y\pa_x\theta_L=0, \\
\omega_L(t,x,\pm 1)=\theta_L(t,x,\pm 1)=0, \\
\omega_L|_{t=0}=\omega^{in}, \quad \theta_L|_{t=0}=\theta^{in},
\end{cases}
\end{align*}
there is the improved inviscid damping estimates  in contrast to $\mathbb{R}\times[-1,1]$,
\begin{align*}
 &\|u^{(1)}_L\|_{L^\infty}+(1+t)\|(u^{(1)}_L-t u^{(2)}_L)\|_{L^\infty}+(1+t)\|u^{(2)}_{L}\|_{L^\infty} \nonumber\\ 
 \lesssim& (1+t)^{-1}e^{-c\nu^{\frac{1}{3}} t}(\tilde{E}_0+\nu^{-\frac{1}{3}}\tilde{F}_0),
\end{align*}
than that given in Proposition \ref{prop-2.4}. In addition, we also show the refined estimates of $(\omega-\omega_L,\theta-\theta_L)$  accordingly. For more details, see Remark \ref{rem-3.1}, \ref{rem-4.1} and \ref{rem-5.1}.

\textbf{The Main Ideas and Outline.} It is known that the quasi-linear approach is an effective way to obtain the optimal transition threshold, which has been used in literature \cite{CLM-26, CWZ-24, NZ-25, WZ-21, WZ-26}. Then we decompose $(\omega,\theta)=(\omega_L,\theta_L)+(\omega_e,\theta_e)$, where $(\omega_{L}, \theta_{L})$ satisfies the linear system
\begin{align}\label{otL-eq}
\begin{cases}
\pa_t \omega_L+\nu(1-t^2\pa_x^2)\omega_L+y\pa_x\omega_L+\pa_x\theta_L=0, \\
\pa_t \theta_L+\nu(1-t^2\pa_x^2)\theta_L+y\pa_x\theta_L=0, \\
\omega_L(t,x,\pm 1)=\theta_L(t,x,\pm 1)=0, \\
\omega_L|_{t=0}=\omega^{in}, \quad \theta_L|_{t=0}=\theta^{in},
\end{cases}\tag{$\text{Bou}_{\text{L}}$}
\end{align}
and $(\omega_{e}, \theta_{e})$ satisfies the nonlinear system
\begin{align}\label{ote-eq}
\begin{cases}
\pa_t \omega_e-\nu\Delta\omega_e+y\pa_x\omega_e+u\cdot\na \omega_e+u_e\cdot\na \omega_L+E_\mathfrak{L}+\pa_x\theta_e=0, \\
\pa_t \theta_e-\nu\Delta\theta_e+y\pa_x\theta_e+u\cdot\na \theta_e+u_e\cdot \na \theta_L+T_\mathfrak{L}=0, \\
\omega_e(t,x,\pm 1)=\theta_e(t,x,\pm 1)=0, \\
\omega_e|_{t=0}=\theta_e|_{t=0}=0,
\end{cases}\tag{$\text{Bou}_{\text{e}}$}
\end{align}
with
\begin{align*}
&E_\mathfrak{L}=E_{r_L}+u_L\cdot\na\omega_L,\quad E_{r_L}=\pa_t\omega_L-\nu\Delta\omega_L+y\pa_x\omega_L+\pa_x\theta_L,\\
&T_\mathfrak{L}=T_{r_L}+u_L\cdot\na\theta_L, \quad \,\, T_{r_L}=\pa_t\theta_L-\nu\Delta\theta_L+y\pa_x\theta_L.
\end{align*}

While relaxing the magnitude of initial perturbations as much as possible,  the coupling effect gives rise to new difficulties in both the linear system \eqref{otL-eq} and nonlinear system \eqref{ote-eq} in contrast to the Navier-Stokes system.

The linear system \eqref{otL-eq} admits the explicit solution
\begin{align*}
\begin{cases}
\omega_{L}=e^{-\nu t+\frac{1}{3}\nu  t^3\pa_x^2-yt\pa_x}(\omega^{in}-t\pa_x\theta^{in}), \\
\theta_{L}=e^{-\nu t+\frac{1}{3}\nu  t^3\pa_x^2-yt\pa_x}\theta^{in}.
\end{cases}
\end{align*}
It is obvious that the term $t\pa_x\theta^{in}$ leads to linear
transient growth, analogous to the lift-up effect. We introduce the modified vorticity,
\begin{align*}
\alpha_{L}=\omega_{L}+t\pa_x\theta_{L},
\end{align*}
 where $(\alpha_L,\theta_L)$ provide the weak enhanced dissipation and inviscid damping estimates.

Next, we highlight the main difficulties incurring in the nonlinear system \eqref{ote-eq}, dividing it into two time regimes
$t\leq\nu^{-\frac{1}{6}}$ and $t>\nu^{-\frac{1}{6}}$.

\medskip\noindent\textbf{\emph{Short-time regime.}} For $t\leq\nu^{-\frac{1}{6}}$, we shall apply the energy method with a carefully chosen workspace. The main obstruction is to balance the interaction between the coupling effect
  and the nonlinear advection.  More precisely,
  \begin{itemize}
\item
  If we employ the rough tool to close the estimates, bounding $\partial_x\theta_e$ in the $L_t^\infty$ space and directly using viscous dissipation to control the term $\partial_x^2\theta_e$ arising from $\partial_x(u_e\cdot\nabla\theta_e)$, it will bring a $\nu^{-\frac{1}{2}}$ loss and impose the condition $|\theta^{\mathrm{in}}|\lesssim \nu^{\frac{5}{6}}$. However, from a dynamical point of view, such loss does not appear to be essential. The key point is to establish  $L_{t}^\infty$ estimate for $\partial_x^2\theta_e$, which allows us  to relax the restriction to $|\theta^{\mathrm{in}}|\lesssim \nu^{\frac{2}{3}+}$.

\item 
To relax the admissible scale of $\theta_e$ as much as possible, we require a more refined treatment of the transport term $u_e\cdot\na\theta_e$. To this end, we introduce a $\nu-$dependent frequency decomposition to reorganize energy allocation among different modes, which in particular yields a gain of $\nu^{\frac{1}{12}}|\ln\nu|^{-1}$ power for $\pa_x^2\theta_e$.
\end{itemize}

\medskip\noindent\textbf{\emph{Long-time regime.}}
For $t>\nu^{-\frac{1}{6}}$, we further decompose $(\omega_e,\theta_e)=(\omega_\text{b},\theta_\text{b})+(\omega_\mathfrak{E},\theta_\mathfrak{E})$,  where $(\omega_\text{b},\theta_\text{b})$ satisfy
\begin{align*}
\begin{cases}
\pa_t \omega_\text{b}-\nu\Delta\omega_\text{b}+y\pa_x\omega_\text{b}+u\cdot\na \omega_e+u_e^{(1)}\pa_x \omega_L+u_e^{(2)}(\pa_y+t\pa_x) \omega_L+E_\mathfrak{L}=-\pa_x\theta_\text{b}, \\
\pa_t \theta_\text{b}-\nu\Delta\theta_\text{b}+y\pa_x\theta_\text{b}+u\cdot\na \theta_e+u_e^{(1)}\pa_x \theta_L+u_e^{(2)}(\pa_y+t\pa_x) \theta_L+T_\mathfrak{L}=0, \\
\omega_\text{b}(t,x,\pm 1)=\theta_\text{b}(t,x,\pm 1)=0, \\
\omega_\text{b}|_{t=T_0}=\omega_e(T_0),\quad  \theta_\text{b}|_{t=T_0}=\theta_e(T_0),
\end{cases}
\end{align*}
and $(\omega_\mathfrak{E},\theta_\mathfrak{E})=(\omega_{r},\theta_{r})+(\omega_\mathfrak{o},0)$ satisfy
\begin{align*}
\begin{cases}
\pa_t \omega_{{r}}-\nu\Delta\omega_{{r}}+y\pa_x\omega_{{r}}=-\pa_x\theta_{r}, \\
\pa_t \theta_{r}-\nu\Delta\theta_{r}+y\pa_x\theta_{r}=tu_e^{(2)}\pa_x\theta_L, \\
\omega_{r}(t,x,\pm 1)=\theta_{r}(t,x,\pm 1)=0, \\
\omega_{r}|_{t=T_0}=\theta_{r}|_{t=T_0}=0,
\end{cases}
\end{align*}
and
\begin{align*}
\begin{cases}
\pa_t \omega_{\mathfrak{o}}-\nu\Delta\omega_\mathfrak{o}+y\pa_x\omega_\mathfrak{o}=tu_e^{(2)}\pa_x\omega_L, \\
\omega_\mathfrak{o}(t,x,\pm 1)=0, \quad \omega_\mathfrak{o}|_{t=T_0}=0,
\end{cases}
\end{align*}
with $T_0=\nu^{-\frac{1}{6}}$.

 To obtain the sharp stability threshold $(\omega,\theta)\sim (\nu^{\frac{1}{3}},\nu^{\frac{2}{3}+})$, the main difficulty for system \eqref{ote-eq} lies in the term  $|\pa_x|^{\frac{1}{3}}\theta_e$.  Physically, this quantity is associated with the buoyancy coupling, at the enhanced-dissipation time scale $(\nu k^2)^{-\frac{1}{3}}$,  it measures the vorticity generated by the horizontal buoyancy gradient through the baroclinic mechanism.

\vskip 0.03in

\textbf{\textrm{I}.} For $(\omega_\text{b}, \theta_\text{b})$,  $|\pa_x|^{\frac{1}{3}}$-derivative to the equation of $\theta_\text{b}$ destroys the transport structure $u_e\cdot\na\theta_e$. This generates a commutator and hence a loss of  $|\pa_x|^{\frac{1}{3}}$-derivative. Since a sufficiently high-order estimate for $u_e$ is unavailable, controlling this commutator costs an extra factor of $\nu^{-\frac{1}{6}}$. To get around this obstacle,  we adopt a $\nu$-dependent frequency decomposition of $\pa_x u_e$ analogous to that used for short-time regime, with different splitting points.

\vskip 0.03in

\textbf{\textrm{I\!I}.} For $\omega_{\mathfrak{o}}$, it can be treated by the same argument as the Navier-Stokes equations.

\vskip 0.03in

\textbf{\textrm{I\!I\!I}.} For $(\omega_{r}, \theta_{r})$, the trouble is to control nonlinear echo chains associated with the temperature-vorticity coupling.  The $|\partial_x|^{\frac{1}{3}}\theta_r$ produces an extra  $\frac{1}{3}$-order horizontal derivative in the reaction term $t u_e^{(2)}\partial_x\theta_L$. To overcome the increased difficulty arising from the growth factor $t$ and the additional $\frac{1}{3}$-order derivative loss, we develop a two-level frequency-dependent temporal decomposition technique,  with the two levels matched to the infinite and finite principles to the equation accordingly combined with some resolvent estimates. More precisely, 

\vskip 0.03in

\textbf{\textrm{I\!I\!I-1}.} For $|k|\geq \nu^{-\frac{1}{2}}$, viscous dissipation is sufficient to compensate for the additional streamwise loss. We shall utilize energy estimates combined with the suitable resolvent bounds to reduce $\nu^{-\frac{1}{6}}$ loss.

\vskip 0.03in

\textbf{\textrm{I\!I\!I-2}.} For $0<|k|\leq \nu^{-\frac{1}{2}}$, define $$\omega_{r,k}=\int_{\mathbb{R}}e^{-ikx}\omega_{r}dx, \quad \theta_{r,k}=\int_{\mathbb{R}}e^{-ikx}\theta_{r}dx.$$ 
   We first use the dyadic-decomposition on $(\nu^{-\frac{1}{6}},+\infty)$, and divide  
$(\omega_{r,k},\theta_{r,k})\!=\!\sum_{j\geq 1}\!(\omega_{r,k,j},\theta_{r,k,j})$, 
where
 \begin{align*}
\begin{cases}
\pa_t \omega_{r,k,j}-\nu(\pa_y^2-k^2)\omega_{r,k,j}+iky\omega_{r,k,j}=-ik\theta_{r,k,j}, \quad t>T_{j-1}\\
\pa_t \theta_{r,k,j}-\nu(\pa_y^2-k^2)\theta_{r,k,j}+iky\theta_{r,k,j}=\int_{\mathbb{R}}iltu_{e,k-l}^{(2)}\theta_{L,l}dl\chi_j, \\
\omega_{r,k,j}(t,k,\pm 1)=\theta_{r,k,j}(t,k,\pm 1)=0, \\
\omega_{r,k,j}|_{t=T_{j-1}}=\theta_{r,k,j}|_{t=T_{j-1}}=0,
\end{cases}
\end{align*}
with $T_j=2^{j}\nu^{-\frac{1}{3}}(j\geq 1)$ and 
\begin{align}\label{chi-t-j}
\chi_j=
\begin{cases}
1, \quad t\in(T_{j-1},T_j],\\
0, \quad t>T_j.
\end{cases}
\end{align}
We further clarify $(\omega_{r,k}(t),\theta_{r,k}(t))=\sum_{j'=1}^{j}(\omega_{r,k,j'}(t),\theta_{r,k,j'}(t))$ for $t\in I_j:=(T_{j-1},T_j]$, and $(\nu^{-\frac{1}{6}},+\infty)=\cup_jI_j$. Then we apply the infinite superposition principles on $I_j$.

\vskip 0.03in

\textbf{\textrm{I\!I\!I-2-1}.}  $|k|\leq 1$. The above process suffices to control the echo cascades since the factor $|k|^{\frac{1}{3}}$ is negligible. 

\vskip 0.03in

\textbf{\textrm{I\!I\!I-2-2}.} $1\leq |k|\leq \nu^{-\frac{1}{2}}$. 
By making use of the  dyadic decomposition and the above infinite superposition, our task can be reduced to the problem concerning convergence series. Exploiting enhanced dissipation on each dyadic interval, we obtain $$\|\omega_r\|_X+\nu^{-\frac{1}{3}}\||\pa_x|^{\frac{1}{3}}\theta_r\|_X\lesssim \nu^{-\frac{2}{3}}\||\pa_x|^{\frac{1}{3}}\mathcal{M}u_e\|_{L^2_tL^2_{x,y}}F_0,$$  but will incur a loss of $\nu^{-\frac{1}{6}}$. Hence, we fail to achieve the sharp threshold $(\omega^{in}, \theta^{in})\sim(\nu^{\frac{1}{3}}, \nu^{\frac{2}{3}})$.
To make up for this loss, the crucial novelty is constructing the second-level, frequency-dependent decomposition on $I_j$ and splitting the equation accordingly.  This together with the enhanced dissipation in each additive interval, provides us the gain of one streamwise derivative for the forcing term. Specifically, 
\begin{itemize}
\item   
     We employ the same-scale additive decomposition on $I_j$ to decompose $\theta_{r,k,j}$ as $$ \theta_{r,k,j}=\int_{l\in\mathbb{R}}\sum_{1\leq m \leq\mathfrak{m}_j}\theta_{r,k,j,l,m}dl$$ with  $\mathfrak{m}_j:={[2^{j}|k|^{\frac{2}{3}}]+1}$. For readability, we set $\tilde{\theta}_m=\theta_{r,k,j,l,m}$, where $\tilde{\theta}_m$ satisfies
\begin{align*}
\qquad \quad
\begin{cases}
\pa_t \tilde{\theta}_{m}-\nu(\pa_y^2-k^2)\tilde{\theta}_{m}+iky\tilde{\theta}_{m}=iltu_{e,k-l}^{(2)}\theta_{L,l}\tilde{\chi}_m, \quad t\in(\tau_{m-1},T_j],\\
\tilde{\theta}_{m}(\pm 1)=0, \qquad \tilde{\theta}_{m}|_{t=\tau_{m-1}}=0,
\end{cases}
\end{align*}
with $\tau_m=T_{j-1}+\nu^{-\frac{1}{3}}|k|^{-\frac{2}{3}}m$ and
\begin{align*}
\tilde{\chi}_{m}=
\begin{cases}
1, \quad t\in(\tau_{m-1},\tau_m],\\
0, \quad t\in(\tau_m,T_j].
\end{cases}
\end{align*}

 We further clarify  $\sum_{m'= 1}^ {\mathfrak{m}_j}\tilde{\theta}_{m'}(t)=\sum_{m'=1}^m\tilde{\theta}_{m'}(t)$ when $t\in (\tau_{m-1},\tau_m]$, and
 \begin{align*}
   I_j=&\mathop{\textstyle\bigcup}_{1\leq m\leq\mathfrak{m}_j}(\tau_{m-1},\tau_{m}]\backslash(T_j,\tau_{\mathfrak{m}_j}]:=\mathop{\textstyle\bigcup}_{1\leq m\leq\mathfrak{m}_j}\mathcal{I}_m.
 \end{align*}
Then we apply the finite superposition principle to the above equation on $\mathcal{I}_m$. We point out that the  interval length $\tau_m-\tau_{m-1}=\nu^{-\frac{1}{3}}|k|^{-\frac{2}{3}}$ helps us to reduce the horizontal derivatives  requirement on each sub-time interval via the frequency-derivative transformation. This process offset the $\nu^{-\frac{1}{6}}$-loss in each sub-time interval.

\item However, this compensation come at the cost of a polynomial rate growth when summing over the decomposed pieces.   Let $\tilde{\theta}=\sum_{m= 1}^ {\mathfrak{m}_j}\theta_{r,k,j,l,m}$. Roughly speaking,    \begin{align*}
 \qquad \qquad\ \, \|\tilde{\theta}\|_{(T_{j-1},T_j]}^2\!\leq \sum_{n=1}^{\mathfrak{m}_j}\|\tilde{\theta}\|_{(\tau_{n-1},\tau_n]}^2\!\lesssim\sum_{n=1}^{\mathfrak{m}_j} \Big\|\sum_{m=1}^n\tilde{\theta}_m\Big\|_{(\tau_{n-1},\tau_n]}^2\!\lesssim\sum_{n=1}^{\mathfrak{m}_j}n^2 \sum_{m=1}^n\|\tilde{\theta}_m\|_{(\tau_{n-1},\tau_n]}^2,
\end{align*}
where we omit the concrete definition of the norm here and below for simplicity. To overcome this difficulty, we establish the space-time estimate  
     \begin{align*}
\qquad \qquad \|\tilde{\theta}_m\|_{(\tau_{n-1},\tau_n]}\lesssim e^{-c\nu^{\frac{1}{3}}|k|^{\frac{2}{3}}(\tau_{n-1}-\tau_m)}\|\tilde{\theta}_m\|_{(\tau_{m-1},\tau_m]}\lesssim e^{-c(n-1-m)}\|\tilde{\theta}_m\|_{(\tau_{m-1},\tau_m]}
\end{align*}
for the dissipative equation,  rather than in a purely inviscid framework as shown in \cite{CLM-26, WZ-26}. The enhanced dissipation yields sufficient temporal decay to absorb this summation growth and close the estimates.
 
 \end{itemize}

\medskip
\noindent{\bf Notations.}
Throughout this paper, we write $A\lesssim B$ if $A\leq CB$ holds for a constant $C>0$ independent of all relevant parameters.  We define the Fourier transform in the $x$-variable by
\begin{align*}
f_k=\int_{\mathbb{R}}e^{-ikx}f(t,x,y)dx.
\end{align*}
For $t_0, t_1\in [0,+\infty)$, we also define the following norms
\begin{align*}
&\|f\|_{W(t_0,t_1)}:=\|f\|_{L^\infty_t(t_0,t_1;L^2_{x,y})}+\nu^{\frac{1}{2}}\|\nabla f\|_{L^2_t(t_0,t_1;L^2_{x,y})},\\
&\|f\|_{X(t_0,t_1)}:=\|f\|_{L^\infty_t(t_0,t_1;L^2_{x,y})}+\nu^{\frac{1}{2}}\|\na f\|_{L^2_t(t_0,t_1;L^2_{x,y})}+\|\mathcal{M}\na \Delta^{-1}f\|_{L^2_t(t_0,t_1;L^2_{x,y})}\\
&\qquad \qquad \quad+\nu^{\frac{1}{2}}\|\pa_x f\|_{L^1_t(t_0,t_1;L^2_{x,y})},\\
&\|f\|_{Y(t_0,t_1)}:=\|f\|_{L^\infty_t(t_0,t_1;L^2_{x,y})}+\nu^{\frac{1}{2}}\|\na f\|_{L^2_t(t_0,t_1;L^2_{x,y})}+\|\mathcal{M}\na \Delta^{-1}f\|_{L^2_t(t_0,t_1;L^2_{x,y})},\\
&\|f\|_{\mathcal{Y}(t_0,t_1)}:=\|f\|_{L^\infty_t(t_0,t_1;L^2_y)}+\nu^{\frac{1}{2}}\|(\pa_y,k) f\|_{L^2_t (t_0,t_1;L^2_y)}\nonumber\\
&\qquad \qquad \quad+\|(|k|^{\frac{1}{2}}+|k|)(\pa_y,k)(\pa_y^2-k^2)^{-1}f\|_{L^2_t(t_0,t_1;L^2_y)},\\
&\|f\|_{\mathcal{Z}_m}:=\|f\|_{\mathcal{Y}(\tau_{m-1},\tau_m)}+\nu^{\frac{1}{6}}|k|^{\frac{1}{3}}\|f\|_{L^2_t(\tau_{m-1},\tau_m;L^2_y)}.
\end{align*}

To streamline the presentation, we set
\begin{align}\label{A-S-B-S}
 A_{s}=\sum_{j=0}^s\|(\pa_y,k)^{j}\omega_{k}^{in}\|_{L^2_y},  \quad B_{s}=\sum_{j=0}^s\|(\pa_y,k)^{j}\theta_{k}^{in}\|_{L^2_y}, \quad G_0=E_0+\nu^{-\frac{1}{3}}F_0.
 \end{align}

\section{\texorpdfstring{Estimates for the linear system \eqref{otL-eq}}%
  {Estimates for the linear system (BouL)}}\label{Linear-es}
In this section, we establish the enhanced dissipation and inviscid damping estimates for the linear system \eqref{otL-eq}.

 We start by applying the horizontal Fourier transform to \eqref{otL-eq},
\begin{align}\label{o-t-L-eq-k}
\begin{cases}
\pa_t \omega_{L,k}+\nu(1+k^2t^2)\omega_{L,k}+iky\omega_{L,k}+ik\theta_{L,k}=0, \\
\pa_t \theta_{L,k}+\nu(1+k^2t^2)\theta_{L,k}+iky\theta_{L,k}=0, \\
\omega_{L,k}(t,k,\pm 1)=\theta_{L,k}(t,k,\pm 1)=0, \\
\omega_{L,k}|_{t=0}=\omega^{in}_k, \quad \theta_{L,k}|_{t=0}=\theta^{in}_k.
\end{cases}
\end{align}
\begin{lemma}\label{lem-2.0}
Assume that $\phi_{L,k}=(\pa^2_y-k^2)^{-1}\omega_{L,k}$. Then the following estimates hold
\begin{align}
&\|\pa_y(\pa_y+ikt)\phi_{L,k}\|_{L^2_y}\lesssim\Big(\frac{\textbf{\textup 1}_{|k|\leq 1}}{1+|k|t}+\frac{\textbf{\textup 1}_{|k|\geq 1}}{|k|^2(1+t)}\Big)\|(\pa_y,k)^2(e^{ikyt}\omega_{L,k})\|_{L^2_y}, \label{py-py-psi-es-2-k}\\
&\|\pa_y\phi_{L,k}\|_{L^2_y}
\lesssim\Big(\frac{\textbf{\textup 1}_{|k|\leq 1}}{1+|k|t}+\frac{\textbf{\textup 1}_{|k|\geq 1}}{|k|^2(1+t)}\Big)\|(\pa_y,k)(e^{ikyt}\omega_{L,k})\|_{L^2_y}, \label{py-psi-es-2-k}\\
&\|\phi_{L,k}\|_{L^2_y}
\lesssim\Big(\frac{\textbf{\textup 1}_{|k|\leq 1}}{(1+|k|t)^2}+\frac{\textbf{\textup 1}_{|k|\geq 1}}{|k|^4(1+t)^2}\Big)\|(\pa_y,k)^2(e^{ikyt}\omega_{L,k})\|_{L^2_y}. \label{psi-es-2-k}
\end{align}
\end{lemma}
\begin{proof}
The proof can be derived by replacing $\psi$ and $w$ with $\phi_{L,k}$ and $\omega_{L,k}$ respectively in \cite[Lemma A.4]{CLM-26}, together with the Poincar\'e inequality.
\end{proof}
\begin{lemma}
\cite[Lemma A.3]{CLM-26} Let $f$ satisfy the linear initial-boundary value problem
\begin{align*}
\begin{cases}
\pa_t f+\nu(1+k^2t^2)f+iky f=0, \\
f(t,k,\pm 1)=0, \qquad f|_{t=0}=f^{in}.
\end{cases}
\end{align*}
For $s\in \mathbb{N}$,  the following estimates hold
\begin{align}
&\|(\pa_y,k)^s(e^{ikyt}f)\|_{L^2_y}\lesssim e^{-\nu t-\frac{1}{3}\nu k^2 t^3}\|f^{in}\|_{H^s},\label{equ-linear-f-es-1}\\
&\|(\pa_y,k)^s(e^{ikyt}f)\|_{L^\infty_y}\lesssim \min
\{1,|k|^{-\frac{1}{2}}\}e^{-\nu t-\frac{1}{3}\nu k^2 t^3}\|f^{in}\|_{H^{s+1}}.\label{equ-linear-f-es-2}
\end{align}
\end{lemma}
Now, we return to the estimates for the solution of the system \eqref{otL-eq}. 
\begin{lemma}\label{lem-2.3}
Let $(\omega_{L,k},\theta_{L,k})$ be the solution of \eqref{o-t-L-eq-k} and $u_{L,k}=(\pa_y, -ik)(\pa_y^2-k^2)^{-1} \omega_{L,k}$.  For $s\in \mathbb{N}$ and the constant $0<c<\frac{1}{3}$, there holds,
\begin{align}
&\|(\pa_y,k)^{s}(e^{ikyt}\omega_{L,k})\|_{L^2_y}\lesssim e^{-\nu t-c\nu k^2 t^3}(A_{s}+\nu^{-\frac{1}{3}}|k|^{\frac{1}{3}}B_{s}), \label{oL-es-2-k}\\
&\|(\pa_y,k)^{s}(e^{ikyt}\theta_{L,k})\|_{L^2_y}\lesssim e^{-\nu t-c\nu k^2 t^3}B_{s}, \label{tL-es-2-k}\\
&\|(\pa_y,k)^{s}(e^{ikyt}\omega_{L,k})\|_{L^\infty_y}\lesssim \min
\{1,|k|^{-\frac{1}{2}}\}e^{-\nu t-c\nu k^2 t^3}(A_{s+1}+\nu^{-\frac{1}{3}}|k|^{\frac{1}{3}}B_{s+1}), \label{oL-es-inf-k}\\
&\|(\pa_y,k)^{s}(e^{ikyt}\theta_{L,k})\|_{L^\infty_y}\lesssim \min
\{1,|k|^{-\frac{1}{2}}\}e^{-\nu t-c\nu k^2 t^3}B_{s+1}, \label{tL-es-inf-k}
\end{align}
and for $j=1,2,$
\begin{align}
\|e^{ikyt}u^{(j)}_{L,k}\|_{L^2_y}\lesssim&\Big(\frac{\textbf{\textup 1}_{|k|\leq 1}|k|^{j-1}}{(1+|k|t)^{j}}+\frac{\textbf{\textup 1}_{|k|\geq 1}}{|k|^{j+1}(1+t)^{j}}\Big) e^{-\nu t-c\nu k^2 t^3}(A_{j}+\nu^{-\frac{1}{3}}|k|^{\frac{1}{3}}B_{j}),\label{ujL-es-2-k}\\
\|e^{ikyt}u^{(j)}_{L,k}\|_{L^\infty_y}\lesssim &\Big(\frac{\textbf{\textup 1}_{|k|\leq 1}|k|^{j-1}}{(1+|k|t)^{j}}+\frac{\textbf{\textup 1}_{|k|\geq 1}}{|k|^{j+1}(1+t)^{j}}\Big) e^{-\nu t-c\nu k^2 t^3}(A_{j+1}+\nu^{-\frac{1}{3}}|k|^{\frac{1}{3}}B_{j+1}),\label{ujL-es-inf-k}
\end{align}
where $A_s,B_s,A_j,B_j$ are defined as in \eqref{A-S-B-S}.
\end{lemma}

\begin{proof}
For convenience, we define $\alpha_{L,k}:=\omega_{L,k}+ikt\theta_{L,k}$, satisfying 
\begin{align*}
\begin{cases}
\pa_t \alpha_{L,k}+\nu(1+k^2t^2)\alpha_{L,k}+iky\alpha_{L,k}=0, \\
\alpha_{L,k}(t,k,\pm 1)=0, \qquad \alpha_{L,k}|_{t=0}=\omega_k^{in}.
\end{cases}
\end{align*}
In view of \eqref{equ-linear-f-es-1} and \eqref{equ-linear-f-es-2}, we obtain
\begin{align}
&\|(\pa_y,k)^s(e^{ikyt}\alpha_{L,k})\|_{L^2_y}\lesssim e^{-\nu t-\frac{1}{3}\nu k^2 t^3}A_{s},\label{py-ks-al-es-2}\\
&\|(\pa_y,k)^s(e^{ikyt}\theta_{L,k})\|_{L^2_y}\lesssim e^{-\nu t-\frac{1}{3}\nu k^2 t^3}B_{s},\label{py-ks-th-es-2}\\
&\|(\pa_y,k)^s(e^{ikyt}\alpha_{L,k})\|_{L^\infty_y}\lesssim \min
\{1,|k|^{-\frac{1}{2}}\}e^{-\nu t-\frac{1}{3}\nu k^2 t^3}A_{s+1},\label{py-ks-al-es-inf}\\
&\|(\pa_y,k)^s(e^{ikyt}\theta_{L,k})\|_{L^\infty_y}\lesssim \min
\{1,|k|^{-\frac{1}{2}}\}e^{-\nu t-\frac{1}{3}\nu k^2 t^3}B_{s+1}.\label{py-ks-th-es-inf}
\end{align}
Letting $\omega_{L,k}=\alpha_{L,k}-ikt\theta_{L,k}$, together with \eqref{py-ks-al-es-2} and \eqref{py-ks-th-es-2}, we derive
\begin{align*}
&\|(\pa_y,k)^s(e^{ikyt}\omega_{L,k})\|_{L^2_y}\lesssim \|(\pa_y,k)^s(e^{ikyt}\alpha_{L,k})\|_{L^2_y}+|k|t\|(\pa_y,k)^s(e^{ikyt}\theta_{L,k}))\|_{L^2_y}\nonumber\\
&\lesssim e^{-\nu t-\frac{1}{3}\nu k^2 t^3}A_{s}+|k|te^{-\nu t-\frac{1}{3}\nu k^2 t^3}B_{s}
\lesssim e^{-\nu t-c\nu k^2 t^3}(A_{s}+\nu^{-\frac{1}{3}}|k|^{\frac{1}{3}}B_{s}).
\end{align*}
Similarly, by \eqref{py-ks-al-es-inf} and \eqref{py-ks-th-es-inf}, we have
\begin{align*}
\|(\pa_y,k)^s(e^{ikyt}\omega_{L,k})\|_{L^\infty_y}\lesssim \min
\{1,|k|^{-\frac{1}{2}}\}e^{-\nu t-c\nu k^2 t^3}(A_{s+1}+\nu^{-\frac{1}{3}}|k|^{\frac{1}{3}}B_{s+1}).
\end{align*}

For $j=1,2$, using \eqref{py-py-psi-es-2-k}--\eqref{psi-es-2-k}, and the interpolation inequality, we directly obtain \eqref{ujL-es-2-k}--\eqref{ujL-es-inf-k}. This completes the proof of Lemma~\ref{lem-2.3}.
\end{proof}

Then we have the following weak enhanced dissipation and inviscid damping estimates.

\begin{proposition}\label{prop-2.4}
Let $(\omega_{L},\theta_{L})$ be the solution of \eqref{otL-eq} and $u_{L}=(\pa_y, -\pa_x)\Delta^{-1} \omega_{L}$.
Define  
\begin{align*}
D_x^s=
\begin{cases}
\pa_x^s, \quad s\in \mathbb{N},\\
|\pa_x|^s, \quad s\in [0,+\infty)\backslash \mathbb{N}.
\end{cases}
\end{align*}
and $L^2:=L^2_{x,y},L^\infty:=L^\infty_{x,y}$. There hold
\begin{itemize}
\item[1.] Weak enhanced dissipation estimates:
If $s_1\geq 0$, $s_2\in\mathbb{N}$, $s_1+s_2\leq 2$, we have
\begin{align}
&\|D_x^{s_1}(\pa_y+t\pa_x)^{s_2}\omega_{L}\|_{L^2}\lesssim (1+\nu t^3)^{-\frac{1}{4}-\frac{s_1}{2}}e^{-\nu t}G_0, \label{px-py-oL-es-2}\\
&\|D_x^{s_1}(\pa_y+t\pa_x)^{s_2}\omega_{L}\|_{L^\infty}\lesssim (1+\nu t^3)^{-\frac{1}{2}-\frac{s_1}{2}}e^{-\nu t}G_0. \label{px-py-oL-es-inf}
\end{align}
If $s_1\geq 0$, $s_2\in\mathbb{N}$, $s_1+s_2\leq 3$, we have
\begin{align}
&\|D_x^{s_1}(\pa_y+t\pa_x)^{s_2}\theta_{L}\|_{L^2}\lesssim (1+\nu t^3)^{-\frac{1}{4}-\frac{s_1}{2}}e^{-\nu t}F_0, \label{px-py-tL-es-2}\\
&\|D_x^{s_1}(\pa_y+t\pa_x)^{s_2}\theta_{L}\|_{L^\infty}\lesssim (1+\nu t^3)^{-\frac{1}{2}-\frac{s_1}{2}}e^{-\nu t}F_0. \label{px-py-tL-es-inf}
\end{align}
\item[2.] Inviscid damping estimates:
If $s_1\in[0,2]$, $s_2\in[0,1]$, we have
\begin{align}
&\|\mathcal{M}u^{(1)}_L\|_{L^2}\lesssim (1+t)^{-1}(1+\ln(1+t))e^{-\nu t}G_0,\label{px-py-u1L-es-2}\\
&\|\mathcal{M}u^{(2)}_{L}\|_{L^2}\lesssim (1+t)^{-2}(1+\ln(1+t))e^{-\nu t}G_0,\label{px-py-u2L-es-2}\\
& \|D_x^{s_1}u^{(1)}_L\|_{L^\infty}+\|D_x^{s_1}(u^{(1)}_L-t u^{(2)}_L)\|_{L^\infty}\lesssim (1+t)^{-1}e^{-\nu t}G_0,\label{px-py-u1L-es-inf}\\
&\|D_x^{s_2+1}u^{(2)}_{L}\|_{L^\infty}+(1+\ln(1+t))^{-1}\|u^{(2)}_{L}\|_{L^\infty}\lesssim (1+t)^{-2}e^{-\nu t}G_0,\label{px-py-u2L-es-inf}\\
&\|D_x^{s_2}(\pa_y+t\pa_x)u^{(1)}_L\|_{L^\infty}+\|D_x^{s_2}(\pa_y+t\pa_x)(u^{(1)}_L-t u^{(2)}_L)\|_{L^\infty}  \lesssim e^{-\nu t}G_0,\label{px-py-u1L-es-inf-2}\\
&\|D_x^{s_2}(\pa_y+t\pa_x)u^{(2)}_{L}\|_{L^\infty}\lesssim (1+t)^{-1}e^{-\nu t}G_0.\label{px-py-u2L-es-inf-2}
\end{align}
\end{itemize}
\end{proposition}

\begin{proof}

We divide the proof into two steps.

\textbf{Step 1: Weak enhanced dissipation estimates.} 
Due to
\begin{align}\label{iden}
(\pa_y+ikt)^{s_2}\omega_{L,k}=e^{-ikyt}\pa_y^{s_2}(e^{ikyt}\omega_{L,k}),
\end{align}
we use Plancherel's theorem, the condition $s_1+s_2\leq 2$, and \eqref{oL-es-2-k} with $s=s_2$ to obtain
\begin{align*}
&\|D_x^{s_1}(\pa_y+t\pa_x)^{s_2}\omega_L\|_{L^2}^2
\lesssim\int_{\mathbb{R}} |k|^{2s_1}\|\pa_y^{s_2}(e^{ikyt}\omega_{L,k})\|_{L^2_y}^2dk\nonumber\\
\lesssim&\int_{\mathbb{R}} |k|^{2s_1}e^{-2\nu t-2c\nu k^2 t^3}(A_{s_2}+\nu^{-\frac{1}{3}}|k|^{\frac{1}{3}}B_{s_2})^2dk\nonumber\\
\lesssim&e^{-2\nu t}\Big(\int_{|k|\leq 1} |k|^{2s_1}e^{-c\nu k^2 t^3}dk\|(A_{s_2},\nu^{-\frac{1}{3}}B_{s_2})\|_{L^\infty_k}^2+e^{-c\nu t^3}\|(A_{s_1+s_2},\nu^{-\frac{1}{3}}B_{s_1+s_2+1})\|_{L^2_k}^2\Big)\nonumber\\
\lesssim& e^{-2\nu t}(1+\nu t^3)^{-\frac{1}{2}-s_1}G_0^2,
\end{align*}
which concludes \eqref{px-py-oL-es-2}. 

Combining \eqref{iden} and \eqref{oL-es-inf-k} with $s=s_2$, we derive
\begin{align*}
&\|D_x^{s_1}(\pa_y+t\pa_x)^{s_2}\omega_L\|_{L^\infty}\lesssim\int_{\mathbb{R}} |k|^{s_1}\|\pa_y^{s_2}(e^{ikyt}\omega_{L,k})\|_{L^\infty_y}dk\nonumber\\
\lesssim&\int_{\mathbb{R}} |k|^{s_1}\min\{1,|k|^{-\frac{1}{2}}\}e^{-\nu t-c\nu k^2 t^3}(A_{s_2+1}+\nu^{-\frac{1}{3}}|k|^{\frac{1}{3}}B_{s_2+1})dk\nonumber\\
\lesssim&e^{-\nu t}\int_{\mathbb{R}} \big(|k|^{s_1}\textbf{1}_{|k|\leq 1}+|k|^{-\frac{7}{6}}\textbf{1}_{|k|\leq 1}\big)e^{-c\nu k^2 t^3}dk\|(A_{s_1+s_2+2},\nu^{-\frac{1}{3}}B_{s_1+s_2+2})\|_{L^\infty_k}\nonumber\\
\lesssim& e^{-\nu t}(1+\nu t^3)^{-\frac{1}{2}-\frac{s_1}{2}}G_0.
\end{align*}
This inequality implies \eqref{px-py-oL-es-inf}. 

By the same argument as deriving in \eqref{px-py-oL-es-2} and \eqref{px-py-oL-es-inf}, and noting $s_1+s_2\leq 3$ together with \eqref{tL-es-2-k} and \eqref{tL-es-inf-k}, we obtain \eqref{px-py-tL-es-2} and \eqref{px-py-tL-es-inf}.

\textbf{Step 2: Inviscid damping estimates.} 

\textbf{\emph{Step 2.1: Estimates of \eqref{px-py-u1L-es-2}--\eqref{px-py-u2L-es-2}.}} Let $\phi_{L,k}=(\pa_y^2-k^2)^{-1} \omega_{L,k}$. Then $u_{L,k}=(\pa_y,-ik) \phi_{L,k}$. Noting 
$$\int_{-1}^{1}\frac{|k|+|k|^2}{(1+|k|t)^2}\,dk \lesssim\int_{-1}^{1}\frac{|k|}{(1+|k|t)^2}\,dk \lesssim(1+t)^{-2}(1+\ln(1+t))$$
and \eqref{ujL-es-2-k}, we obtain
\begin{align*}
\|\mathcal{M}u^{(1)}_L\|_{L^2}^2\lesssim&\int_{\mathbb{R}} (|k|+|k|^2)\Big(\frac{\textbf{\textup 1}_{|k|\leq 1}}{(1+|k|t)^{2}}+\frac{\textbf{\textup 1}_{|k|\geq 1}}{|k|^{4}(1+t)^{2}}\Big)e^{-2\nu t-2c\nu k^2 t^3}(A_{1}+\nu^{-\frac{1}{3}}|k|^{\frac{1}{3}}B_{1})^2dk \nonumber\\
\lesssim&e^{-2\nu t}(1+t)^{-2}\Big((1+\ln(1+t))+\int_{|k|\geq 1} |k|^{-\frac{5}{3}}dk\Big)\|(A_{1},\nu^{-\frac{1}{3}}B_{1})\|_{L^\infty_k}^2\nonumber\\
\lesssim& e^{-2\nu t}(1+t)^{-2}(1+\ln(1+t))G_0^2.
\end{align*}
 Using \eqref{ujL-es-2-k} and $|k|(1+|k|t)^{-1}\lesssim(1+t)^{-1}$ for $|k|\leq 1$, we also have
\begin{equation*}
\begin{aligned}
\|\mathcal{M}u^{(2)}_L\|_{L^2}^2\lesssim& e^{-2\nu t}(1+t)^{-4}(1+\ln(1+t))G_0^2.
\end{aligned}
\end{equation*}

\textbf{\emph{Step 2.2: Estimates of \eqref{px-py-u1L-es-inf}--\eqref{px-py-u2L-es-inf}.}}
Let $\tilde{\omega}_L(t,x,y)=\omega_L(t,x+ty,y)$ and $\tilde{\theta}_L(t,x,y)=\theta_L(t,x+ty,y)$. It follows from \eqref{otL-eq} that
\begin{align*}
\begin{cases}
(\pa_t +\nu(1-t^2\pa_x^2))\tilde{\omega}_L=-\pa_x\tilde{\theta}_L, \\
(\pa_t +\nu(1-t^2\pa_x^2))\tilde{\theta}_L=0, \\
\tilde{\omega}_L(t,x\mp t,\pm 1)=\tilde{\theta}_L(t,x\mp t,\pm 1)=0, \\
\tilde{\omega}_L|_{t=0}=\omega^{in}, \quad \tilde{\theta}_L|_{t=0}=\theta^{in}.
\end{cases}
\end{align*}
A straightforward calculation yields
\begin{align*}
\begin{cases}
\tilde{\omega}_L=e^{-\nu t+\frac{1}{3}\nu t^3\pa_x^2}(\omega^{in}-t\pa_x\theta^{in}), \\
\tilde{\theta}_L=e^{-\nu t+\frac{1}{3}\nu t^3\pa_x^2}\theta^{in}.
\end{cases}
\end{align*}
By the heat-kernel estimate, we obtain
\begin{align}\label{omega-tilde-L}
\|\tilde{\omega}_L\|_{L^1_xL^\infty_y}\lesssim e^{-\nu t}(\|\omega^{in}\|_{L^1_xL^\infty_y}+\nu^{-\frac{1}{3}}\||\pa_x|^{\frac{1}{3}}\theta^{in}\|_{L^1_xL^\infty_y}).
\end{align}
We decompose $\phi_L=\phi_1+(\phi_L-\phi_1)$, where $\phi_1$ satisfies $\pa^2_y  \phi_1=\omega_{L}, 
\phi_1(\pm1)=0$.
By means of \eqref{omega-tilde-L}, we derive
\begin{align}\label{p-y-ph-1}
\|\pa_y\phi_1\|_{L^\infty_y}\leq& \|\omega_{L}\|_{L^1_y}= \int^{1}_{-1}|(\tilde{\omega}_{L}(t,x-ty,y))|dy\nonumber\\
\lesssim& t^{-1}\int_{\mathbb{R}}|\tilde{\omega}_{L}(t,m,y)|dm
 \lesssim t^{-1}e^{-\nu t}\big(\|\omega^{in}\|_{L^1_x L^\infty_y}+\nu^{-\frac{1}{3}}\||\pa_x|^{\frac{1}{3}}\theta^{in}\|_{L^1_xL^\infty_y}\big).
\end{align}
 Using \eqref{ujL-es-2-k} and $\pa^2_y (\phi_1-\phi_{L})=\omega_L-\pa_y^2\phi_L=\pa^2_x\phi_{L}$, we obtain
\begin{align*}
\|\pa_y(\phi_1-\phi_{L})\|_{L^\infty_y}\lesssim &\int^1_{-1}\Big|\int_{\mathbb{R}}e^{ikx} k^2\phi_{L,k} dk\Big|dy\lesssim \int_{\mathbb{R}}|k|\|e^{ikyt}u^{(2)}_{L,k}\|_{L^2_y}dk\\
\lesssim &e^{-\nu t}\int_{\mathbb{R}} \Big(\frac{|k|^2\textbf{1}_{|k|\leq 1}}{(1+|k|t)^2}+\frac{\textbf{1}_{|k|\geq 1}}{|k|^2(1+t)^2}\Big)\big(A_{2}+\nu^{-\frac{1}{3}}|k|^{\frac{1}{3}}B_{2}\big)dk\\
\lesssim &(1+t)^{-2}e^{-\nu t}\|(A_2,\nu^{-\frac{1}{3}}B_2)\|_{L^2_k},
\end{align*}
which together with \eqref{p-y-ph-1} implies
\begin{align}\label{p-y-inf-0}
\|\pa_y \phi_{L}\|_{L^\infty}\lesssim e^{-\nu t}t^{-1}G_0.
\end{align}
On the other hand, from \eqref{ujL-es-inf-k}, we have
\begin{equation*}
\begin{aligned}
&\|\pa_y\phi_{L}\|_{L^\infty}\lesssim \int_{\mathbb{R}}\|e^{ikyt}u^{(1)}_{L,k}\|_{L^\infty_y}dk\\
\lesssim &\int_{|k|\leq 1} \frac{e^{-\nu t}}{1+|k|t}\big(A_{2}+\nu^{-\frac{1}{3}}B_{2}\big)dk+\int_{|k|\geq 1}\frac{e^{-\nu t}}{|k|^2(1+t)}\big(A_{2}+\nu^{-\frac{1}{3}}|k|^{\frac{1}{3}}B_{2}\big)dk\\
\lesssim &e^{-\nu t}\|(A_2,\nu^{-\frac{1}{3}}B_2)\|_{L^2_k},
\end{aligned}
\end{equation*}
which together with \eqref{p-y-inf-0} yields
\begin{align*}
\|u^{(1)}_L\|_{L^\infty}=\|\pa_y \phi_{L}\|_{L^\infty}\lesssim e^{-\nu t}(1+t)^{-1}G_0.
\end{align*}
Similarly, we obtain
\begin{align*}
&\|u^{(1)}_L-tu^{(2)}_L\|_{L^\infty}\lesssim e^{-\nu t}(1+t)^{-1}G_0,\\
&\|u^{(2)}_L\|_{L^\infty}\lesssim e^{-\nu t}(1+t)^{-2}(1+\ln(1+t))G_0.
\end{align*}
Thanks to  \eqref{ujL-es-inf-k}, we directly compute
\begin{align}\label{p-x-s-1-u-1-L-e-s}
&\|D_x^{s_1}u^{(1)}_L\|_{L^\infty}
\lesssim \int_{\mathbb{R}} |k|^{s_1}\|e^{ikyt}u^{(1)}_{L,k}\|_{L^\infty_y}dk\nonumber\\
\lesssim&\int_{\mathbb{R}} |k|^{s_1}\big(\frac{\textbf{\textup 1}_{|k|\leq 1}}{(1+|k|t)}+\frac{\textbf{\textup 1}_{|k|\geq 1}}{|k|^{2}(1+t)}\big) e^{-\nu t-c\nu k^2 t^3}(A_{2}+\nu^{-\frac{1}{3}}|k|^{\frac{1}{3}}B_{2})dk\nonumber\\
\lesssim&\frac{e^{-\nu t}}{1+t}\Big(\int_{|k|\leq 1} |k|^{s_1-1}+\int_{|k|\geq 1} |k|^{-\frac{5}{3}}dk\Big)\|(A_{s_1+2},\nu^{-\frac{1}{3}}B_{s_1+2})\|_{L^\infty_k}\lesssim \frac{e^{-\nu t}}{1+t}G_0.
\end{align}
In view of \eqref{ujL-es-inf-k}, we also have
\begin{align}\label{p-x-s-1-u-2-L-e-s}
&\|D_x^{s_1}u^{(2)}_L\|_{L^\infty}\lesssim e^{-\nu t}(1+t)^{-2}G_0.
\end{align}

\textbf{\emph{Step 2.3: Estimates of \eqref{px-py-u1L-es-inf-2}--\eqref{px-py-u2L-es-inf-2}.}}
 Using $\div u_L=0$, we derive
\begin{equation*}\begin{aligned}
&(\pa_y+t\pa_x)u^{(1)}_L=\omega_L+\pa_xu^{(2)}_L+t\pa_xu^{(1)}_L,\nonumber\\&
(\pa_y+t\pa_x)(u^{(1)}_L-tu^{(2)}_L)=\omega_L+\pa_xu^{(2)}_L+2t\pa_xu^{(1)}_L-t^2\pa_xu^{(2)}_L. 
\end{aligned}\end{equation*}Then employing \eqref{px-py-oL-es-inf}, \eqref{p-x-s-1-u-1-L-e-s}--\eqref{p-x-s-1-u-2-L-e-s} yields
\begin{align*}
&\|D_x^{s_2}(\pa_y+t\pa_x)u^{(1)}_L\|_{L^\infty}
+\|D_x^{s_2}(\pa_y+t\pa_x)(u^{(1)}_L-tu^{(2)}_L)\|_{L^\infty} \nonumber\\
\lesssim&\|D_x^{s_2}\omega_L\|_{L^\infty}+t\|D_x^{s_2+1} u^{(1)}_L\|_{L^\infty}+(1+t^2)\|D_x^{s_2+1} u^{(2)}_L\|_{L^\infty}\lesssim e^{-\nu t}G_0.
\end{align*}
Similarly, we get
\begin{align*}
\|D_x^{s_2}(\pa_y+t\pa_x)u^{(2)}_{L}\|_{L^\infty}\lesssim&\|D_x^{s_2+1} u^{(1)}_L\|_{L^\infty}+t\|D_x^{s_2+1} u^{(2)}_L\|_{L^\infty}\lesssim (1+t)^{-1}e^{-\nu t}G_0.
\end{align*}
 This completes the proof.
\end{proof}

With the above preparations for $\omega_L$ and $\theta_L$, we have the estimates on the remainder terms.
\begin{proposition}\label{prop-2.4-1}
 Under the same assumption as Proposition \ref{prop-2.4}, for $s_1\in[0,2]$, $s_2\in[0,1]$, we have
\begin{align}
&\|D_x^{s_2}E_\mathfrak{L}\|_{L^2}\lesssim \Big((1+t)^{-1}(1+\nu t^3)^{-\frac{3}{4}}+(1+\ln(1+t))(1+t)^{-2}(1+\nu t^3)^{-\frac{1}{4}}\Big)e^{-\nu t}G_0^2\nonumber\\
&\qquad \qquad \qquad+\Big((1+\nu t^3)^{-\frac{1}{4}-\frac{s_2}{2}}+t(1+\nu t^3)^{-\frac{3}{4}-\frac{s_2}{2}}\Big)\nu e^{-\nu t}G_0,\label{Er-es-2}
\end{align}
and
\begin{align}
&\|D_x^{s_1}T_\mathfrak{L}\|_{L^2}\lesssim  \Big((1+t)^{-1}(1+\nu t^3)^{-\frac{3}{4}}+(1+\ln(1+t))(1+t)^{-2}(1+\nu t^3)^{-\frac{1}{4}}\Big)e^{-\nu t}G_0F_0\nonumber\\
&\qquad \qquad \qquad+\Big((1+\nu t^3)^{-\frac{1}{4}-\frac{s_1}{2}}+t(1+\nu t^3)^{-\frac{3}{4}-\frac{s_1}{2}}\Big)\nu e^{-\nu t}F_0,\label{Tr-es-2}
\end{align}
and
\begin{align}\label{py-uL-tL-es-2}
&\|D_x^{s_2}(\pa_y+t\pa_x)(u_L\cdot \na \theta_L)\|_{L^2}\lesssim \Big(\frac{1}{(1+\nu t^3)^{\frac{3}{4}}}+\frac{1}{(1+t)(1+\nu t^3)^{\frac{1}{4}}}\Big)e^{-2\nu t}G_0F_0,
\end{align}
and
\begin{align}\label{TrL-es-2}
&\|D_x^{s_1}T_{rL}\|_{L^2}\lesssim\big((1+\nu t^3)^{-\frac{1}{4}-\frac{s_1}{2}}+t(1+\nu t^3)^{-\frac{3}{4}-\frac{s_1}{2}}\big)\nu e^{-\nu t}F_0,
\end{align}
where 
\begin{align*}
&E_\mathfrak{L}=E_{r_L}+u_L\cdot\na\omega_L,\quad E_{r_L}=\pa_t\omega_L-\nu\Delta\omega_L+y\pa_x\omega_L+\pa_x\theta_L,\\
&T_\mathfrak{L}=T_{r_L}+u_L\cdot\na\theta_L, \quad \,\, T_{r_L}=\pa_t\theta_L-\nu\Delta\theta_L+y\pa_x\theta_L. 
\end{align*}
\end{proposition}
\begin{proof}
Due to $u_L\cdot \na \omega_L=(u_L^{(1)}-tu_L^{(2)})\pa_x \omega_L+u_L^{(2)}(\pa_y+t\pa_x) \omega_L$, we
use \eqref{px-py-oL-es-2}, \eqref{px-py-u1L-es-inf} and
\eqref{px-py-u2L-es-inf} to obtain
\begin{align}\label{uL-oL-es-2}
&\|D_x^{s_2}(u_L\cdot \na \omega_L)\|_{L^2}\nonumber\\
\lesssim&\|D_x^{s_2}(u_L^{(1)}-tu_L^{(2)})\|_{L^\infty}\|\pa_x \omega_L\|_{L^2}+\|u_L^{(1)}-tu_L^{(2)}\|_{L^\infty}\|D_x^{s_2+1} \omega_L\|_{L^2}\nonumber\\
&+\|D_x^{s_2}u_L^{(2)}\|_{L^\infty}\|(\pa_y+t\pa_x) \omega_L\|_{L^2}+\|u_L^{(2)}\|_{L^\infty}\|D_x^{s_2}(\pa_y+t\pa_x) \omega_L\|_{L^2}\nonumber\\
\lesssim& \big((1+t)^{-1}(1+\nu t^3)^{-\frac{3}{4}}+(1+\ln(1+t))(1+t)^{-2}(1+\nu t^3)^{-\frac{1}{4}}\big)e^{-2\nu t}G_0^2.
\end{align}
Thanks to $$\big(-(1+t^2k^2)-(\pa_y^2-k^2)\big)\omega_{L,k}=-e^{-ikyt}\big(\pa_y^2-2ikt\pa_y-k^2+1\big)(e^{ikyt}\omega_{L,k}),$$
then we have
\begin{align*}
\|\hat{E}_{r_L}\|_{L^2_y}\lesssim&\nu\|(\pa_y^2-2ikt\pa_y-k^2+1)(e^{ikyt}\omega_{L,k})\|_{L^2_y}\nonumber\\
\lesssim& \nu e^{-\nu t} e^{-c\nu k^2 t^3}\big((A_{2}+\nu^{-\frac{1}{3}}|k|^{\frac{1}{3}}B_{2})+ t |k|(A_{1}+\nu^{-\frac{1}{3}}|k|^{\frac{1}{3}}B_{1})\big).
\end{align*}
Then from \eqref{oL-es-2-k} and  $E_{r_L}=-\nu(1-t^2\pa_x^2)\omega_{L}-\nu \Delta\omega_{L}$, we derive
\begin{align}\label{ErL-es-2}
\|D_x^{s_2}E_{r_L}\|_{L^2}\lesssim&\nu e^{-\nu t}\Big(\!\int_{|k|\leq 1} |k|^{2s_2}(1+k^2t^2)e^{-c\nu k^2 t^3}dk+\!\int_{|k|\geq 1} |k|^{-\frac{4}{3}}(1+t^2)e^{-c\nu k^2 t^3}dk\Big)^{\frac{1}{2}}G_0\nonumber\\
\lesssim&\Big(\textbf{1}_{\nu t^3\leq 1} (1+t) +\textbf{1}_{\nu t^3\geq 1} \big((\nu t^3)^{-\frac{1}{4}-\frac{s_2}{2}}+(\nu t^3)^{-\frac{3}{4}-\frac{s_2}{2}}t\big)\Big)\nu e^{-\nu t}G_0\nonumber\\
\lesssim& \Big((1+\nu t^3)^{-\frac{1}{4}-\frac{s_2}{2}}+t(1+\nu t^3)^{-\frac{3}{4}-\frac{s_2}{2}}\Big)\nu e^{-\nu t}G_0.
\end{align}
This estimate together with \eqref{uL-oL-es-2} yields \eqref{Er-es-2}.

Collecting \eqref{px-py-oL-es-2}, \eqref{px-py-tL-es-2},
 \eqref{px-py-u1L-es-inf}--\eqref{px-py-u2L-es-inf-2} and by the same way as leading to \eqref{uL-oL-es-2}, we obtain
\begin{align}\label{uL-tL-es-2-1}
\|D_x^{s_1}(u_L\cdot \na \theta_L)\|_{L^2}
\lesssim& \Big(\frac{1}{(1+t)(1+\nu t^3)^{\frac{3}{4}}}+\frac{1+\ln(1+t)}{(1+t)^{2}(1+\nu t^3)^{\frac{1}{4}}}\Big)e^{-2\nu t}G_0F_0,
\end{align}
and
\begin{align}\label{py-uL-tL-es-2-2}
&\|D_x^{s_2}(\pa_y+t\pa_x)(u_L\cdot \na \theta_L)\|_{L^2}\lesssim \Big(\frac{1}{(1+\nu t^3)^{\frac{3}{4}}}+\frac{1}{(1+t)(1+\nu t^3)^{\frac{1}{4}}}\Big)e^{-2\nu t}G_0F_0.
\end{align}
From \eqref{tL-es-2-k} and $e^{ikyt}(\pa_y+t\pa_x)\theta_{L,k}=\pa_y(e^{ikyt}\theta_{L,k})$, by the same way as leading to
 \eqref{ErL-es-2}, we have
\begin{align}\label{TrL-es-2-3}
&\|D_x^{s_1}T_{rL}\|_{L^2}\lesssim\big((1+\nu t^3)^{-\frac{1}{4}-\frac{s_1}{2}}+t(1+\nu t^3)^{-\frac{3}{4}-\frac{s_1}{2}}\big)\nu e^{-\nu t}F_0.
\end{align}
Combining \eqref{uL-tL-es-2-1} with \eqref{TrL-es-2-3}  yields \eqref{Tr-es-2}.  This completes the proof of Proposition \ref{prop-2.4-1}.
\end{proof}

\section{\texorpdfstring{Preliminaries for the linearized equation of the system \eqref{ote-eq}}%
  {Preliminaries for the linearized equation of the system (Boue)}}\label{pre-lem}

In this section, we present several lemmas that will facilitate the nonlinear estimates.

\begin{lemma}\label{lem-3.1}
Let $u_e = (\pa_y, -\pa_x)\Delta^{-1} \omega_e$. For $j=0,1$, the following estimates hold:
\begin{align}
&\|\pa_x^{j}u_e^{(1)}\|_{L^\infty_{x,y}}\lesssim(\nu^{\frac{1}{4}}\|\pa_x^{1+j}\omega_e\|_{L^2_{x,y}}^{\frac{1}{2}}
+|\ln\nu|\|\pa_x^j\omega_e\|_{L^2_{x,y}}^{\frac{1}{2}})
\|\pa_x^{1+j}u_e\|_{L^2_{x,y}}^{\frac{1}{2}}
+\|\pa_x^{j}\omega_e\|_{L^2_{x,y}}, \label{u1-e-inf-es}\\
&\|\pa_x^{j}u_e^{(2)}\|_{L^\infty_{x,y}}\lesssim\nu^{\frac{1}{4}}\|\pa_x^{2+j}u_e\|_{L^2_{x,y}}^{\frac{1}{2}}
\|\pa_x^{1+j}u_e\|_{L^2_{x,y}}^{\frac{1}{2}}+|\ln\nu|\|\pa_x^{1+j}u_e\|_{L^2_{x,y}}, \label{u2-e-inf-es}
\end{align}
and
\begin{align}
&\||\pa_x|^{\frac{1}{3}}u_e^{(1)}\|_{L^\infty_{x,y}}
\lesssim\nu^{\frac{1}{12}}\|\pa_xu_e\|_{L^2_{x,y}}^{\frac{1}{2}}
\|\pa_x\omega_e\|_{L^2_{x,y}}^{\frac{1}{2}}+\nu^{-\frac{1}{6}}\|\pa_xu_e\|_{L^2_{x,y}}^{\frac{1}{2}}\|\omega_e\|_{L^2_{x,y}}^{\frac{1}{2}}
+\|\mathcal{M}u_e\|_{L^2_{x,y}},
\label{px-13-ue-1-es}\\
&\||\pa_x|^{\frac{1}{3}}u_e^{(2)}\|_{L^\infty_{x,y}}\lesssim\nu^{\frac{1}{12}}\|\pa_xu_e\|_{L^2_{x,y}}^{\frac{1}{2}}
\|\pa_x\omega_e\|_{L^2_{x,y}}^{\frac{1}{2}}+\nu^{-\frac{1}{6}}\|\pa_xu_e\|_{L^2_{x,y}}
,\label{px-13-ue-2-es}\\
&\||\pa_x|^{\frac{1}{3}}u_e^{(1)}\|_{L^2_x L^\infty_y}\lesssim\|\mathcal{M}u_e\|_{L^2_{x,y}}^{\frac{1}{2}}\|\omega_e\|_{L^2_{x,y}}^{\frac{1}{2}}
,\label{px-13-ue-3-es}\\
&\||\pa_x|^{\frac{1}{3}}u_e^{(2)}\|_{L^2_x L^\infty_y}\lesssim\|\mathcal{M}u_e\|_{L^2_{x,y}}
.\label{px-13-ue-4-es}
\end{align}
\end{lemma}
\begin{proof}
Let $I_1=\{k:|k|\leq 1\}$, $I_2=\{k:1\leq |k|\leq \nu^{-\frac{1}{2}}\}$, $I_3=\{k:|k|\geq \nu^{-\frac{1}{2}}\}$. For \eqref{u1-e-inf-es}, 
\begin{align}\label{u1-e-inf-es-2-0}
\|\pa_x^{j}u_e^{(1)}\|_{L^\infty_{x,y}}\lesssim&\int_{I_1}|k|^{j}\|u_{e,k}^{(1)}\|_{L^\infty_y}dk+\int_{I_2\cup I_3}|k|^{j}\|u_{e,k}^{(1)}\|_{L^\infty_y}dk.
\end{align}
Using \eqref{lem-A.2-es} yields
\begin{align}\label{u1-e-inf-es-2-1}
\int_{I_1}|k|^{j}\|u_{e,k}^{(1)}\|_{L^\infty_y}dk\lesssim\int_{I_1}\|k^{j}\omega_{e,k}\|_{L^2_y}dk
\lesssim\|\pa_x^{j}\omega_e\|_{L^2_{x,y}}.
\end{align}
By the G-N inequality
 \begin{align}\label{GN ineq}
 \|u_{e,k}^{(1)}\|_{L^\infty_y}\lesssim\|\pa_yu_{e,k}^{(1)}\|_{L^2_y}^{\frac{1}{2}}\|u_{e,k}^{(1)}\|_{L^2_y}^{\frac{1}{2}}+\|u_{e,k}^{(1)}\|_{L^2_y},
 \end{align}
 and  \eqref{lem-A.2-es}, we derive
\begin{align}\label{u1-e-inf-es-2-2}
&\int_{I_2\cup I_3}|k|^{j}\|u_{e,k}^{(1)}\|_{L^\infty_y}dk\nonumber\\
\lesssim&\int_{I_2}|k|^{-\frac{1}{2}}\|k^{j}\pa_yu_{e,k}^{(1)}\|_{L^2_y}^{\frac{1}{2}}\|k^{1+j}u_{e,k}^{(1)}\|_{L^2_y}^{\frac{1}{2}}
+|k|^{-1}\|k^{1+j}u_{e,k}^{(1)}\|_{L^2_y}dk\nonumber\\
&+\int_{I_3}|k|^{-1}\|k^{1+j}\pa_yu_{e,k}^{(1)}\|_{L^2_y}^{\frac{1}{2}}\|k^{1+j}u_{e,k}^{(1)}\|_{L^2_y}^{\frac{1}{2}}
+|k|^{-1}\|k^{1+j}u_{e,k}^{(1)}\|_{L^2_y}dk\nonumber\\
\lesssim&\nu^{\frac{1}{4}}\|\pa_x^{1+j}\omega_e\|_{L^2_{x,y}}^{\frac{1}{2}}\|\pa_x^{1+j}u_e\|_{L^2_{x,y}}^{\frac{1}{2}}
+\|\pa_x^{j}\omega_e\|_{L^2_{x,y}}+|\ln\nu|\|\pa_x^j\omega_e\|_{L^2_{x,y}}^{\frac{1}{2}}\|\pa_x^{1+j}u_e\|_{L^2_{x,y}}^{\frac{1}{2}}.
\end{align}
Combining \eqref{u1-e-inf-es-2-0}, \eqref{u1-e-inf-es-2-1} and \eqref{u1-e-inf-es-2-2}, we have
\begin{align}\label{u1-e-inf-es-2}
\|\pa_x^{j}u_e^{(1)}\|_{L^\infty_{x,y}}\lesssim&(\nu^{\frac{1}{4}}\|\pa_x^{1+j}\omega_e\|_{L^2_{x,y}}^{\frac{1}{2}}+|\ln\nu|\|\pa_x^j\omega_e\|_{L^2_{x,y}}^{\frac{1}{2}})\|\pa_x^{1+j}u_e\|_{L^2_{x,y}}^{\frac{1}{2}}
+\|\pa_x^{j}\omega_e\|_{L^2_{x,y}}.
\end{align}
For \eqref{u2-e-inf-es}, by means of 
\begin{align}\label{interpo}
\|u_{e,k}^{(2)}\|_{L^\infty_y}\lesssim \|\pa_yu_{e,k}^{(2)}\|_{L^2_y}^{\frac{1}{2}}\|u_{e,k}^{(2)}\|_{L^2_y}^{\frac{1}{2}},
\end{align}
the Poincar\'e inequality, $\pa_yu_{e,k}^{(2)}=-iku_{e,k}^{(1)}$ and \eqref{lem-A.2-es} gives
\begin{align}\label{u2-e-inf-es-2}
&\|\pa_x^{j}u_e^{(2)}\|_{L^\infty_{x,y}}\lesssim\int_{\mathbb{R}}|k|^{j}\|u_{e,k}^{(2)}\|_{L^\infty_y}dk\nonumber\\
\lesssim
&\int_{I_1}\|k^{j}\pa_yu_{e,k}^{(2)}\|_{L^2_y}^{\frac{1}{2}}\|k^{j}u_{e,k}^{(2)}\|_{L^2_y}^{\frac{1}{2}}
dk+\int_{I_2}|k|^{-\frac{1}{2}}\|k^{j}\pa_yu_{e,k}^{(2)}\|_{L^2_y}^{\frac{1}{2}}\|k^{1+j}u_{e,k}^{(2)}\|_{L^2_y}^{\frac{1}{2}}
dk\nonumber\\
&+\int_{I_3}|k|^{-1}\|k^{1+j}\pa_yu_{e,k}^{(2)}\|_{L^2_y}^{\frac{1}{2}}\|k^{1+j}u_{e,k}^{(2)}\|_{L^2_y}^{\frac{1}{2}}
dk\nonumber\\
\lesssim&\nu^{\frac{1}{4}}\|\pa_x^{2+j}u_e\|_{L^2_{x,y}}^{\frac{1}{2}}\|\pa_x^{1+j}u_e\|_{L^2_{x,y}}^{\frac{1}{2}}
+|\ln\nu|\|\pa_x^{1+j}u_e\|_{L^2_{x,y}}.
\end{align}
Using \eqref{GN ineq}, we obtain
\begin{align}\label{px-13-ue-1-es-2}
&\||\pa_x|^{\frac{1}{3}}u_e^{(1)}\|_{L^\infty_{x,y}}
\lesssim\int_{\mathbb{R}}|k|^{\frac{1}{3}}\|u_{e,k}^{(1)}\|_{L^\infty_y}dk\nonumber\\
\lesssim&\int_{I_1\cup I_2}|k|^{-\frac{1}{6}}\|ku_{e,k}^{(1)}\|_{L^2_y}^{\frac{1}{2}}
\|\pa_yu_{e,k}^{(1)}\|_{L^2_y}^{\frac{1}{3}}\|\pa_yu_{e,k}^{(1)}\|_{L^2_y}^{\frac{1}{6}}
+|k|^{-\frac{1}{6}}\|k^{\frac{1}{2}}u_{e,k}^{(1)}\|_{L^2_y}dk\nonumber\\
&+\int_{I_3}|k|^{-\frac{2}{3}}\|ku_{e,k}^{(1)}\|_{L^2_y}^{\frac{1}{2}}
\|k\pa_yu_{e,k}^{(1)}\|_{L^2_y}^{\frac{1}{2}}
+|k|^{-\frac{2}{3}}\|ku_{e,k}^{(1)}\|_{L^2_y}dk\nonumber\\
\lesssim&\nu^{\frac{1}{12}}\|\pa_xu_e\|_{L^2_{x,y}}^{\frac{1}{2}}
\|\pa_x\omega_e\|_{L^2_{x,y}}^{\frac{1}{2}}+\nu^{-\frac{1}{6}}\|\pa_xu_e\|_{L^2_{x,y}}^{\frac{1}{2}}\|\omega_e\|_{L^2_{x,y}}^{\frac{1}{2}}+\|\mathcal{M}u_e\|_{L^2_{x,y}}.
\end{align}
It follows from \eqref{interpo} and  $\pa_yu_{e,k}^{(2)}=-ik u_{e,k}^{(1)}$ that
\begin{align}\label{px-13-ue-2-es-2}
&\||\pa_x|^{\frac{1}{3}}u_e^{(2)}\|_{L^\infty_{x,y}}\lesssim\int_{\mathbb{R}}|k|^{\frac{1}{3}}\|u_{e,k}^{(2)}\|_{L^\infty_y}dk\nonumber\\
\lesssim&\int_{I_1\cup I_2}|k|^{-\frac{1}{6}}\|ku_{e,k}^{(2)}\|_{L^2_y}^{\frac{1}{2}}
\|\pa_yu_{e,k}^{(2)}\|_{L^2_y}^{\frac{1}{3}}\|\pa_yu_{e,k}^{(2)}\|_{L^2_y}^{\frac{1}{6}}dk+\int_{I_3}|k|^{-\frac{2}{3}}\|ku_{e,k}^{(2)}\|_{L^2_y}^{\frac{1}{2}}
\|k\pa_yu_{e,k}^{(2)}\|_{L^2_y}^{\frac{1}{2}}dk\nonumber\\
\lesssim&\nu^{\frac{1}{12}}\|\pa_xu_e\|_{L^2_{x,y}}^{\frac{1}{2}}\|\pa_x\omega_e\|_{L^2_{x,y}}^{\frac{1}{2}}+\nu^{-\frac{1}{6}}\|\pa_xu_e\|_{L^2_{x,y}}
.
\end{align}
From \eqref{GN ineq}, \eqref{interpo} and \eqref{lem-A.2-es}, we obtain
\begin{align}\label{px-13-ue-1-es-2-inf}
&\||\pa_x|^{\frac{1}{3}}u_e^{(1)}\|_{L^2_x L^\infty_y}^2
\lesssim\int_{\mathbb{R}}|k|^{\frac{2}{3}}\|u_{e,k}^{(1)}\|_{L^\infty_y}^2dk\nonumber\\
\lesssim&\int_{I_1}|k|^{\frac{1}{6}}\|k^{\frac{1}{2}}u_{e,k}^{(1)}\|_{L^2_y}
\|\pa_yu_{e,k}^{(1)}\|_{L^2_y}
+|k|^{\frac{1}{6}}\|k^{\frac{1}{2}}u_{e,k}^{(1)}\|_{L^2_y}\|u_{e,k}^{(1)}\|_{L^2_y}dk\nonumber\\
&+\int_{I_2\cup I_3}|k|^{-\frac{1}{3}}\|ku_{e,k}^{(1)}\|_{L^2_y}
\|\pa_yu_{e,k}^{(1)}\|_{L^2_y}
+|k|^{-\frac{1}{3}}\|ku_{e,k}^{(1)}\|_{L^2_y}\|u_{e,k}^{(1)}\|_{L^2_y}dk\nonumber\\
\lesssim&\|\mathcal{M}u_e\|_{L^2_{x,y}}\|\omega_e\|_{L^2_{x,y}}.
\end{align}
and
\begin{align}\label{px-13-ue-1-es-2-inf-2}
\||\pa_x|^{\frac{1}{3}}u_e^{(2)}\|_{L^2_x L^\infty_y}^2\lesssim&\int_{I_1}|k|^{\frac{1}{6}}\|k^{\frac{1}{2}}u_{e,k}^{(2)}\|_{L^2_y}
\|\pa_yu_{e,k}^{(2)}\|_{L^2_y}dk+\int_{I_2\cup I_3}|k|^{-\frac{1}{3}}\|ku_{e,k}^{(2)}\|_{L^2_y}
\|\pa_yu_{e,k}^{(2)}\|_{L^2_y}dk\nonumber\\
\lesssim&\|\mathcal{M}u_e\|_{L^2_{x,y}}^2.
\end{align}
Combining \eqref{u1-e-inf-es-2}, \eqref{u2-e-inf-es-2}--\eqref{px-13-ue-1-es-2-inf-2}, one  completes the proof of Lemma~\ref{lem-3.1}.
\end{proof}
\begin{lemma}\label{lem-3.1-1}
Let   $\theta_e\in H^2_0$. The estimate holds:
\begin{align}
\|\theta_e\|_{L^\infty_{x,y}}\lesssim&(\nu^{-\frac{1}{12}}|\ln\nu|\| \theta_e\|_{L^2_{x,y}}^{\frac{1}{2}}+\nu^{\frac{1}{12}}|\ln\nu|^{-1}\|\pa_x^2 \theta_e\|_{L^2_{x,y}}^{\frac{1}{2}})\|\pa_y\theta_e\|_{L^2_{x,y}}^{\frac{1}{2}}.\label{th-e-inf-es}
\end{align}
\end{lemma}
\begin{proof}
Let $I_4=\{k:|k|\leq \nu^{-\frac{1}{6}}|\ln\nu|^{2}\}$ and $I_5=\{k: |k|\geq \nu^{-\frac{1}{6}}|\ln\nu|^{2}\}$. Utilizing
\begin{align*}
\|\theta_{e,k}\|_{L^\infty_y}\lesssim \|\theta_{e,k}\|_{L^2_y}^{\frac{1}{2}}\|\pa_y\theta_{e,k}\|_{L^2_y}^{\frac{1}{2}},
\end{align*}
we get
\begin{align*}
\|\theta_e\|_{L^\infty_{x,y}}\lesssim&\int_{I_4}\|\theta_{e,k}\|_{L^2_y}^{\frac{1}{2}} \|\pa_y\theta_{e,k}\|_{L^2_y}^{\frac{1}{2}}dk+\int_{I_5}
|k|^{-1}\|k^2\theta_{e,k}\|_{L^2_y}^{\frac{1}{2}}\|\pa_y\theta_{e,k}\|_{L^2_y}^{\frac{1}{2}}dk\nonumber\\
\lesssim&\nu^{-\frac{1}{12}}|\ln\nu|\| \theta_e\|_{L^2_{x,y}}^{\frac{1}{2}}\|\pa_y\theta_e\|_{L^2_{x,y}}^{\frac{1}{2}}+\nu^{\frac{1}{12}}|\ln\nu|^{-1}\|\pa_x^2 \theta_e\|_{L^2_{x,y}}^{\frac{1}{2}}\|\pa_y\theta_e\|_{L^2_{x,y}}^{\frac{1}{2}}.
\end{align*}
This completes the proof of Lemma~\ref{lem-3.1-1}.
\end{proof}

\begin{Remark}\label{rem-3.1}
 We emphasize that for the domain $\mathbb{T}\times[-1,1]$, the frequencies can be separated into two regimes $|k|\leq \nu^{-\frac{1}{6}}$, and $|k|\geq \nu^{-\frac{1}{6}}$. There holds:
\begin{align*}
\|\theta_e\|_{L^\infty_{x,y}}\lesssim&(\nu^{-\frac{1}{12}}\| \theta_e\|_{L^2_{x,y}}^{\frac{1}{2}}+\nu^{\frac{1}{12}}\|\pa_x^2 \theta_e\|_{L^2_{x,y}}^{\frac{1}{2}})\|\pa_y\theta_e\|_{L^2_{x,y}}^{\frac{1}{2}}.
\end{align*}
\end{Remark}
Next, we provide the space-time estimates of the solution for the problem
\begin{align}\label{f-st-es}
\begin{cases}
\pa_t f-\nu\Delta f+y\pa_xf=g,\\
 f(t,x,\pm1)=0, \quad f|_{t=t_0}=f(t_0).
\end{cases}
\end{align}
\begin{lemma}\label{lem-3.2}
Let $f$ be the solution of \eqref{f-st-es}.
For any $t\in [t_0,t_1]$ with $t_0\geq\nu^{-\frac{1}{6}}$, we have
\begin{align}
\|f(t)\|_{X}\lesssim &\|f(t_0)\|_{L^2}+\| g\|_{L^1L^2}, \label{fX-es}\\
\nu^{\frac13}\||\pa_x|^{\frac{2}{3}} f\|_{L^1 L^2}\lesssim &|\ln\nu|\|f(t_0)\|_{L^2}+\|g\|_{L^1 L^2}. \label{pxfX-es}
\end{align}
\end{lemma}
\begin{proof}
The estimate \eqref{fX-es} is proved in \cite[Proposition 2.4]{CLM-26}. We use the same argument that leads to
 \cite[(2.22)]{CLM-26}, replacing  $\mathcal{M}$ and \cite[(2.28)]{CLM-26} by $|\pa_x|^{\frac{2}{3}}$ and
\begin{align*}
\int_{\nu^{-\frac{1}{6}}}^{\infty}\||\pa_x|^{\frac{2}{3}}S(t,\nu^{-\frac{1}{6}})f(t_0)\|_{L^2}dt\lesssim&\int_{\nu^{-\frac{1}{6}}}^{\infty}\nu^{-\frac{1}{3}}(1+t)^{-1}(1+\nu t)^{-\frac{1}{6}}dt\|f(t_0)\|_{L^2}\nonumber\\
\lesssim&\nu^{-\frac{1}{3}}|\ln\nu|\|f(t_0)\|_{L^2},
\end{align*}
where $S$ is the solution operator of \eqref{f-st-es}. Then we deduce \eqref{pxfX-es}.
\end{proof}

Next we consider the solution of
\begin{align}\label{f-ft-es}
\begin{cases}
\pa_t f-\nu(\pa_y^2-k^2) f+ikyf=g,\\
 f(t,k,\pm1)=0, \quad f|_{t=t_0}=f(t_0).
\end{cases}
\end{align}
In what follows, we denote
$\|f\|_{L^p_tL^q_y}:=\|f\|_{L^p_t(t_0,t_1;L^q_{y})}$, $\|f\|_{\mathcal{Y}}:=\|f\|_{\mathcal{Y}(t_0,t_1)}.$
\begin{lemma}\label{lem-3.3}
Let $f$ be the solution of \eqref{f-ft-es}, then for $t\in [t_0,t_1]$, the estimate holds:
\begin{align}\label{fX-es-k}
\|f\|_{\mathcal{Y}}+\nu^{\frac{1}{6}}|k|^{\frac{1}{3}}\|f(t)\|_{L^2_tL^2_y}+\nu^{\frac{1}{3}}|k|^{\frac{2}{3}}\|f(t)\|_{L^1_tL^2_y}
\lesssim &\|f(t_0)\|_{L^2_y}+\| g\|_{L^1_tL^2_y}.
\end{align}
\end{lemma}

\begin{proof}
For $s\geq 0$, let $\tilde{S}(t,s)$  be the solution operator of
\begin{align*}
\begin{cases}
\pa_t \tilde{f}-\nu(\pa_y^2-k^2) \tilde{f}+iky \tilde{f}=0,\\
\tilde{f}(t,k,\pm1)=0,\quad  \tilde{f}|_{t=s}=\tilde{h}(k,y).
\end{cases}
\end{align*}
For $t\geq s$, we have $\tilde{f}(t,k,y)=\tilde{S}(t,s) \tilde{h}(k,y)$. From \cite[Lemma 2.2--2.3]{CLM-26}, we obtain
 \begin{align}
\|\tilde{S}(t,0)\tilde{h}(k,y)\|_{\mathcal{Y}} \lesssim&\|\tilde{h}\|_{L^2_y},\label{Sk-hk-Yk}\\
\int_{0}^{\infty}\|\tilde{S}(t,0)\tilde{h}(k,y)\|_{L^2_y}dt\lesssim&\int_{0}^{\infty}(1+\nu k^2t^3)^{-\frac{1}{2}}dt\|\tilde{h}\|_{L^2_y}\lesssim\nu^{-\frac{1}{3}}|k|^{-\frac{2}{3}}\|\tilde{h}\|_{L^2_y}.\label{Sk-hk-12}
\end{align}
The solution of \eqref{f-ft-es} can be expressed as
\begin{align}\label{solution of fk}
f(t)=\tilde{S}(t,t_0)f(t_0)+\int^t_{t_0}\tilde{S}(t,s)g(s)ds.
\end{align}
Thanks to \eqref{Sk-hk-Yk} and \eqref{solution of fk}, we deduce
\begin{align}\label{fk-Yk-es}
\|f(t)\|_{\mathcal{Y}}\lesssim\|f(t_0)\|_{L^2_y}+\| g\|_{L^1_tL^2_y}.
\end{align}
Similarly, using \eqref{Sk-hk-12} and \eqref{solution of fk}, we obtain
\begin{align}\label{fk-12-es}
\nu^{\frac{1}{3}}|k|^{\frac{2}{3}}\|f(t)\|_{L^1_tL^2_y}\lesssim \|f(t_0)\|_{L^2_y}+\| g\|_{L^1_tL^2_y}.
\end{align}
Collecting \eqref{fk-Yk-es} and \eqref{fk-12-es} with the interpolation inequality yields
\begin{align*}
\nu^{\frac{1}{6}}|k|^{\frac{1}{3}}\|f(t)\|_{L^2_tL^2_y}\lesssim \nu^{\frac{1}{6}}|k|^{\frac{1}{3}}\|f(t)\|_{L^\infty_tL^2_y}^{\frac{1}{2}}\|f(t)\|_{L^1_tL^2_y}^{\frac{1}{2}}
\lesssim \|f(t_0)\|_{L^2_y}+\| g\|_{L^1_tL^2_y},
\end{align*}
and so we complete the proof of Lemma~\ref{lem-3.3}.
\end{proof}
\begin{lemma}\cite[Propositions 6.4 and 6.5]{CLWZ-20}\label{lem-3.4}
Let $f$ be the solution of the problem
\begin{align*}
\begin{cases}
\pa_t f-\nu(\pa_y^2-k^2) f+ikyf=-ikg_{1}-\pa_yg_{2},\\
 f(t,k,\pm1)=0, \quad f|_{t=t_0}=f(t_0).
\end{cases}
\end{align*}
Then for $t\in [t_0,t_1]$, there holds
\begin{align}\label{fk-h1-h2-es}
\|f\|_{\mathcal{Y}}+\nu^{\frac{1}{6}}|k|^{\frac{1}{3}}\|f\|_{L^2_t L^2_y}\lesssim &\|f(t_0)\|_{L^2_y}+\min\{\nu^{-\frac{1}{6}}|k|^{\frac{2}{3}},\nu^{-\frac{1}{2}}\}\|g_{1}\|_{L^2_tL^2_y}+\nu^{-\frac{1}{2}} \|g_{2}\|_{L^2_tL^2_y}.
\end{align}
\end{lemma}
\begin{lemma} \cite[Lemma 6.3]{CLWZ-20}\label{lem-3.5}
Let $f\in H^2_y\cap H^1_0(-1,1)$ be the solution of \eqref{f-ft-es} with $g=0$. 
 Then for any $|k|\geq 1$, there exist constants $C,c>0$ such that
\begin{align}\label{semigroup}
\|e^{-(t-t_0)L_k}f_{0}\|_{L^2_y}\leq C e^{-c\nu^{\frac{1}{3}}|k|^{\frac{2}{3}}(t-t_0)-\nu (t-t_0)}\|f_{0}\|_{L^2_y}.
\end{align}
\end{lemma}

Classifying frequencies in the horizontal direction, we establish the space-time estimates for solutions of \eqref{f-ft-es}: Lemma~\ref{lem-3.6} for $|k|\leq 1$ and Lemma~\ref{lem-3.7} for  $1 \leq |k| \leq \nu^{-\frac{1}{2}}$.

\begin{lemma}\label{lem-3.6}
Let $f$ be the solution of \eqref{f-ft-es} on $[t_0,t_1]$ with the initial condition $f(t_0)=0$, where $t_0 > \nu^{-\frac{1}{6}}$. 
Assume $k,l\in\mathbb{R}$, $k\neq l$, and $|k|\leq 1$, there holds
\begin{align*}
&\|f\|_{\mathcal{Y}}^2+\nu^{\frac{1}{3}}|k|^{\frac{2}{3}}\|f\|_{L^2_t L^2_y}^2\lesssim\langle\nu^{\frac{1}{3}}t_1\rangle^3(1+|l|^2) |k-l|^{-1}\langle k-l\rangle^{-1}\|(\pa_y,\langle k-l\rangle)(e^{ilyt}g)\|_{L^2_tL^2_y}^2.
\end{align*}
\end{lemma}
\begin{proof}
We first decompose $f =f_1+f_2$, where $(f_1, f_2)$ solves
\begin{align*}
\begin{cases}
\pa_t f_1+ikyf_{1}=g, \quad t\in (t_0,t_1],\\
\pa_t f_2-\nu(\pa_y^2-k^2) f_2+ikyf_{2}=\nu (\pa_y^2-k^2) f_1, \\
f_1(t,k,\pm 1)=f_2(t,k,\pm 1)=0, \quad f_1|_{t=t_0}=f_2|_{t=t_0}=0.
\end{cases}
\end{align*}
In view of \eqref{fk-h1-h2-es}, together with the condition $|k|\leq 1$, we obtain
\begin{align}\label{fk-2-leq-es}
\|f_2\|_{\mathcal{Y}}+\nu^{\frac{1}{6}}|k|^{\frac{1}{3}}\|f_2\|_{L^2_t L^2_y}\lesssim& \nu^{\frac{5}{6}}|k|^{\frac{5}{3}}\|f_1\|_{L^2_tL^2_y}+\nu^{\frac{1}{2}}\|\pa_yf_{1}\|_{L^2_tL^2_y}\nonumber\\
\lesssim&\|f_1\|_{\mathcal{Y}}+\nu^{\frac{1}{6}}|k|^{\frac{1}{3}}\|f_1\|_{L^2_t L^2_y}.
\end{align}
 It follows from \cite[Lemma 2.1]{CLM-26} that
\begin{align}\label{fk-1-leq-es-inf}
&\|f_1\|_{L^\infty_t L^2_{y}}^2+(|k|+|k|^2)\|(\pa_y,k)(\pa_y^2-k^2)^{-1}f_1\|_{L^2_tL^2_{y}}^2\nonumber\\
\lesssim& |k-l|^{-1}\langle k-l\rangle^{-1}\|(\pa_y,\langle k-l\rangle)(e^{ilyt}g)\|_{L^2_tL^2_y}^2,
\end{align}
and
\begin{align*}
\nu\|(\pa_y,k)f_1\|_{L^2_t L^2_y}^2\lesssim& \langle\nu^{\frac{1}{3}}t_1\rangle^3(|k-l|^2+|l|^2) |k-l|^{-1}\langle k-l\rangle^{-1}\|(\pa_y,\langle k-l\rangle)(e^{ilyt}g)\|_{L^2_tL^2_y}^2\nonumber\\
\lesssim& \langle\nu^{\frac{1}{3}}t_1\rangle^3(1+|l|^2) |k-l|^{-1}\langle k-l\rangle^{-1}\|(\pa_y,\langle k-l\rangle)(e^{ilyt}g)\|_{L^2_tL^2_y}^2.
\end{align*}
Using \eqref{fk-1-leq-es-inf}, we deduce
\begin{align}\label{fk-1-leq-es-23}
\nu^{\frac{1}{3}}|k|^{\frac{2}{3}}\|f_1\|_{L^2_t L^2_y}^2\lesssim&\nu^{\frac{1}{3}}|k|^{\frac{2}{3}}t_1\|f_1\|_{L^\infty_t L^2_y}^2\nonumber\\
\lesssim&\langle\nu^{\frac{1}{3}}t_1\rangle^3|k-l|^{-1}\langle k-l\rangle^{-1}\|(\pa_y,\langle k-l\rangle)(e^{ilyt}g)\|_{L^2_tL^2_y}^2.
\end{align}
Combining \eqref{fk-2-leq-es}--\eqref{fk-1-leq-es-23}, one completes the proof of Lemma~\ref{lem-3.6}.
\end{proof}
 Compared with the inviscid equations \cite[Lemma 2.1]{CLM-26}, the estimates for dissipative equation require one lower derivative estimate on the external force term. Specifically,
\begin{lemma}\label{lem-3.7}
Let $f$ be the solution of \eqref{f-ft-es} on $[t_0,t_1]$ with the initial condition $f(t_0) = 0$, where $t_0 > \nu^{-\frac{1}{6}}$. Suppose $k,l \in \mathbb{R}$, $k \neq l$ and $1 \leq |k| \leq \nu^{-\frac{1}{2}}$. Then, we have
\begin{align*}
&\|f\|_{\mathcal{Y}}^2+\nu^{\frac{1}{3}}|k|^{\frac{2}{3}}\|f\|_{L^2_t L^2_y}^2\nonumber\\
\lesssim&\big((\langle\nu^{\frac{1}{3}}t_1\rangle^3|l|^2+1) |k-l|^{-1}\langle k-l\rangle^{-1}+\nu^{\frac{1}{3}}|k|^{-\frac{4}{3}}\big)\|(\pa_y,\langle k-l\rangle)(e^{ilyt}g)\|_{L^2_tL^2_y}^2.
\end{align*}
\end{lemma}
\begin{proof}
For $m=1,2,...,\tilde{m}$ with $\tilde{m}=[\nu^{\frac{1}{3}}|k|^{\frac{2}{3}}(t_1-t_0)]+1$, define $\tau_{m}=t_0+\nu^{-\frac{1}{3}}|k|^{-\frac{2}{3}}m$, and
\begin{align*}
\mathcal{I}_m=[\tau_{m-1},\tau_{m}]\backslash(t_1,\tau_{\tilde{m}}].
\end{align*}
Recall
\begin{align*}
 \|f\|^2_{\mathcal{Z}_m}=\|f\|_{\mathcal{Y}(\mathcal{I}_m; L^2_y)}^2+\nu^{\frac{1}{3}}|k|^{\frac{2}{3}}\|f\|_{L^2_t(\mathcal{I}_m;L^2_y)}^2.
 \end{align*}
Decompose $f = \sum_{m=1}^{\tilde{m}} f_m$, where
 \begin{align*}
\begin{cases}
\pa_t f_{m}-\nu(\pa_y^2-k^2)f_{m}+ikyf_{m}=g \tilde{\chi}_{m}, \quad t\in(\tau_{m-1},t_1],\\
f_m(t,k,\pm 1)=0, \quad f_m |_{t=\tau_{m-1}}=0,
\end{cases}
\end{align*}
with
\begin{align*}
\tilde{\chi}_{m}=
\begin{cases}
1, \quad t\in(\tau_{m-1},\tau_m],\\
0, \quad t\in(\tau_m,t_1].
\end{cases}
\end{align*}
Thanks to \eqref{fX-es-k} and \eqref{semigroup}, we obtain for all $n>m$
\begin{align*}
\|f_m\|_{\mathcal{Z}_n}\leq\|f_m(\tau_{n-1})\|_{L^2_y}\lesssim e^{-c\nu^{\frac{1}{3}}|k|^{\frac{2}{3}}(\tau_{n-1}-\tau_m)}\|f_m(\tau_m)\|_{L^2_y}\lesssim e^{-c(n-1-m)}\|f_m\|_{\mathcal{Z}_m}.
\end{align*}
In light of the Cauchy--Schwarz inequality, we derive
\begin{align*}
\Big(\sum_{1\leq m\leq n}\|f_m\|_{\mathcal{Z}_n}\Big)^2\lesssim \sum_{1\leq m\leq n-1}e^{c(n-1-m)}\|f_m\|_{\mathcal{Z}_n}^2+\|f_n\|_{\mathcal{Z}_n}^2\lesssim \sum_{1\leq m\leq n} e^{-c(n-m)}\|f_m\|_{\mathcal{Z}_m}^2.
\end{align*}
The Minkowski inequality gives
\begin{align}\label{f-k-m-es}
&\|f\|_{\mathcal{Y}}^2+\nu^{\frac{1}{3}}|k|^{\frac{2}{3}}\|f\|_{L^2_t L^2_y}^2\leq \sum_{1\leq n\leq \tilde{m}}\|f\|^2_{\mathcal{Z}_n}\leq \sum_{1\leq n\leq\tilde{m}}\Big\|\sum_{1\leq m\leq n}f_m\Big\|_{\mathcal{Z}_n}^2\nonumber\\
\lesssim& \sum_{1\leq n\leq\tilde{m}} \sum_{1\leq m\leq n} e^{-c(n-m)}\|f_m\|_{\mathcal{Z}_m}^2\lesssim\sum_{1\leq m\leq\tilde{m}} \|f_m\|_{\mathcal{Z}_m}^2.
\end{align}
We further decompose $f_m = f_{m,1} + f_{m,2}$, where $f_{m,1}$ and $f_{m,2}$ solve
\begin{align}\label{fm-12-eq}
\begin{cases}
\pa_t f_{m,1}+ikyf_{m,1}=g\tilde{\chi}_m, \quad t\in(\tau_{m-1},t_1], \\
\pa_t f_{m,2}-\nu(\pa_y^2-k^2) f_{m,2}+ikyf_{m,2}=\nu (\pa_y^2-k^2) f_{m,1}, \\
f_{m,1}(t,k,\pm 1)=f_{m,2}(t,k,\pm 1)=0, \\
f_{m,1}|_{t=\tau_{m-1}}=f_{m,2}|_{t=\tau_{m-1}}=0.
\end{cases}
\end{align}
For $f_{m,2}$, using \eqref{fk-h1-h2-es} with $\nu |k|\min\{\nu^{-\frac16}|k|^{\frac23}, \nu^{-\frac12}\}\leq \nu^{\frac{5}{6}}|k|^{\frac{5}{3}}$ we obtain
\begin{align}\label{fm-2-es}
\|f_{m,2}\|_{\mathcal{Z}_m}\lesssim \nu^{\frac{5}{6}}|k|^{\frac{5}{3}}\|f_{m,1}\|_{L^2_tL^2_y}+\nu^{\frac{1}{2}}\|\pa_yf_{m,1}\|_{L^2_tL^2_y}\lesssim&\|f_{m,1}\|_{\mathcal{Z}_m}.
\end{align}
For $f_{m,1}$, by \cite[Lemma 2.1]{CLM-26}, we have
\begin{align}\label{fm-1-es}
&\|f_{m,1}\|_{L^\infty(\mathcal{I}_m;L^2_y)}+(|k|^{\frac{1}{2}}+|k|)\|(\pa_y,k)(\pa_y^2-k^2)^{-1}f_{m,1}\|_{L^2(\mathcal{I}_m;L^2_y)}\nonumber\\
\lesssim & |k-l|^{-\frac{1}{2}}\langle k-l\rangle^{-\frac{1}{2}}\|(\pa_y,\langle k-l\rangle)(e^{ilyt}g\chi_m)\|_{L^2_tL^2_y}.
\end{align}
The first equation $\eqref{fm-12-eq}_1$ in $\eqref{fm-12-eq}$ has the equivalent form
\begin{align*}
\pa_t(e^{ilyt}f_{m,1})+i(k-l)y(e^{ilyt}f_{m,1})=e^{ilyt}g\chi_m, \quad e^{ily\tau_{m-1}}f_{m,1}(\tau_{m-1})=0.
\end{align*}
Differentiating with respect to $y$ yields
\begin{align*}
&(\pa_t+i(k-l)y)\pa_y(e^{ilyt}f_{m,1})+i(k-l)(e^{ilyt}f_{m,1})=\pa_y(e^{ilyt}g\chi_m), \,\, \nonumber\\
&\pa_y(e^{ily\tau_{m-1}}f_{m,1}(\tau_{m-1}))=0,
\end{align*}
which implies
\begin{align}\label{fm-1-k-l-es}
e^{i(k-l)yt}\pa_y(e^{ilyt}f_{m,1})+\int_{\tau_{m-1}}^t i(k-l)e^{ikyt'} f_{m,1}
-e^{i(k-l)yt'}\pa_y(e^{ilyt'} g(t') \chi_m)dt'=0.
\end{align}
Let $\tilde{t}= \nu^{-\frac{1}{3}}|k|^{-\frac{2}{3}}$. For $t\in \mathcal{I}_m$, taking the $L^2_y$ norm of \eqref{fm-1-k-l-es}, we obtain
\begin{align*}
\|\pa_y(e^{ilyt}f_{m,1})\|_{L^2_y}
\leq& |k-l|\tilde{t}\|f_{m,1}\|_{L^\infty_t L^2_y}+\tilde{t}^{\frac{1}{2}}\|\pa_y(e^{ilyt} g \chi_m)\|_{L^2_t L^2_y}.
\end{align*}
Due to $\pa_y f_{m,1}(t)= e^{-ilyt}\pa_y(e^{ilyt}f_{m,1}) - ilt f_{m,1}$, we have
\begin{align}\label{fm-1-k-l-es-2}
&\|\pa_y f_{m,1}\|_{L^2_t(\mathcal{I}_m;L^2_y)}\leq\|\pa_y(e^{ilyt}f_{m,1})\|_{L^2_t(\mathcal{I}_m;L^2_y)}+\||l| tf_{m,1}\|_{L^2_t(\mathcal{I}_m;L^2_y)}\nonumber\\
\lesssim&\||k-l|\tilde{t}+|l|t\|_{L^2_t(\mathcal{I}_m)}\|f_{m,1}\|_{L^\infty_t(\mathcal{I}_m;L^2_y)}+\tilde{t}\|\pa_y(e^{ilyt} g \chi_{m})\|_{L^2_t(\mathcal{I}_m;L^2_y)}.
\end{align}
Using the condition $\nu k^2\leq1$ and \eqref{fm-1-k-l-es-2} , we deduce
\begin{align}\label{nufm1L2L2}
&\nu\|(\pa_y,k) f_{m,1}\|_{L^2_t(\mathcal{I}_m;L^2_y)}^2+\nu^{\frac{1}{3}}|k|^{\frac{2}{3}}\|f_{m,1}\|_{L^2_t(\mathcal{I}_m;L^2_y)}^2\nonumber\\
\lesssim&\nu\|\pa_y f_{m,1}\|_{L^2_t(\mathcal{I}_m;L^2_y)}^2+\nu^{\frac{1}{3}}|k|^{\frac{2}{3}}\| f_{m,1}\|_{L^2_t(\mathcal{I}_m;L^2_y)}^2\nonumber\\
\lesssim&(\nu|k-l|^2\tilde{t}^3+\nu|l|^2\tilde{t}_m^2\tilde{t}
+\nu^{\frac{1}{3}}|k|^{\frac{2}{3}}\tilde{t})\|f_{m,1}\|_{L^\infty_t(\mathcal{I}_m;L^2_y)}^2 +\nu\tilde{t}^2\|\pa_y(e^{ilyt} g
\chi_m)\|_{L^2_t(\mathcal{I}_m;L^2_y)}^2.
\end{align}
Noting $\nu^{-\frac13}|k|^{-\frac23}=\tilde{t}\leq \tau_m$, we get
\begin{align*}
\nu|k-l|^2\tilde{t}^3+\nu|l|^2 \tau_m^2\tilde{t}+\nu^{\frac{1}{3}}|k|^{\frac{2}{3}}\tilde{t}\leq \nu |k|^2\tilde{t}^3+\nu |l|^2 \tau_m^3+\nu^{\frac{1}{3}}|k|^{\frac{2}{3}}\tilde{t}\lesssim \langle\nu^{\frac{1}{3}} \tau_m\rangle^3|l|^2+1,
\end{align*}
which combined with \eqref{fm-1-es} and \eqref{nufm1L2L2} implies
\begin{align*}
&\nu\|(\pa_y,k) f_{m,1}\|_{L^2_t(\mathcal{I}_m;L^2_y)}^2+\nu^{\frac{1}{3}}|k|^{\frac{2}{3}}\|f_{m,1}\|_{L^2_t(\mathcal{I}_m;L^2_y)}^2\nonumber\\
\lesssim&\Big((\langle\nu^{\frac{1}{3}}\tau_m\rangle^3|l|^2+1) |k-l|^{-1}\langle k-l\rangle^{-1}+\nu^{\frac{1}{3}}|k|^{-\frac{4}{3}}\Big)\|(\pa_y,\langle k-l\rangle)(e^{ilyt}g\chi_m)\|_{L^2_tL^2_y}^2.
\end{align*}
This inequality along with \eqref{fm-1-es} gives
\begin{align*}
\|f_{m,1}\|_{\mathcal{Z}_m}^2\leq&\Big((\langle\nu^{\frac{1}{3}}\tau_m\rangle^3|l|^2+1) |k-l|^{-1}\langle k-l\rangle^{-1}+\nu^{\frac{1}{3}}|k|^{-\frac{4}{3}}\Big)\|(\pa_y,\langle k-l\rangle)(e^{ilyt}g\chi_{m})\|_{L^2_tL^2_y}^2,
\end{align*}
which, together with \eqref{f-k-m-es} and \eqref{fm-2-es}, yields
\begin{align*}
&\|f\|_{\mathcal{Y}}^2+\nu^{\frac{1}{3}}|k|^{\frac{2}{3}}\|f\|_{L^2_t L^2_y}^2\lesssim\sum_{1\leq m\leq\tilde{m}} \big(\|f_{m,1}\|_{\mathcal{Z}_m}^2+\|f_{m,2}\|_{\mathcal{Z}_m}^2\big)\nonumber\\
\lesssim&\sum_{1\leq m\leq \tilde{m}}\Big((\langle\nu^{\frac{1}{3}}\tau_m\rangle^3|l|^2+1) |k-l|^{-1}\langle k-l\rangle^{-1}+\nu^{\frac{1}{3}}|k|^{-\frac{4}{3}}\Big)\|(\pa_y,\langle k-l\rangle)(e^{ilyt}g\chi_m)\|_{L^2_tL^2_y}^2\nonumber\\
\lesssim&\Big((\langle\nu^{\frac{1}{3}}t_1\rangle^3|l|^2+1) |k-l|^{-1}\langle k-l\rangle^{-1}+\nu^{\frac{1}{3}}|k|^{-\frac{4}{3}}\Big)\|(\pa_y,\langle k-l\rangle)(e^{ilyt}g)\|_{L^2_tL^2_y}^2.
\end{align*}
Thus we complete the proof of Lemma~\ref{lem-3.7}.
\end{proof}

\section{\texorpdfstring%
  {Estimates for the system \eqref{ote-eq} in the case $t\leq\nu^{-\frac{1}{6}}$}%
  {Estimates for the system (Boue) in the case t <= nu\textasciicircum(-1/6)}}\label{non-es-leq}

In this section, we focus on the space-time estimates for $(\omega_e, \theta_e)$ in the short-time regime, which relies on the energy method.  For convenience, we denote
$\|f\|_{L^pL^q}:=\|f\|_{L^p_t(0,\nu^{-\frac{1}{6}};L^q_{x,y})}$, $\|f\|_{L^p_t}:=\|f\|_{L^p_t(0,\nu^{-\frac{1}{6}})}$, $\|f\|_{W}:=\|f\|_{W(0,\nu^{-\frac{1}{6}})}$
and
\begin{align}\label{E1-df}
\mathcal{E}_1=&\|\omega_e\|_{W}+\nu^{-\frac{1}{3}}|\ln\nu|^4\|\theta_e\|_{W}+\nu^{-
\frac{1}{6}}\|\pa_x\theta_e\|_{W}+\|\pa_y\theta_e\|_{W}\nonumber\\
&+\nu^{\frac{1}{6}}\|\pa_x\omega_e\|_{W}+\|\pa_x^2\theta_e\|_{W}+\nu^{\frac{1}{6}}\|\pa_x\pa_y\theta_e\|_{W}.
\end{align}

We rewrite the error equation $(\omega_e,\theta_e)$:
\begin{align}\label{oe-te-eq-leq}
\begin{cases}
\pa_t \omega_e-\nu\Delta\omega_e+y\pa_x\omega_e+u_e\cdot\na \omega_L+u_L\cdot\na \omega_e+u_e\cdot\na \omega_e+E_\mathfrak{L}+\pa_x\theta_e=0, \\
\pa_t \theta_e-\nu\Delta\theta_e+y\pa_x\theta_e+u_e\cdot\na \theta_L+u_L\cdot \na \theta_e+u_e\cdot \na \theta_e+T_\mathfrak{L}=0, \\
\omega_e(t,x,\pm 1)=\theta_e(t,x,\pm 1)=0, \\
\omega_e|_{t=0}=\theta_e|_{t=0}=0.
\end{cases}
\end{align}

\begin{proposition}\label{prop-4.1}
Let $(\omega_{e},\theta_{e})$ be the solution of \eqref{oe-te-eq-leq} and $u_{e}=(\pa_y, -\pa_x)\Delta^{-1} \omega_{e}$. For $t\leq\nu^{-\frac{1}{6}}$, the following weak enhanced dissipation estimates hold:
\begin{align}\label{oe-geq-es}
\|\omega_{e}\|_{L^2_{x,y}}\lesssim \nu^{\frac{1}{3}}|\ln\nu|(E_0+\nu^{-\frac{1}{3}}|\ln\nu|^4F_0), \qquad \|\theta_{e}\|_{L^2_{x,y}} \lesssim\nu^{\frac{1}{3}}|\ln\nu|F_0,
\end{align}
as well as the inviscid damping estimate
\begin{align}\label{ue-geq-es}
&\|\mathcal{M}u_e\|_{L^2 L^2}\lesssim \nu^{\frac{1}{4}}|\ln\nu|(E_0+\nu^{-\frac{1}{3}}|\ln\nu|^4F_0).
\end{align}
\end{proposition}
\begin{proof}
To prove \eqref{oe-geq-es} and \eqref{ue-geq-es}, we claim
\begin{align}\label{claim}
\mathcal{E}_1\lesssim \nu^{\frac{1}{3}}|\ln\nu|(E_0+\nu^{-\frac{1}{3}}|\ln\nu|^4F_0).
\end{align}
\textbf{Step 1: Estimates of $\|\omega_e\|_{W}$ and $\|\pa_x\omega_e\|_{W}$.}
By $\eqref{oe-te-eq-leq}_1$, we have
\begin{equation}\label{oe-px-oe-eq}
\left\{ 
\begin{aligned}
&(\pa_t-\nu\Delta+y\pa_x) \omega_e+u_e\cdot \na \omega_L+u_L\cdot \na \omega_e+u_e\cdot \na \omega_e+E_\mathfrak{L}+\pa_x\theta_e=0,\\
&(\pa_t-\nu\Delta+y\pa_x) \pa_x\omega_e+\pa_x(u_e\cdot \na \omega_L+u_L\cdot \na \omega_e+u_e\cdot \na \omega_e)+\pa_xE_\mathfrak{L}+\pa_x^2\theta_e=0.
\end{aligned}
\right. 
\end{equation}
Taking the inner product of $\eqref{oe-px-oe-eq}$ with $(\omega_e, \pa_x\omega_e)$, and integrating in time, we obtain
\begin{align}\label{oe-px-oe-es}
\|\omega_e\|_{W}\lesssim&\|u_e\cdot\na \omega_L\|_{L^1 L^2}+\|E_\mathfrak{L}\|_{L^1 L^2}+\|\pa_x\theta_e\|_{L^1 L^2}\nonumber\\
\lesssim&\|u_e\cdot\na \omega_L\|_{L^1 L^2}+\|E_\mathfrak{L}\|_{L^1 L^2}+\nu^{-\frac{1}{6}}\|\pa_x\theta_e\|_{L^\infty L^2}\nonumber\\
:=&I_{1,eL}+I_{1,r}+\nu^{-\frac{1}{6}}\|\pa_x\theta_e\|_{W},
\end{align}
and
\begin{align}\label{px-oe-px-oe-es}
\|\pa_x\omega_e\|_{W}\lesssim&\|\pa_x (u_e\cdot\na \omega_L)\|_{L^1 L^2}
+\|\pa_x u_L\cdot\na \omega_e\|_{L^1 L^2}+\|\pa_x u_e\cdot\na \omega_e\|_{L^1 L^2}\nonumber\\
&+\|\pa_x E_\mathfrak{L}\|_{L^1 L^2}+\nu^{-\frac{1}{6}}\|\pa_x^2\theta_e\|_{L^\infty L^2}\nonumber\\
:=&I_{2,eL}+I_{2,Le}+I_{2,ee}+I_{2,r}+\nu^{-\frac{1}{6}}\|\pa_x^2\theta_e\|_{W}.
\end{align}
\textbf{\emph{Step 1.1: Estimates of $I_{1,eL}$ and $I_{2,eL}$.}}  Using \eqref{es-u-p-x-u-2}, \eqref{px-py-oL-es-inf} and \eqref{E1-df}, we obtain
\begin{align}\label{I1-eL-es}
I_{1,eL}
\lesssim&\|u_e^{(1)}-tu_e^{(2)}\|_{L^\infty L^2} \|\pa_x\omega_L\|_{L^1 L^\infty}+\|u_e^{(2)}\|_{L^\infty L^2} \|(\pa_y+t\pa_x)\omega_L\|_{L^1 L^\infty}\nonumber\\
\lesssim&\nu^{-
\frac{1}{3}}\|\omega_e\|_{W}G_0\lesssim\nu^{-
\frac{1}{3}}\mathcal{E}_1G_0,
\end{align}
and
\begin{align}\label{I2-eL-es}
I_{2,eL}\lesssim&\|\pa_x u_e\cdot\na \omega_L\|_{L^1 L^2}+\| u_e\cdot\na \pa_x \omega_L\|_{L^1 L^2}\nonumber\\
\lesssim&\|\pa_xu_e^{(1)}-t\pa_xu_e^{(2)}\|_{L^\infty L^2} \|\pa_x\omega_L\|_{L^1 L^\infty}+\|\pa_xu_e^{(2)}\|_{L^\infty L^2} \|(\pa_y+t\pa_x)\omega_L\|_{L^1 L^\infty}\nonumber\\
&+\|u_e^{(1)}-tu_e^{(2)}\|_{L^\infty L^2} \|\pa_x^2\omega_L\|_{L^1 L^\infty}+\|u_e^{(2)}\|_{L^\infty L^2} \|\pa_x(\pa_y+t\pa_x)\omega_L\|_{L^1 L^\infty}\nonumber\\
\lesssim&\nu^{-
\frac{1}{3}}\|\omega_e\|_{W}G_0\lesssim\nu^{-\frac{1}{3}}\mathcal{E}_1G_0.
\end{align}
\textbf{\emph{Step 1.2: Estimate of $I_{2,Le}$.}} Noting that \eqref{px-py-u1L-es-inf}, \eqref{px-py-u2L-es-inf} and \eqref{E1-df}, we derive
\begin{align}\label{I2-Le-es}
I_{2,Le}\lesssim&\|\pa_xu_L^{(1)}\|_{L^1 L^\infty} \|\pa_x\omega_e\|_{L^\infty L^2}+\|\pa_x u_L^{(2)}\|_{L^2 L^\infty} \|\pa_y\omega_e\|_{L^2 L^2}\nonumber\\
\lesssim&\|(1+t)^{-1}\|_{L^1_t}G_0\|\pa_x\omega_e\|_{L^\infty L^2}+\|(1+t)^{-2}\|_{L^2_t}G_0\|\nabla\omega_e\|_{L^2 L^2}\nonumber\\
\lesssim&|\ln\nu|G_0\|\pa_x\omega_e\|_{W}+\nu^{-\frac{1}{2}}G_0\|\omega_e\|_{W}
\lesssim\nu^{-\frac{1}{2}}\mathcal{E}_1G_0.
\end{align}
\textbf{\emph{Step 1.3: Estimate of $I_{2,ee}$.}} Thanks to \eqref{u1-e-inf-es}, \eqref{u2-e-inf-es} and \eqref{E1-df}, we obtain
\begin{align}\label{p-x-u-e-1-2-2-inf}
\|\pa_xu_e\|_{L^2 L^\infty} \lesssim&\nu^{\frac{1}{4}}\|1\|_{L^2_t}^{\frac{1}{2}} \|\pa_x^{2}\omega_e\|_{L^2L^2}^{\frac{1}{2}}
\|\pa_x\omega_e\|_{L^\infty L^2}^{\frac{1}{2}}+|\ln\nu|\|1\|_{L^2_t}\|\pa_x\omega_e\|_{L^\infty L^2}\nonumber\\
\lesssim&\nu^{-\frac{1}{12}}|\ln\nu|\|\pa_x\omega_e\|_{W}\lesssim \nu^{-\frac{1}{4}}|\ln\nu|\mathcal{E}_1.
\end{align}
It follows from \eqref{p-x-u-e-1-2-2-inf} that
\begin{align}\label{I2-ee-es}
I_{2,ee}\leq &\|\pa_xu_e^{(1)}\pa_x\omega_e\|_{L^1 L^2} +\|\pa_x u_e^{(2)}\pa_y\omega_e\|_{L^1 L^2}\nonumber\\
\lesssim&\|1\|_{L^2_t}\|\pa_xu_e^{(1)}\|_{L^2 L^\infty} \|\pa_x\omega_e\|_{L^\infty L^2}+\|\pa_x u_e^{(2)}\|_{L^2 L^\infty} \|\pa_y\omega_e\|_{L^2 L^2}\nonumber\\
\lesssim& \nu^{-\frac{3}{4}}|\ln\nu|\mathcal{E}_1^2.
\end{align}
\textbf{\emph{Step 1.4: Estimates of $I_{1,r}$ and $I_{2,r}$.}}  From \eqref{Er-es-2} and  $G_0\leq c\nu^{\frac{1}{3}}$, we deduce
\begin{align}\label{Ir-es}
I_{1,r}+I_{2,r}=& \| E_\mathfrak{L}\|_{L^1 L^2}+\|\pa_x E_\mathfrak{L}\|_{L^1 L^2}\nonumber\\
\lesssim& \big\|\big((1+t)^{-1}+(1+t)^{-2}(1+\ln(1+t))\big)e^{-\nu t}G_0^2+(1+t)\nu e^{-\nu t}G_0\big\|_{L^1_t}\nonumber\\
\lesssim&c\nu^{\frac{1}{3}}|\ln\nu|G_0.
\end{align}
\textbf{\emph{Step 1.5: Estimates of $\|\omega_e\|_{W}$ and $\|\pa_x\omega_e\|_{W}$.}} Thanks to \eqref{oe-px-oe-es}, \eqref{I1-eL-es} and \eqref{Ir-es}, we get
\begin{align}\label{oe-es-leq}
\|\omega_e\|_{W}
\lesssim&\nu^{-
\frac{1}{3}}\mathcal{E}_1G_0+c\nu^{\frac{1}{3}}|\ln\nu|G_0 +\nu^{-\frac{1}{6}}\|\pa_x\theta_e\|_{W}.
\end{align}
Combining \eqref{px-oe-px-oe-es} with  \eqref{I2-eL-es}--\eqref{Ir-es}, we have
\begin{align}\label{px-oe-es-leq}
\|\pa_x\omega_e\|_{W}\lesssim& \nu^{-\frac{1}{2}}\mathcal{E}_1G_0+\nu^{-\frac{3}{4}}|\ln\nu|\mathcal{E}_1^2+c\nu^{\frac{1}{3}} |\ln\nu|G_0+\nu^{-\frac{1}{6}}\|\pa_x^2\theta_e\|_{W}.
\end{align}

\textbf{Step 2: Estimates of $\|\theta_e\|_{W}$, $\|\pa_x\theta_e\|_{W}$ and $\|\pa_x^2\theta_e\|_{W}$.}
By $\eqref{oe-te-eq-leq}_2$, we have
\begin{align}\label{te-px-te-eq}
\begin{cases}
\pa_t \theta_e-\nu\Delta\theta_e+y\pa_x\theta_e+u_e\cdot \na \theta_L+u_L\cdot \na \theta_e+u_e\cdot \na \theta_e+T_\mathfrak{L}=0,\\
\pa_t \pa_x\theta_e-\nu\Delta\pa_x\theta_e+y\pa_x^2\theta_e+\pa_x(u_e\cdot \na \theta_L+u_L\cdot \na \theta_e+u_e\cdot \na \theta_e)+\pa_xT_\mathfrak{L}=0,\\
\pa_t \pa_x^2\theta_e-\nu\Delta\pa_x^2\theta_e+y\pa_x^3\theta_e+\pa_x^2(u_e\cdot \na \theta_L+u_L\cdot \na \theta_e+u_e\cdot \na \theta_e)+\pa_x^2T_\mathfrak{L}=0.
\end{cases}
\end{align}
Taking the inner product of $\eqref{te-px-te-eq}$  with $(\theta_e, \pa_x\theta_e, \pa_x^2\theta_e)$ and integrating in time yield
\begin{align}
\|\theta_e\|_{W}\lesssim&I_{3,eL}+I_{3,r},\label{te-px-te-es}\\
\|\pa_x\theta_e\|_{W}^2
\lesssim&(I_{4,eL}+I_{4,Le}+I_{4,r})\|\pa_x\theta_e\|_{{W}}+I_{4,ee},\label{px-te-px-te-es}\\
\|\pa_x^2\theta_e\|_{W}^2
\lesssim&(I_{5,eL}+I_{5,Le}+I_{5,r})\|\pa_x^2\theta_e\|_{{W}}+I_{5,ee},\label{px2-te-px-te-es}
\end{align}
where
\begin{align*}
&I_{3,eL}:=\|u_e\cdot\na \theta_L\|_{L^1 L^2}, \quad\quad\quad I_{3,r}:=\|T_\mathfrak{L}\|_{L^1 L^2},
\end{align*}
and
 \begin{align*}
I_{4,eL}&:=\| \pa_x(u_e\cdot\na \theta_L)\|_{L^1 L^2}, & I_{4,Le}&:=\| \pa_xu_L\cdot\na \theta_e\|_{L^1 L^2},\\
I_{4,r}&:=\| \pa_xT_\mathfrak{L}\|_{L^1 L^2}, & I_{4,ee}&:=\|\langle\pa_xu_e\cdot\na \theta_e,\pa_x \theta_e\rangle\|_{L^1_t},
\end{align*}
 and
\begin{align*}
I_{5,eL}&:=\| \pa_x^2(u_e\cdot\na \theta_L)\|_{L^1 L^2}, & I_{5,Le}&:=\| \pa_x(\pa_xu_L\cdot\na \theta_e)\|_{L^1 L^2},\\
I_{5,r}&:=\| \pa_x^2T_\mathfrak{L}\|_{L^1 L^2}, & I_{5,ee}&:=\|\langle \pa_x(\pa_x u_e\cdot\na \theta_e),\pa_x^2 \theta_e\rangle\|_{L^1_t}.
\end{align*}
\textbf{\emph{Step 2.1: Estimates of $I_{3,eL}$, $I_{4,eL}$ and $I_{5,eL}$.}}
By making use of \eqref{es-u-p-x-u-2}, \eqref{px-py-tL-es-inf} and \eqref{E1-df}, we derive
\begin{align}\label{I3-eL-es}
I_{3,eL}\lesssim&\|u_e^{(1)}-tu_e^{(2)}\|_{L^\infty L^2} \|\pa_x\theta_L\|_{L^1 L^\infty}+\|u_e^{(2)}\|_{L^\infty L^2} \|(\pa_y+t\pa_x)\theta_L\|_{L^1 L^\infty}\nonumber\\
\lesssim&\nu^{-
\frac{1}{3}}\|\omega_e\|_{W}F_0\lesssim\nu^{-\frac{1}{3}}\mathcal{E}_1F_0.
\end{align}
Similarly, we have
\begin{align*}
&\|u_e\cdot\na \pa_x\theta_L\|_{L^1 L^2}+\|\pa_x u_e\cdot\na \theta_L\|_{L^1 L^2}\lesssim\nu^{-
\frac{1}{3}}\|\omega_e\|_{W} F_0,\\
&\|u_e\cdot\na \pa_x^2\theta_L\|_{L^1 L^2}+\|\pa_x u_e\cdot\na \pa_x \theta_L\|_{L^1 L^2}\lesssim\nu^{-
\frac{1}{3}}\|\omega_e\|_{W} F_0,\\
&\|\pa_x^2u_e\cdot\na \theta_L\|_{L^1 L^2}\lesssim\nu^{-
\frac{1}{3}}\|\pa_x\omega_e\|_{W} F_0.
\end{align*}
These estimates imply that
\begin{align}
&I_{4,eL}\lesssim\nu^{-\frac{1}{3}}\|\omega_e\|_{W}F_0\lesssim\nu^{-\frac{1}{3}}\mathcal{E}_1F_0, \label{I4-eL-es}\\
&I_{5,eL}\lesssim\nu^{-\frac{1}{3}}(\|\omega_e\|_{W}+\|\pa_x\omega_e\|_W)F_0\lesssim \nu^{-\frac{1}{2}}\mathcal{E}_1 F_0.\label{I5-eL-es}
\end{align}
\textbf{\emph{Step 2.2: Estimates of $I_{4,Le}$ and $I_{5,Le}$.}} Thanks to \eqref{px-py-u1L-es-inf} and \eqref{px-py-u2L-es-inf}, we obtain
\begin{align}\label{I4-Le-es}
I_{4,Le}\lesssim&\|\pa_x u_L^{(1)}\|_{L^1 L^\infty} \|\pa_x\theta_e\|_{L^\infty L^2}+\|\pa_x u_L^{(2)}\|_{L^1 L^\infty} \|\pa_y\theta_e\|_{L^\infty L^2}\nonumber\\
\lesssim&(\|(1+t)^{-1}\|_{L^1_t} \|\pa_x\theta_e\|_{L^\infty L^2}+\|(1+t)^{-2}\|_{L^1_t} \|\pa_y\theta_e\|_{L^\infty L^2})G_0
\nonumber\\
\lesssim&|\ln\nu|\|\pa_x\theta_e\|_{W}G_0
+\|\pa_y\theta_e\|_{W}G_0\lesssim\mathcal{E}_1G_0.
\end{align}
In a similar way, we get
\begin{align}\label{I5-Le-es}
I_{5,Le}\lesssim&|\ln\nu|(\|\pa_x\theta_e\|_{W}+\|\pa_x^2\theta_e\|_{W})G_0+(\|\pa_y\theta_e\|_{W} +\|\pa_x\pa_y\theta_e\|_{W} )G_0\lesssim \nu^{-\frac{1}{6}}\mathcal{E}_1G_0.
\end{align}
\textbf{\emph{Step 2.3: Estimates of $I_{4,ee}$ and $I_{5,ee}$.}} We decompose
\begin{align*}
I_{5,ee}\lesssim&\|\langle\pa_x^2u_e\cdot\na \theta_e,\pa_x^2 \theta_e\rangle\|_{L^1_t}+\|\langle\pa_xu_e\cdot\na\pa_x \theta_e,\pa_x^2 \theta_e\rangle\|_{L^1_t}\\
:=&I_{5,ee,1}+I_{5,ee,2}.
\end{align*}
The conditions $\theta_e|_{y=\pm 1}=0$ and $\operatorname{div} u_{e}=0$ imply
\begin{align*}
I_{4,ee}
=&\|\langle\pa_xu_e^{(1)}\theta_e,\pa_x^2 \theta_e\rangle+\langle\pa_xu_e^{(2)} \theta_e,\pa_x\pa_y \theta_e\rangle\|_{L^1_t},\\
I_{5,ee,1}
=&\|\langle\pa_x^2u_e^{(1)}\theta_e,\pa_x^3 \theta_e\rangle+\langle\pa_x^2u_e^{(2)} \theta_e,\pa_x^2\pa_y \theta_e\rangle\|_{L^1_t},\\
I_{5,ee,2}
=&\|\langle\pa_xu_e^{(1)}\pa_x^2\theta_e,\pa_x^2 \theta_e\rangle+\langle\pa_xu_e^{(2)} \pa_x\pa_y\theta_e,\pa_x^2 \theta_e\rangle\|_{L^1_t}.
\end{align*}
\textbf{\emph{Step 2.3.1: Estimates of $I_{4,ee}$ and $I_{5,ee,1}$.}} It follows from \eqref{th-e-inf-es} that
\begin{align}\label{theta_2_inf_es}
\|\theta_e\|_{L^2 L^\infty}\lesssim&(\nu^{-\frac{1}{6}}|\ln\nu|\| \theta_e\|_{L^\infty L^2}^{\frac{1}{2}}+|\ln\nu|^{-1}\|\pa_x^2 \theta_e\|_{L^\infty L^2}^{\frac{1}{2}})\|\pa_y\theta_e\|_{L^\infty L^2}^{\frac{1}{2}}\lesssim|\ln\nu|^{-1}\mathcal{E}_1.
\end{align}
This inequality, together with $\|\pa_xu_e\|_{L^\infty L^2}\leq \|\om_e\|_{W}$ and \eqref{E1-df} gives
\begin{align}\label{I4-ee-es}
I_{4,ee}\lesssim \|\pa_xu_e\|_{L^\infty L^2}\|\theta_e\|_{L^2 L^\infty}\|\pa_x\na \theta_e\|_{L^2 L^2}
\lesssim\nu^{-\frac{1}{2}}|\ln\nu|^{-1}\mathcal{E}_1^2\|\pa_x\theta_e\|_{W}.
\end{align}
Similarly, we have
\begin{align}\label{I5-ee1-es}
I_{5,ee,1}\lesssim&\|\pa_x^2u_e\|_{L^\infty L^2}\|\theta_e\|_{L^2 L^\infty}\|\pa_x^2\na \theta_e\|_{L^2 L^2}\lesssim\nu^{-\frac{2}{3}}|\ln\nu|^{-1}\mathcal{E}_1^2\|\pa_x^2\theta_e\|_{W}.
\end{align}
\textbf{\emph{Step 2.3.2: Estimate of $I_{5,ee,2}$.}} Making use of \eqref{E1-df} and \eqref{p-x-u-e-1-2-2-inf}, we achieve
\begin{align}\label{I5-ee2-es}
I_{5,ee,2}\lesssim&\|\langle\pa_xu_e^{(1)}\pa_x^2 \theta_e,\pa_x^2 \theta_e\rangle\|_{L^1_t}+\|\langle\pa_xu_e^{(2)}\pa_x\pa_y \theta_e,\pa_x^2 \theta_e\rangle\|_{L^1_t}\nonumber\\
\lesssim&\nu^{-\frac{1}{12}}\|\pa_xu_e^{(1)}\|_{L^2 L^\infty}\|\pa_x^2\theta_e\|_{L^\infty L^2}^2+\nu^{-\frac{1}{12}}\|\pa_xu_e^{(2)}\|_{L^2 L^\infty}\|\pa_x\pa_y\theta_e\|_{L^\infty L^2}\|\pa_x^2\theta_e\|_{L^\infty L^2}\nonumber\\
\lesssim &\nu^{-\frac{1}{2}}|\ln\nu|\mathcal{E}_1^2\|\pa_x^2\theta_e\|_{W}.
\end{align}
\textbf{\emph{Step 2.4: Estimates of $I_{3,r}$, $I_{4,r}$ and $I_{5,r}$.}}  From \eqref{Tr-es-2} and  $G_0\leq c\nu^{\frac{1}{3}}$, we get
\begin{align}\label{Ir-es-3}
I_{3,r}+I_{4,r}+I_{5,r}\lesssim& \big\|\big((1+t)^{-1}+(1+t)^{-2}(1+\ln(1+t))\big)e^{-\nu t}G_0F_0\nonumber\\
&\quad +(1+t)\nu e^{-\nu t}F_0\big\|_{L^1_t}
\lesssim c\nu^{\frac{1}{3}}|\ln\nu|F_0.
\end{align}
\textbf{\emph{Step 2.5: Estimates of $\|\theta_e\|_{W}$, $\|\pa_x\theta_e\|_{W}$ and $\|\pa_x^2\theta_e\|_W$.}} Putting \eqref{I3-eL-es} and \eqref{Ir-es-3} into \eqref{te-px-te-es} yields
\begin{align}\label{te-es-leq}
\|\theta_e\|_{W}
\lesssim \nu^{-
\frac{1}{3}}\mathcal{E}_1F_0+c\nu^{\frac{1}{3}}|\ln\nu|F_0.
\end{align}
Substituting  \eqref{I4-eL-es}, \eqref{I4-Le-es}, \eqref{I4-ee-es} and \eqref{Ir-es-3} into \eqref{px-te-px-te-es}, we obtain
\begin{align}\label{px-te-es-leq}
\|\pa_x\theta_e\|_{W}
\lesssim\mathcal{E}_1G_0+\nu^{-\frac{1}{2}}|\ln\nu|^{-1}\mathcal{E}_1^2
+c\nu^{\frac{1}{3}}|\ln\nu|F_0.
\end{align}
Inserting \eqref{I5-eL-es}, \eqref{I5-Le-es}, \eqref{I5-ee1-es}--\eqref{Ir-es-3} into \eqref{px2-te-px-te-es}, we derive
\begin{align}\label{px2-te-es-leq}
\|\pa_x^2\theta_e\|_{W}
\lesssim\nu^{-\frac{1}{6}}\mathcal{E}_1G_0+\nu^{-\frac{2}{3}}|\ln\nu|^{-1}\mathcal{E}_1^2 +c\nu^{\frac{1}{3}}|\ln\nu|F_0.
\end{align}
\textbf{Step 3: Estimates of $\|\pa_y\theta_e\|_{W}$ and $\|\pa_x\pa_y\theta_e\|_{W}$.} 
Taking the inner product of $\eqref{te-px-te-eq}_1$ with $(\pa_y^2\theta_e, \pa_x\pa_y^2\theta_e)$ and integral by parts produce
\begin{align}
\|\pa_y\theta_e\|_{W}^2
\lesssim&(I_{6,eL}+I_{6,Le}+I_{6,r}+\nu^{-\frac{1}{2}}I_{y,r}+\nu^{-\frac{1}{6}}\|\pa_x\theta_e\|_{W})\|\pa_y\theta_e\|_{W}+I_{6,ee},\label{te-py-te-es}\\
\|\pa_x\pa_y\theta_e\|_{W}^2
\lesssim&(I_{7,eL}+I_{7,Le}+I_{7,r}+\nu^{-\frac{1}{2}}I_{y,r}
+\nu^{-\frac{1}{6}}\|\pa_x^2\theta_e\|_{W})\|\pa_x\pa_y\theta_e\|_{W}+I_{7,ee}.\label{px-te-py-te-es}
\end{align}
Where
\begin{align*}
I_{6,eL}&:=\| \pa_y(u_e\cdot\na \theta_L)\|_{L^1 L^2}, & I_{6,Le}&:=\| \pa_yu_L\cdot\na \theta_e\|_{L^1 L^2},\\
I_{6,r}&:=\| \pa_y(u_L\cdot\na\theta_L)\|_{L^1 L^2}, & I_{6,ee}&:=\|\langle\pa_yu_e\cdot\na \theta_e,\pa_y \theta_e\rangle\|_{L^1_t}, \quad I_{y,r}:=\|T_{r_L}\|_{L^2L^2}
\end{align*}
 and
\begin{align*}
I_{7,eL}&:=\| \pa_x\pa_y(u_e\cdot\na \theta_L)\|_{L^1 L^2}, & I_{7,Le}&:=\| \pa_x(\pa_yu_L\cdot\na \theta_e)\|_{L^1 L^2},\\
I_{7,r}&:=\| \pa_x\pa_y(u_L\cdot\na\theta_L)\|_{L^1 L^2}, & I_{7,ee}&:=\|\langle \pa_x(\pa_y u_e\cdot\na \theta_e)+\pa_x u_e\cdot\na \pa_y\theta_e,\pa_x\pa_y \theta_e\rangle\|_{L^1_t}.
\end{align*}
\textbf{\emph{Step 3.1: Estimates of $I_{6,eL}$ and $I_{7,eL}$.}}
Using  \eqref{px-py-tL-es-inf}, we derive
\begin{align*}
 &\|\pa_x\pa_y\theta_L\|_{L^1 L^\infty}+ \|\pa_y(\pa_y+t\pa_x)\theta_L\|_{L^\infty L^\infty}\nonumber\\
\lesssim&\nu^{-
\frac{1}{6}} \|\pa_x(\pa_y+t\pa_x)\theta_L\|_{L^\infty L^\infty}+\nu^{-\frac{1}{3}}\|\pa_x^2\theta_L\|_{L^\infty L^\infty}\nonumber\\
&+ \|(\pa_y+t\pa_x)^2\theta_L\|_{L^\infty L^\infty}+\nu^{-\frac{1}{6}}\|\pa_x(\pa_y+t\pa_x)\theta_L\|_{L^\infty L^\infty}\lesssim\nu^{-
\frac{1}{3}}F_0,
\end{align*}
which together with \eqref{es-u-p-x-u-2} implies
\begin{align}\label{I6-eL-es-1}
\|u_e\cdot\na \pa_y\theta_L\|_{L^1 L^2}\lesssim&\|u_e^{(1)}\!-tu_e^{(2)}\|_{L^\infty L^2} \|\pa_x\pa_y\theta_L\|_{L^1 L^\infty}+\|u_e^{(2)}\|_{L^1 L^2} \|\pa_y(\pa_y+t\pa_x)\theta_L\|_{L^\infty L^\infty}\nonumber\\
\lesssim&\nu^{-
\frac{1}{2}}\|\omega_e\|_{W}F_0.
\end{align}
On the other hand, we obtain by \eqref{px-py-tL-es-inf} and \eqref{es-u-p-x-u-2},
\begin{align}\label{I6-eL-es-2}
\|\pa_y u_e\cdot\na \theta_L\|_{L^1 L^2}\lesssim&\|\pa_yu_e^{(1)}\!-t\pa_yu_e^{(2)}\|_{L^\infty L^2} \|\pa_x\theta_L\|_{L^1 L^\infty}+\|\pa_yu_e^{(2)}\|_{L^\infty L^2} \|(\pa_y+t\pa_x)\theta_L\|_{L^1 L^\infty}\nonumber\\
\lesssim&\|\omega_e\|_{L^\infty L^2} (\nu^{-
\frac{1}{3}}\|\pa_x\theta_L\|_{L^\infty L^\infty}
+\nu^{-
\frac{1}{6}} \|(\pa_y+t\pa_x)\theta_L\|_{L^\infty L^\infty})\nonumber\\
\lesssim&\nu^{-
\frac{1}{3}}\|\omega_e\|_{W}F_0.
\end{align}
Combining \eqref{I6-eL-es-1} and \eqref{I6-eL-es-2}, we get
\begin{align}\label{I6-eL-es}
I_{6,eL}\lesssim&\nu^{-
\frac{1}{2}}\|\omega_e\|_{W}F_0\lesssim \nu^{-
\frac{1}{2}}\mathcal{E}_1 F_0.
\end{align}
Similarly, we obtain
\begin{align*}
&\|u_e\cdot\na \pa_x\pa_y\theta_L\|_{L^1 L^2}+\|\pa_xu_e\cdot\na \pa_y\theta_L\|_{L^1 L^2}\lesssim \nu^{-
\frac{1}{2}}\|\omega_e\|_{{W}}F_0,\\
&\|\pa_y u_e\cdot\na \pa_x\theta_L\|_{L^1 L^2}+\|\pa_x\pa_y u_e\cdot\na \theta_L\|_{L^1 L^2}\lesssim \nu^{-
\frac{1}{3}}\|\omega_e\|_{{W}}F_0.
\end{align*}
Hence, we have
\begin{align}\label{I7-eL-es}
I_{7,eL}\lesssim&\nu^{-
\frac{1}{2}}\|\omega_e\|_{{W}}F_0+\nu^{-
\frac{1}{3}}\|\pa_x\omega_e\|_{{W}}F_0\lesssim\nu^{-
\frac{1}{2}}\mathcal{E}_1F_0.
\end{align}
\textbf{\emph{Step 3.2: Estimates of $I_{6,Le}$ and $I_{7,Le}$.}} The divergence free condition, \eqref{px-py-oL-es-inf}, \eqref{px-py-u1L-es-inf}, \eqref{px-py-u2L-es-inf} and \eqref{E1-df} give
\begin{align}\label{I6-Le-es}
I_{6,Le}
\lesssim&\|\pa_y u_L^{(1)}\|_{L^1 L^\infty} \|\pa_x\theta_e\|_{L^\infty L^2}+\|\pa_y u_L^{(2)}\|_{L^1 L^\infty} \|\pa_y\theta_e\|_{L^\infty L^2}\nonumber\\
\leq &(\|\omega_L\|_{L^1 L^\infty}+\|\pa_x u_L^{(2)}\|_{L^1 L^\infty}) \|\pa_x\theta_e\|_{L^\infty L^2}+\|\pa_x u_L^{(1)}\|_{L^1 L^\infty} \|\pa_y\theta_e\|_{L^\infty L^2}\nonumber\\
\lesssim&\nu^{-\frac{1}{6}}\|\pa_x\theta_e\|_{W}G_0+|\ln\nu|\|\pa_y\theta_e\|_{W}G_0\lesssim|\ln\nu|\mathcal{E}_1G_0.
\end{align}
Similarly, we have
\begin{align*}
\|\pa_xu_L\cdot\na \pa_y\theta_e\|_{L^1 L^2}\lesssim&\|\pa_x u_L^{(1)}\|_{L^1 L^\infty} \|\pa_x\pa_y\theta_e\|_{L^\infty L^2}+\|\pa_x u_L^{(2)}\|_{L^2 L^\infty} \|\pa_y^2\theta_e\|_{L^2 L^2}\nonumber\\
\lesssim&|\ln\nu|\|\pa_x\pa_y\theta_e\|_{W}G_0+\nu^{-\frac{1}{2}}\|\pa_y\theta_e\|_{W}G_0\lesssim\nu^{-\frac{1}{2}}\mathcal{E}_1G_0,
\end{align*}
and
\begin{align*}
\| \pa_yu_L\cdot\na \pa_x\theta_e\|_{L^1 L^2}
\lesssim&\nu^{-\frac{1}{6}}\|\pa_x^2\theta_e\|_{W}G_0
+|\ln\nu|\|\pa_x\pa_y\theta_e\|_{W}G_0\lesssim\nu^{-\frac{1}{6}}|\ln\nu|\mathcal{E}_1G_0,\\
\| \pa_x\pa_yu_L\cdot\na  \theta_e\|_{L^1 L^2}\lesssim&\nu^{-\frac{1}{6}}\|\pa_x\theta_e\|_{W}G_0+|\ln\nu|\|\pa_y\theta_e\|_{W}G_0\lesssim|\ln\nu|\mathcal{E}_1G_0.
\end{align*}
These estimates imply that
\begin{align}\label{I7-Le-es}
I_{7,Le}\lesssim&\nu^{-\frac{1}{2}}\mathcal{E}_1G_0.
\end{align}
\textbf{\emph{Step 3.3: Estimates of $I_{6,ee}$ and $I_{7,ee}$.}} Noting that
\begin{align*}
&\langle u_e^{(1)}\pa_x\pa_y\theta_e,\pa_y \theta_e\rangle=-\langle \pa_yu_e^{(1)}\pa_x\theta_e,\pa_y \theta_e\rangle-\langle u_e^{(1)}\pa_x\theta_e,\pa_y^2 \theta_e\rangle,
\end{align*}
we deduce
\begin{align}\label{I6-ee-ts}
I_{6,ee}=&\|\langle\pa_yu_e^{(1)}\pa_x\theta_e-\pa_xu_e^{(1)}\pa_y \theta_e,\pa_y \theta_e\rangle\|_{L^1_t}\nonumber\\
=&\|\langle\pa_yu_e^{(1)}\pa_x\theta_e,\pa_y \theta_e\rangle+2\langle u_e^{(1)}\pa_x\pa_y\theta_e,\pa_y \theta_e\rangle\|_{L^1_t}\nonumber\\
=&\|-\langle\pa_yu_e^{(1)}\pa_x\theta_e,\pa_y \theta_e\rangle-2\langle u_e^{(1)}\pa_x\theta_e,\pa_y^2 \theta_e\rangle\|_{L^1_t}\nonumber\\
\lesssim&\|\langle\pa_yu_e^{(1)}\pa_x\theta_e,\pa_y \theta_e\rangle\|_{L^1_t}+\|\langle u_e^{(1)}\pa_x\theta_e,\pa_y^2 \theta_e\rangle\|_{L^1_t}:=I_{6,ee,1}+I_{6,ee,2}.
\end{align}
By the expression of $I_{7,ee}$, we have
\begin{align}\label{I7-ee-ts}
I_{7,ee}\lesssim&\|\langle \pa_xu_e\cdot\na \pa_y\theta_e,\pa_x\pa_y \theta_e\rangle\|_{L^1_t}+\|\langle \pa_yu_e\cdot\na \pa_x\theta_e,\pa_x\pa_y \theta_e\rangle\|_{L^1_t}\nonumber\\
&+\|\langle \pa_x\pa_yu_e\cdot\na  \theta_e,\pa_x\pa_y \theta_e\rangle\|_{L^1_t}\nonumber\\
:=&I_{7,ee,1}+I_{7,ee,2}+I_{7,ee,3}.
\end{align}
\textbf{\emph{Step 3.3.1: Estimate of $I_{6,ee}$.}} Using 
$\|f\|_{L^4L^4}\lesssim\|f\|_{L^\infty L^2}^{\frac{1}{2}}\|\na f\|_{L^2L^2}^{\frac{1}{2}}$ when $f|_{\pm 1}=0$ and \eqref{E1-df}, we derive
\begin{align}\label{I6-ee1-es}
I_{6,ee,1}\lesssim&\|1\|_{L^2_t}\|\pa_y u_e^{(1)}\|_{L^4 L^4} \|\pa_x\theta_e\|_{L^4 L^4}\|\pa_y \theta_e\|_{L^\infty L^2}\nonumber\\
\lesssim&\nu^{-\frac{1}{12}}\|\omega_e\|_{L^\infty L^2}^\frac{1}{2} \|\na \omega_e\|_{L^2 L^2}^\frac{1}{2} \|\pa_x\theta_e\|_{L^\infty L^2}^\frac{1}{2} \|\partial_x\na \theta_e\|_{L^2 L^2}^\frac{1}{2} \|\pa_y \theta_e\|_{L^\infty L^2} \nonumber\\
\lesssim&\nu^{-
\frac{7}{12}}\|\omega_e\|_{W}\|\pa_x\theta_e\|_{W}\|\pa_y\theta_e\|_{W}\lesssim\nu^{-
\frac{5}{12}}\mathcal{E}_1^2\|\pa_y\theta_e\|_{W}.
\end{align}
Thanks to \eqref{u1-e-inf-es}, \eqref{u2-e-inf-es} and \eqref{E1-df}, we obtain
\begin{align}\label{u-e-1-2-2-inf}
\|u_e\|_{L^2 L^\infty} \lesssim&\nu^{\frac{1}{4}}\|1\|_{L^2_t}^{\frac{1}{2}} \|\pa_x\omega_e\|_{L^2L^2}^{\frac{1}{2}}
\|\omega_e\|_{L^\infty L^2}^{\frac{1}{2}}+|\ln\nu|\|1\|_{L^2_t}\|\omega_e\|_{L^\infty L^2}\nonumber\\
\lesssim&\nu^{-\frac{1}{12}}|\ln\nu|\|\omega_e\|_{W}\lesssim \nu^{-\frac{1}{12}}|\ln\nu|\mathcal{E}_1.
\end{align}
It follows from \eqref{u-e-1-2-2-inf} that
\begin{align}\label{I6-ee2-es}
I_{6,ee,2}\lesssim&\|u_e^{(1)}\|_{L^2 L^\infty} \|\pa_y^2\theta_e\|_{L^2 L^2}\|\pa_x \theta_e\|_{L^\infty L^2}\lesssim\nu^{-
\frac{5}{12}}|\ln\nu|\mathcal{E}_1^2\|\pa_y\theta_e\|_{W}.
\end{align}
Putting \eqref{I6-ee1-es} and \eqref{I6-ee2-es} into \eqref{I6-ee-ts}, we get
\begin{align}\label{I6-ee-es}
I_{6,ee}\lesssim&\nu^{-
\frac{5}{12}}|\ln\nu|\mathcal{E}_1^2\|\pa_y\theta_e\|_{W}.
\end{align}
\textbf{\emph{Step 3.3.2: Estimate of $I_{7,ee}$.}} By \eqref{E1-df} and \eqref{p-x-u-e-1-2-2-inf}, we obtain
\begin{align*}
I_{7,ee,1}=&\|\langle\pa_xu_e^{(1)}\pa_x\pa_y\theta_e,\pa_x\pa_y \theta_e\rangle+\langle\pa_xu_e^{(2)} \pa_y^2\theta_e,\pa_x\pa_y \theta_e\rangle\|_{L^1_t}\nonumber\\
\lesssim&\|1\|_{L^2_t}\|\pa_xu_e^{(1)}\|_{L^2 L^\infty}\|\pa_x\pa_y\theta_e\|_{L^\infty L^2}^2+\|\pa_xu_e^{(2)}\|_{L^2 L^\infty}\|\pa_y^2\theta_e\|_{L^2 L^2}\|\pa_x\pa_y\theta_e\|_{L^\infty L^2}\nonumber\\
\lesssim&\nu^{-
\frac{3}{4}}|\ln\nu|\mathcal{E}_1^2\|\pa_x\pa_y\theta_e\|_{W}.
\end{align*}
By the same way as in the division \eqref{I6-ee-ts} and the estimates \eqref{I6-ee1-es}-\eqref{I6-ee2-es}, we conclude
\begin{align*}
I_{7,ee,2}\lesssim&\nu^{-
\frac{7}{12}}\|\omega_e\|_{W}\|\pa_x^2\theta_e\|_W\|\pa_x\pa_y\theta_e\|_W+\|u_e^{(1)}\|_{L^2 L^\infty} \|\pa_x\pa_y^2\theta_e\|_{L^2 L^2}\|\pa_x^2 \theta_e\|_{L^\infty L^2}\nonumber\\
\lesssim&\nu^{-
\frac{7}{12}}|\ln\nu|\mathcal{E}_1^2\|\pa_x\pa_y\theta_e\|_{W}.
\end{align*}
Utilizing \eqref{theta_2_inf_es}, we have
\begin{align*}
I_{7,ee,3}
=&\|\langle\pa_x\pa_yu_e^{(1)}\theta_e,\pa_x^2\pa_y \theta_e\rangle+\langle\pa_x\pa_yu_e^{(2)} \theta_e,\pa_x\pa_y^2 \theta_e\rangle\|_{L^1_t}\nonumber\\
\lesssim&\|\pa_x\omega_e\|_{L^\infty L^2}\|\theta_e\|_{L^2 L^\infty}\|\pa_x\pa_y\na \theta_e\|_{L^2 L^2}\lesssim\nu^{-\frac{2}{3}}|\ln\nu|^{-1}\mathcal{E}_1^2\|\pa_x\pa_y\theta_e\|_{W}.
\end{align*}
Putting the above three estimates into \eqref{I7-ee-ts}, we get
\begin{align}\label{I7-ee-es}
I_{7,ee}\lesssim&\nu^{-\frac{3}{4}}|\ln\nu|\mathcal{E}_1^2\|\pa_x\pa_y\theta_e\|_{W}.
\end{align}
\textbf{\emph{Step 3.4: Estimates of $I_{6,r}$, $I_{7,r}$ and $I_{y,r}$.}} From  \eqref{Tr-es-2}--\eqref{TrL-es-2} and $G_0\leq c\nu^{\frac{1}{3}}$, 
\begin{align}\label{Ir-es-6-7}
I_{6,r}+I_{7,r}+\nu^{-\frac{1}{2}}I_{y,r}\lesssim& \big\|\big((1+\nu t^3)^{-\frac{3}{4}}+(1+t)^{-1}(1+\nu t^3)^{-\frac{1}{4}}\big)e^{-\nu t}\|_{L^1_t}G_0F_0\nonumber\\
&+\nu^{\frac{1}{2}}\big\|\big((1+\nu t^3)^{-\frac{1}{4}}+(1+\nu t^3)^{-\frac{3}{4}}t\big) e^{-\nu t}\big\|_{L^2_t}F_0\lesssim c\nu^{\frac{1}{6}}F_0.
\end{align}
\textbf{\emph{Step 3.5: Estimates of $\|\pa_y\theta_e\|_{W}$ and $\|\pa_x\pa_y\theta_e\|_W$.}} Substituting \eqref{I6-eL-es}, \eqref{I6-Le-es}, \eqref{I6-ee-es} and \eqref{Ir-es-6-7} into \eqref{te-py-te-es}, we deduce
\begin{align}\label{py-te-es-leq}
\|\pa_y\theta_e\|_{W}\lesssim&\nu^{-\frac{1}{6}}\mathcal{E}_1G_0+\nu^{-
\frac{5}{12}}|\ln\nu|\mathcal{E}_1^2+c\nu^{\frac{1}{6}}F_0+\nu^{-\frac{1}{6}}\|\pa_x\theta_e\|_{W}.
\end{align}
Plugging \eqref{I7-eL-es}, \eqref{I7-Le-es}, \eqref{I7-ee-es}--\eqref{Ir-es-6-7} into \eqref{px-te-py-te-es}, we derive
\begin{align}\label{py2-te-es-leq}
&\|\pa_x\pa_y\theta_e\|_{W}
\lesssim\nu^{-\frac{1}{2}}\mathcal{E}_1G_0+\nu^{-\frac{3}{4}}|\ln\nu|\mathcal{E}_1^2 +c\nu^{\frac{1}{6}}F_0+\nu^{-\frac{1}{6}}\|\pa_x^2\theta_e\|_{W}.
\end{align}
\textbf{Step 4: Proof of Proposition \ref{prop-4.1}.} 
Inserting the estimates \eqref{oe-es-leq}--\eqref{px-oe-es-leq}, \eqref{te-es-leq}, \eqref{py-te-es-leq}--\eqref{py2-te-es-leq}
 into \eqref{E1-df}, and noting that $$G_0=E_0+\nu^{-\frac{1}{3}}F_0\leq c\nu^{\frac{1}{3}},\quad F_0\leq  c\nu^{\frac{2}{3}}|\ln\nu|^{-4},$$ 
 we obtain
\begin{align*}
\mathcal{E}_1
\lesssim&\nu^{-\frac{1}{3}}\mathcal{E}_1(E_0+\nu^{-\frac{1}{3}}|\ln\nu|^4F_0)
+\nu^{-\frac{7}{12}}|\ln\nu|\mathcal{E}_1^2\nonumber\\
&+c\nu^{\frac{1}{3}}|\ln\nu|(E_0+\nu^{-\frac{1}{3}}|\ln\nu|^4F_0)+\nu^{-\frac{1}{6}}\|\pa_x\theta_e\|_{W}+\|\pa_x^2\theta_e\|_{W}.
\end{align*}
Putting \eqref{px-te-es-leq} and \eqref{px2-te-es-leq} into the above estimate, we arrive at
\begin{align*}
\mathcal{E}_1
\lesssim&\nu^{-\frac{1}{3}}\mathcal{E}_1(E_0+\nu^{-\frac{1}{3}}|\ln\nu|^4F_0)+\nu^{-\frac{2}{3}}|\ln\nu|^{-1}\mathcal{E}_1^2 +c\nu^{\frac{1}{3}}|\ln\nu|(E_0+\nu^{-\frac{1}{3}}|\ln\nu|^4F_0).
\end{align*}
Thus the claim \eqref{claim} is concluded. Combining \eqref{claim} and \eqref{te-es-leq}, Proposition~\ref{prop-4.1} is completed.
\end{proof}

\begin{proposition}\label{prop-4.2}
Let $(\omega_{e}, \theta_{e})$ be the solution of \eqref{oe-te-eq-leq}. For each $t\in(0,\nu^{-\frac{1}{6}})$, we have
\begin{align*}
&\||\pa_x|^{\frac{1}{3}}\theta_e(t)\|_{L^2_{x,y}}\lesssim \nu^{\frac{2}{3}}|\ln\nu|^2(E_0+\nu^{-\frac{1}{3}}|\ln\nu|^{4}F_0).
\end{align*}
\end{proposition}
\begin{proof}

Using the energy method for $\eqref{oe-te-eq-leq}_{2}$, together with $$\||\pa_x|^{\frac{1}{3}}(fg)\|_{L^2_y} \lesssim \|(|\pa_x|^{\frac{1}{3}}f)\,g\|_{L^2_y} + \|f\,(|\pa_x|^{\frac{1}{3}}g)\|_{L^2_y},$$ we derive
\begin{align}
&\| |\pa_x|^{\frac{1}{3}}(u_e\cdot\na \theta_L)\|_{L^1 L^2}
\lesssim\nu^{-\frac{1}{3}}\|\omega_e\|_{W}F_0\lesssim\nu^{\frac{1}{3}}|\ln\nu|F_0,\label{p-x-1-3-u-e-t-L}\\
&\| |\pa_x|^{\frac{1}{3}}u_L\cdot\na \theta_e\|_{L^1 L^2}\lesssim|\ln\nu|\|\pa_x\theta_e\|_{W}G_0 +\|\pa_y\theta_e\|_{W}G_0 \lesssim\nu^{\frac{2}{3}}|\ln\nu|G_0,\label{p-x-1-3-u-L-t-e}\\
&\| |\pa_x|^{\frac{1}{3}}T_\mathfrak{L}\|_{L^1 L^2}\lesssim c\nu^{\frac{1}{3}}|\ln\nu|F_0.\label{p-x-1-3-T-L}
\end{align}
By \eqref{px-13-ue-1-es} and \eqref{px-13-ue-2-es}, we obtain
\begin{align}\label{p-x-1-3-u-e-inf-inf}
\||\pa_x|^{\frac{1}{3}}u_e\|_{L^\infty L^\infty}
\lesssim&\nu^{\frac{1}{12}}\|\pa_x\omega_e\|_{W}^{\frac{1}{2}}\|\omega_e\|_{W}^{\frac{1}{2}}
+\nu^{-\frac{1}{6}}\|\omega_e\|_{W}\nonumber\\
\lesssim&\nu^{\frac{1}{6}}|\ln\nu|(E_0+\nu^{-\frac{1}{3}}|\ln\nu|^4F_0).
\end{align}
From \eqref{p-x-1-3-u-e-inf-inf}, we have
\begin{align}\label{p-x-1-3-u-e-t-e}
&\| |\pa_x|^{\frac{1}{3}}u_e\cdot\na \theta_e\|_{L^1 L^2}\nonumber\\
\lesssim&\nu^{-\frac{1}{6}}\||\pa_x|^{\frac{1}{3}}u_e^{(1)}\|_{L^\infty L^\infty}\|\pa_x\theta_e\|_{L^\infty L^2}+\nu^{-\frac{1}{6}}\||\pa_x|^{\frac{1}{3}}u_e^{(2)}\|_{L^\infty L^\infty}\|\pa_y\theta_e\|_{L^\infty L^2}\nonumber\\
\lesssim&\nu^{\frac{1}{3}}|\ln\nu|^{2}(E_0+\nu^{-\frac{1}{3}}|\ln\nu|^4F_0)^2\lesssim\nu^{\frac{2}{3}}|\ln\nu|^{2}(E_0+\nu^{-\frac{1}{3}}|\ln\nu|^4F_0).
\end{align}
It follows from \eqref{p-x-1-3-u-e-t-L}--\eqref{p-x-1-3-T-L} and \eqref{p-x-1-3-u-e-t-e} that 
\begin{align*}
\||\pa_x|^{\frac{1}{3}}\theta_e\|_{W}\lesssim&\| |\pa_x|^{\frac{1}{3}}(u_e\cdot\na \theta_L)\|_{L^1 L^2}+\| |\pa_x|^{\frac{1}{3}}u_L\cdot\na \theta_e\|_{L^1 L^2}\nonumber\\
&+\| |\pa_x|^{\frac{1}{3}}u_e\cdot\na \theta_e\|_{L^1 L^2}+\| |\pa_x|^{\frac{1}{3}}T_\mathfrak{L}\|_{L^1 L^2}\nonumber\\
\lesssim&\nu^{\frac{2}{3}}|\ln\nu|^2(E_0+\nu^{-\frac{1}{3}}|\ln\nu|^{4}F_0).
\end{align*}
Thus we complete the proof of Proposition~\ref{prop-4.2}.
\end{proof}

\begin{Remark}\label{rem-4.1}
For the domain $\mathbb{T}\times [-1,1]$, set 
\begin{align*}
\tilde{\mathcal{E}}_1=&\|\omega_e\|_{W}+\nu^{-\frac{1}{3}}\|\theta_e\|_{W}+\nu^{-
\frac{1}{6}}\|\pa_x\theta_e\|_{W}+\|\pa_y\theta_e\|_{W}\nonumber\\
&+\nu^{\frac{1}{6}}\|\pa_x\omega_e\|_{W}+\|\pa_x^2\theta_e\|_{W}+\nu^{\frac{1}{6}}\|\pa_x\pa_y\theta_e\|_{W}.
\end{align*}
 Substituting the estimate of Remark \ref{rem-3.1} into \eqref{theta_2_inf_es}, and by  analogous argument to Propositions \ref{prop-4.1} and \ref{prop-4.2}, we obtain for $t\in(0,\nu^{-\frac{1}{6}})$,
 $$\|\omega_e\|_{L^2_{x,y}}+\nu^{-\frac{1}{3}}\|\theta_e\|_{L^2_{x,y}}+\nu^{-\frac{1}{3}}\||\pa_x|^{\frac{1}{3}}\theta_e\|_{L^2_{x,y}}\lesssim\nu^\frac{1}{3}(E_0+\nu^{-\frac{1}{3}}F_0).$$
\end{Remark}

\section{\texorpdfstring%
  {Estimates for the system \eqref{ote-eq} in the case $t>\nu^{-\frac{1}{6}}$}%
  {Estimates for the system (Boue) in the case t > nu\textasciicircum(-1/6)}}\label{non-es-geq}

In this section, we provide the nonlinear estimates for $t> T_0=\nu^{-\frac{1}{6}}$. We  decompose $\omega_e=\omega_\text{b}+\omega_\mathfrak{E}$, $\theta_e=\theta_\text{b}+\theta_\mathfrak{E}$ such that
\begin{align}\label{o-1-t-1-eq}
\begin{cases}
\pa_t \omega_\text{b}-\nu\Delta\omega_\text{b}+y\pa_x\omega_\text{b}+u\cdot\na \omega_e+u_e^{(1)}\pa_x \omega_L+u_e^{(2)}(\pa_y+t\pa_x) \omega_L+E_\mathfrak{L}=-\pa_x\theta_\text{b}, \\
\pa_t \theta_\text{b}-\nu\Delta\theta_\text{b}+y\pa_x\theta_\text{b}+u\cdot\na \theta_e+u_e^{(1)}\pa_x \theta_L+u_e^{(2)}(\pa_y+t\pa_x) \theta_L+T_\mathfrak{L}=0, \\
\omega_\text{b}(t,x,\pm 1)=\theta_\text{b}(t,x,\pm 1)=0, \\
\omega_\text{b}|_{t=T_0}=\omega_e(T_0),\quad  \theta_\text{b}|_{t=T_0}=\theta_e(T_0),
\end{cases}
\end{align}
and
\begin{align}\label{o-2-t-2-eq}
\begin{cases}
\pa_t \omega_\mathfrak{E}-\nu\Delta\omega_\mathfrak{E}+y\pa_x\omega_\mathfrak{E}=tu_e^{(2)}\pa_x\omega_L-\pa_x\theta_\mathfrak{E}, \\
\pa_t \theta_\mathfrak{E}-\nu\Delta\theta_\mathfrak{E}+y\pa_x\theta_\mathfrak{E}=tu_e^{(2)}\pa_x\theta_L, \\
\omega_\mathfrak{E}(t,x,\pm 1)=\theta_\mathfrak{E}(t,x,\pm 1)=0, \\
\omega_\mathfrak{E}|_{t=T_0}=\theta_\mathfrak{E}|_{t=T_0}=0.
\end{cases}
\end{align}

Next, we state the space-time estimates for $(\omega_\text{b}, \theta_\text{b})$  and $(\omega_\mathfrak{E}, \theta_\mathfrak{E})$. Finally, we derive the estimate for $(\omega_e, \theta_e)$ and complete the proof of Theorem \ref{th-1.1}. For convenience, we define
\begin{align*}
\mathcal{E}_{2}:=&\|\omega_{e}\|_{X(T_0,T)}+\nu^{-\frac{1}{3}}|\ln\nu|^4\|\theta_{e}\|_{X(T_0,T)}
+\nu^{-\frac{1}{3}}\||\pa_x|^{\frac{1}{3}}\theta_{e}\|_{X(T_0,T)},\nonumber\\
\mathcal{E}_{2,0}:=&\nu^{\frac{1}{3}}|\ln\nu|^3(E_0+\nu^{-\frac{1}{3}}|\ln\nu|^{4}F_0).
\end{align*}
and
$$\|f\|_{L^pL^q}:=\|f\|_{L^p_t(T_0,T;L^q_{x,y})}, \quad \|f\|_{L^p_t}:=\|f\|_{L^p_t(T_0,T)}, \quad \|f\|_{\Box}:=\|f\|_{\Box(T_0,T)}, $$where $\Box$ can be denoted as $X$, $Y$ or $\mathcal{Y}$. 
\subsection{Nonlinear estimates of  \texorpdfstring{$(\omega_\text{b}$, $\theta_\text{b})$}{omega\_b, theta\_b}}
\begin{lemma}\label{lem-5.1}
Let $g_1$ and $g_2$ be 
\begin{equation*}
\begin{aligned}
g_1&=(u_e^{(1)}\pa_x \omega_L+u_e^{(2)}(\pa_y+t\pa_x) \omega_L)+u_L\cdot\na\omega_e+u_e\cdot\na\omega_e+E_\mathfrak{L},\\
g_2&=(u_e^{(1)}\pa_x \theta_L+u_e^{(2)}(\pa_y+t\pa_x) \theta_L)+u_L\cdot\na\theta_e+u_e\cdot\na\theta_e+T_\mathfrak{L}.
\end{aligned}
\end{equation*}
 We have the following estimates:
\begin{align}
\|g_1\|_{L^1L^2}
\lesssim&\nu^{-\frac{1}{3}}\mathcal{E}_2G_0+\nu^{-\frac{1}{2}}|\ln\nu|\mathcal{E}_2^2 +\mathcal{E}_{2,0},\label{g-1-g-2-px-g-2-leq-1}\\
\|g_2\|_{L^1L^2}\lesssim&\nu^{-\frac{1}{3}}\mathcal{E}_2F_0+\nu^{-\frac{1}{6}}|\ln\nu|^{-3}\mathcal{E}_2^2
+\nu^{\frac{1}{3}}|\ln\nu|^{-6}\mathcal{E}_{2,0},\label{g-1-g-2-px-g-2-leq-2}\\
\||\pa_x|^{\frac{1}{3}}g_2\|_{L^1L^2}\lesssim&\mathcal{E}_2G_0 +\nu^{-\frac{2}{9}}\mathcal{E}_2^2+\nu^{\frac{1}{3}}|\ln\nu|^{-6}\mathcal{E}_{2,0}.\label{g-1-g-2-px-g-2-leq}
\end{align}
\end{lemma}

 \begin{proof}
 We only present the proof of \eqref{g-1-g-2-px-g-2-leq} here,  since the estimates \eqref{g-1-g-2-px-g-2-leq-1}--\eqref{g-1-g-2-px-g-2-leq-2} can be obtained in a similar manner. For convenience, we denote
 $$J_{eL}:= (u_e^{(1)}\pa_x \theta_L+u_e^{(2)}(\pa_y+t\pa_x) \theta_L), \quad  J_{Le}:=u_L\cdot\na\theta_e, \quad J_{ee}:=u_e\cdot\na\theta_e.$$
 
\textbf{Step 1: Estimate of $|\pa_x|^{\frac{1}{3}}J_{eL}$}. Using \eqref{px-py-oL-es-inf} and \eqref{es-u-p-x-u-2} produces
\begin{align*}
\||\pa_x|^{\frac{1}{3}}J_{eL}\|_{L^1L^2}
\lesssim&\||\pa_x|^{\frac{1}{3}}u_e^{(1)}\pa_x\theta_L\|_{L^1L^2}
+\||\pa_x|^{\frac{1}{3}}u_e^{(2)}(\pa_y+t\pa_x)\theta_L\|_{L^1L^2} \nonumber\\
&+\|u_e^{(1)}|\pa_x|^{\frac{4}{3}}\theta_L\|_{L^1L^2}
+\|u_e^{(2)}|\pa_x|^{\frac{1}{3}}(\pa_y+t\pa_x)\theta_L\|_{L^1L^2}\nonumber\\
\lesssim&\|\omega_e\|_{L^\infty L^2}\big(\|\pa_x\theta_L\|_{L^1 L^\infty}
+\|(\pa_y+t\pa_x)\theta_L\|_{L^1 L^\infty}\nonumber\\
&+\||\pa_x|^{\frac{4}{3}}\theta_L\|_{L^1 L^\infty}
+\||\pa_x|^{\frac{1}{3}}(\pa_y+t\pa_x)\theta_L\|_{L^1 L^\infty}\big)\nonumber\\
\lesssim&\|\omega_e\|_{X}\|(1+\nu t^3)^{-\frac{1}{2}}\|_{L^1_t}F_0\lesssim\nu^{-\frac{1}{3}}\mathcal{E}_2F_0.
\end{align*}

\textbf{Step 2: Estimate of $|\pa_x|^{\frac{1}{3}}J_{Le}$}. In light of \eqref{px-py-u1L-es-inf}--\eqref{px-py-u2L-es-inf}, we obtain
\begin{align*}
\||\pa_x|^{\frac{1}{3}}J_{Le}\|_{L^1L^2}\lesssim&\||\pa_x|^{\frac{1}{3}}u_L^{(1)}\|_{L^\infty L^\infty}\|\pa_x\theta_e    \|_{L^1L^2}+\||\pa_x|^{\frac{1}{3}}u_L^{(2)}\|_{L^2 L^\infty}\|\pa_y\theta_e\|_{L^2L^2}\nonumber\\
&+\|u_L^{(1)}\|_{L^\infty L^\infty}\||\pa_x|^{\frac{4}{3}}\theta_e    \|_{L^1L^2}+\|u_L^{(2)}\|_{L^2 L^\infty}\||\pa_x|^{\frac{1}{3}}\pa_y\theta_e\|_{L^2L^2}\nonumber\\
\lesssim&\nu^{-\frac{1}{6}}(\|(1+t)^{-1}\|_{L^\infty_t}+\|(1+t)^{-2}(1+\ln(1+t))\|_{L^2_t})\mathcal{E}_2G_0\nonumber\\
\lesssim&\mathcal{E}_{2}G_0.
\end{align*}

\textbf{Step 3: Estimate of $|\pa_x|^{\frac{1}{3}}J_{ee}$}.  Thanks to \eqref{u1-e-inf-es}--\eqref{u2-e-inf-es}, we have
\begin{align}\label{ue-1-px-t-1-12}
&\|u_e^{(1)}|\pa_x|^{\frac{4}{3}}\theta_e\|_{L^1L^2}\lesssim\big\|\|u_e^{(1)}\|_{L^\infty}\||\pa_x|^{\frac{4}{3}}\theta_e\|_{L^2}
\big\|_{L^1_t}\nonumber\\
\lesssim&\nu^{\frac{1}{4}}\|\pa_x\omega_e\|_{L^2L^2}^{\frac{1}{2}}\|\pa_x u_e\|_{L^2L^2}^{\frac{1}{2}}
\||\pa_x|^{\frac{4}{3}}\theta_e\|_{L^2L^2}+\|\omega_e\|_{L^\infty L^2}\||\pa_x|^{\frac{4}{3}}\theta_e\|_{L^1L^2}\nonumber\\
&+|\ln\nu|\|\omega_e\|_{L^\infty L^2}^{\frac{1}{2}}\|\pa_x u_e\|_{L^2L^2}^{\frac{1}{2}}\||\pa_x|^{\frac{4}{3}}\theta_e\|_{L^1L^2}^{\frac{1}{2}}\||\pa_x|^{\frac{4}{3}}\theta_e\|_{L^2L^2}^{\frac{1}{2}}\nonumber\\
\lesssim&\nu^{-\frac{1}{2}}|\ln\nu|\|\omega_e\|_{X}\||\pa_x|^{\frac{1}{3}}\theta_e\|_{X}\lesssim\nu^{-\frac{1}{6}}|\ln\nu|\mathcal{E}_2^2,
\end{align}
and
\begin{align}
&\|u_e^{(2)}|\pa_x|^{\frac{1}{3}}\pa_y\theta_e\|_{L^1L^2}\lesssim\big\|\|u_e^{(2)}\|_{L^\infty}\||\pa_x|^{\frac{1}{3}}\pa_y\theta_e\|_{L^2}
\big\|_{L^1_t}\nonumber\\
\lesssim&\nu^{\frac{1}{4}}\|\pa_x\omega_e\|_{L^2L^2}^{\frac{1}{2}}\|\pa_x u_e\|_{L^2L^2}^{\frac{1}{2}}
\||\pa_x|^{\frac{1}{3}}\pa_y\theta_e\|_{L^2L^2}+|\ln\nu|\|\pa_x u_e\|_{L^2L^2}\||\pa_x|^{\frac{1}{3}}\pa_y\theta_e\|_{L^1L^2}\nonumber\\
\lesssim&\nu^{-\frac{1}{2}}|\ln\nu|\|\omega_e\|_{X}\||\pa_x|^{\frac{1}{3}}\theta_e\|_{X}\lesssim\nu^{-\frac{1}{6}}|\ln\nu|\mathcal{E}_2^2.
\end{align}
Due to \eqref{px-13-ue-1-es}, \eqref{px-13-ue-3-es}, we obtain
\begin{align*}
\Big\|\||\pa_x|^{\frac{1}{3}}u_e^{(1)}\|_{L^2_xL^\infty_y}^{\frac{2}{3}}\Big\|_{L^6_t}
\lesssim&\|\mathcal{M}u_e\|_{L^2L^2}^{\frac{1}{3}}\|\omega_e\|_{L^\infty L^2}^{\frac{1}{3}}
\lesssim\|\omega_e\|_{X}^{\frac{2}{3}}\lesssim\mathcal{E}_2^{\frac{2}{3}},
\end{align*}
and
\begin{align*}
&\Big\|
\||\pa_x|^{\frac{1}{3}}u_e^{(1)}\|_{L^\infty_xL^\infty_y}^{\frac{1}{3}}
\||\pa_x|^{\frac{4}{3}}\theta_e\|_{L^2_xL^2_y}\Big\|_{L^{\frac{6}{5}}_t}\nonumber\\
\lesssim&\big(\nu^{\frac{1}{36}}\|\pa_x\omega_e\|_{L^2L^2}^{\frac{1}{6}}\|\pa_xu_e\|_{L^2L^2}^{\frac{1}{6}}
+\|\mathcal{M}u_e\|_{L^2L^2}^{\frac{1}{3}}\big)\||\pa_x|^{\frac{4}{3}}\theta_e\|_{L^1L^2}^{\frac{1}{3}}
\||\pa_x|^{\frac{4}{3}}\theta_e\|_{L^2L^2}^{\frac{2}{3}}\nonumber\\
&+\nu^{-\frac{1}{18}}\|\omega_e\|_{L^\infty L^2}^{\frac{1}{6}}\|\pa_x u_e\|_{L^2L^2}^{\frac{1}{6}}\||\pa_x|^{\frac{4}{3}}\theta_e\|_{L^1L^2}^{\frac{1}{2}}
\||\pa_x|^{\frac{4}{3}}\theta_e\|_{L^2L^2}^{\frac{1}{2}}\nonumber\\
\lesssim&\nu^{-\frac{5}{9}}\|\omega_e\|_{X}^{\frac{1}{3}}\||\pa_x|^{\frac{1}{3}}\theta_e\|_{X}\lesssim\nu^{-\frac{2}{9}}\mathcal{E}_2^{\frac{4}{3}},
\end{align*}
which implies
\begin{align}\label{px-ue-1-t-1-12}
&\||\pa_x|^{\frac{1}{3}}u_e^{(1)}\pa_x\theta_e\|_{L^1L^2}\nonumber\\
\lesssim&\Big\|\||\pa_x|^{\frac{1}{3}}u_e^{(1)}\|_{ L^2_xL^\infty_y}^{\frac{2}{3}}\Big\|_{L^6_t}\Big\|
\||\pa_x|^{\frac{1}{3}}u_e^{(1)}\|_{L^\infty_xL^\infty_y}^{\frac{1}{3}}
\||\pa_x|^{\frac{4}{3}}\theta_e\|_{L^2_xL^2_y}\Big\|_{L^{\frac{6}{5}}_t}\lesssim\nu^{-\frac{2}{9}}\mathcal{E}_2^2.
\end{align}
 By means of \eqref{px-13-ue-2-es}, \eqref{px-13-ue-4-es}, we deduce
\begin{align}\label{px-ue-1-t-1-12-2}
&\||\pa_x|^{\frac{1}{3}}u_e^{(2)}\pa_y\theta_e\|_{L^1L^2}\nonumber\\
\lesssim&\||\pa_x|^{\frac{1}{3}}u_e^{(2)}\|_{L^2_tL^2_xL^\infty_y}^{\frac{2}{3}}
\||\pa_x|^{\frac{1}{3}}u_e^{(2)}\|_{L^2_tL^\infty_xL^\infty_y}^{\frac{1}{3}}
\||\pa_x|^{\frac{1}{3}}\pa_y\theta_e\|_{L^2_tL^2_xL^2_y}\nonumber\\
\lesssim&\|\mathcal{M}u_e\|_{L^2L^2}^{\frac{2}{3}}\big(\nu^{\frac{1}{36}}\|\pa_x\omega_e\|_{L^2L^2}^{\frac{1}{6}}
\|\pa_x u_e\|_{L^2L^2}^{\frac{1}{6}}+\nu^{-\frac{1}{18}}\|\pa_x u_e\|_{L^2L^2}^{\frac{1}{3}}\big)
\||\pa_x|^{\frac{1}{3}}\pa_y\theta_e\|_{L^2L^2}\nonumber\\
\lesssim&\nu^{-\frac{5}{9}}\|\omega_e\|_{X}\||\pa_x|^{\frac{1}{3}}\theta_e\|_{X}\lesssim\nu^{-\frac{2}{9}}\mathcal{E}_2^2.
\end{align}
Combining \eqref{ue-1-px-t-1-12}--\eqref{px-ue-1-t-1-12-2}, we deduce that
\begin{align*}
&\||\pa_x|^{\frac{1}{3}}J_{ee}\|_{L^1L^2}\lesssim\||\pa_x|^{\frac{1}{3}}(u_e^{(1)}\pa_x\theta_e)\|_{L^1L^2}
+\||\pa_x|^{\frac{1}{3}}(u_e^{(2)}\pa_y\theta_e)\|_{L^1L^2}\nonumber\\
\lesssim&\|u_e^{(1)}|\pa_x|^{\frac{4}{3}}\theta_e\|_{L^1L^2}
+\|u_e^{(2)}|\pa_x|^{\frac{1}{3}}\pa_y\theta_e\|_{L^1L^2}+\||\pa_x|^{\frac{1}{3}}u_e^{(1)}\pa_x\theta_e\|_{L^1L^2}
+\||\pa_x|^{\frac{1}{3}}u_e^{(2)}\pa_y\theta_e\|_{L^1L^2}\nonumber\\
\lesssim&\nu^{-\frac{2}{9}}\mathcal{E}_2^2.
\end{align*}

\textbf{Step 4: Estimate of $|\pa_x|^{\frac{1}{3}}T_\mathfrak{L} $}. Using \eqref{Tr-es-2} and $G_0\lesssim\nu^{\frac{1}{3}}$, we obtain
\begin{align*}
\||\pa_x|^{\frac{1}{3}}T_\mathfrak{L} \|_{L^1L^2}\lesssim& \Big\|\Big(\frac{1}{(1+t)(1+\nu t^3)^{\frac{3}{4}}}+\frac{1+\ln(1+t)}{(1+t)^{2}(1+\nu t^3)^{\frac{1}{4}}}\Big)e^{-\nu t}\Big\|_{L^1_t}G_0F_0\nonumber\\
&+\big\|\big((1+\nu t^3)^{-\frac{1}{4}}+(1+\nu t^3)^{-\frac{3}{4}}t\big)\nu e^{-\nu t}\big\|_{L^1_t}F_0\lesssim \nu^{\frac{1}{3}}|\ln \nu|F_0.
\end{align*}
Thus we complete the proof of Lemma~\ref{lem-5.1}.
\end{proof}

\begin{proposition}\label{prop-5.2}
Let $(\omega_\text{b}$, $\theta_\text{b})$ be the solution of \eqref{o-1-t-1-eq}. For $T>T_0=\nu^{-\frac{1}{6}}$, it follows
\begin{align*}
&\|\omega_\text{b}\|_{X}+\nu^{-\frac{1}{3}}|\ln\nu|^4\|\theta_\text{b}\|_{X}
+\nu^{-\frac{1}{3}}\||\pa_x|^{\frac{1}{3}}\theta_\text{b}\|_{X}\\
\lesssim&\mathcal{E}_{2,0}+\nu^{-\frac{1}{3}}\mathcal{E}_2(E_0+\nu^{-\frac{1}{3}}|\ln\nu|^4F_0)
+\nu^{-\frac{5}{9}}\mathcal{E}_2^2.
\end{align*}
\end{proposition}

\begin{proof}  In view of \eqref{fX-es}, we obtain
\begin{align}
\|\omega_\text{b}\|_{X}
\lesssim&\|\omega_\text{b}(T_0)\|_{L^2}+\|g_1\|_{L^1L^2}+\|\pa_x\theta_\text{b}\|_{L^1L^2},\label{o-1-geq-es-X-0}\\
\|\theta_\text{b}\|_{X}\lesssim&\|\theta_\text{b}(T_0)\|_{L^2}+\|g_2\|_{L^1L^2},\label{t-1-geq-es-X}\\
\||\pa_x|^{\frac{1}{3}}\theta_\text{b}\|_{X}
\lesssim&\||\pa_x|^{\frac{1}{3}}\theta_\text{b}(T_0)\|_{L^2}+\||\pa_x|^{\frac{1}{3}}g_2\|_{L^1L^2}. \label{px-t-1-geq-es-X}
\end{align}
Replacing $f$ in \eqref{pxfX-es} by $|\pa_x|^{\frac{1}{3}}\theta_\text{b}$, we derive
\begin{align}\label{p-x-t-1-1-2}
\|\pa_x\theta_\text{b}\|_{L^1L^2}\lesssim\nu^{-\frac{1}{3}}|\ln \nu|\||\pa_x|^{\frac{1}{3}}\theta_\text{b}(T_0)\|_{L^2}+\nu^{-\frac{1}{3}}\||\pa_x|^{\frac{1}{3}}g_2\|_{L^1L^2}.
\end{align}
Collecting \eqref{o-1-geq-es-X-0} and \eqref{p-x-t-1-1-2} yields
\begin{align}\label{o-1-geq-es-X}
\|\omega_\text{b}\|_{X}
\lesssim&\|\omega_\text{b}(T_0)\|_{L^2}+\|g_1\|_{L^1L^2}+\nu^{-\frac{1}{3}}|\ln\nu|\||\pa_x|^{\frac{1}{3}}\theta_\text{b}(T_0)\|_{L^2} +\nu^{-\frac{1}{3}}\||\pa_x|^{\frac{1}{3}}g_2\|_{L^1L^2}.
\end{align}
Inserting the estimates in Lemma \ref{lem-5.1} into \eqref{t-1-geq-es-X}, \eqref{px-t-1-geq-es-X} and \eqref{o-1-geq-es-X}, we complete the proof of Proposition \ref{prop-5.2}.
\end{proof}

\subsection{Nonlinear estimates of \texorpdfstring{$(\omega_\mathfrak{E}$, $\theta_\mathfrak{E})$}{omega\_E, theta\_E}}

We further decompose the variable $(\omega_\mathfrak{E},\theta_\mathfrak{E})$ in \eqref{o-2-t-2-eq} as $(\omega_\mathfrak{E},\theta_\mathfrak{E})=(\omega_{r},\theta_{r})+(\omega_\mathfrak{o},0)$, where
\begin{align}\label{o-2-1-o-2-2-eq-1}
\begin{cases}
\pa_t \omega_{{r}}-\nu\Delta\omega_{{r}}+y\pa_x\omega_{{r}}=-\pa_x\theta_{r}, \\
\pa_t \theta_{r}-\nu\Delta\theta_{r}+y\pa_x\theta_{r}=tu_e^{(2)}\pa_x\theta_L, \\
\omega_{r}(t,x,\pm 1)=\theta_{r}(t,x,\pm 1)=0, \\
\omega_{r}|_{t=T_0}=\theta_{r}|_{t=T_0}=0,
\end{cases}
\end{align}
and
\begin{align}\label{o-2-1-o-2-2-eq-2}
\begin{cases}
\pa_t \omega_{\mathfrak{o}}-\nu\Delta\omega_\mathfrak{o}+y\pa_x\omega_\mathfrak{o}=tu_e^{(2)}\pa_x\omega_L, \\
\omega_\mathfrak{o}(t,x,\pm 1)=0, \quad \omega_\mathfrak{o}|_{t=T_0}=0.
\end{cases}
\end{align}
Applying  \cite[Proposition 3.6]{CLM-26} to $\eqref{o-2-1-o-2-2-eq-2}$ and $\eqref{o-2-1-o-2-2-eq-1}_2$ respectively, we obtain 
\begin{proposition}\label{prop-5.1}
Let $\omega_{\mathfrak{o}}$, $\theta_r$ be the solutions of $\eqref{o-2-1-o-2-2-eq-2}$ and $\eqref{o-2-1-o-2-2-eq-1}_2$, respectively. Then
\begin{align*}
&\|\omega_{\mathfrak{o}}\|_{X}\lesssim \nu^{-\frac{1}{3}}\mathcal{E}_2G_0,\\
&\|\theta_r\|_{X}\lesssim \nu^{-\frac{1}{3}}\mathcal{E}_2F_0.
\end{align*}
\end{proposition}

To close the estimates for the system \eqref{ote-eq}, we provide the following proposition to bound $\omega_r$ and $|\pa_x|^{\frac{1}{3}}\theta_r$.

\begin{proposition}\label{prop-5.5}
Let $\omega_r$, $\theta_r$ be the solutions of $\eqref{o-2-1-o-2-2-eq-1}$. For $T> T_0=\nu^{-\frac{1}{6}}$, we have
\begin{align*}
\|\omega_r\|_{X}+\nu^{-\frac{1}{3}}\||\pa_x|^{\frac{1}{3}}\theta_r\|_{X}
\lesssim\nu^{-\frac{2}{3}}\mathcal{E}_2F_0.
\end{align*}
\end{proposition}

\begin{proof}
We first take the horizontal Fourier transform to \eqref{o-2-1-o-2-2-eq-1}:
 \begin{align}\label{equ: om2,the2k}
\begin{cases}
\pa_t \omega_{r,k}-\nu(\pa_y^2-k^2)\omega_{r,k}+iky\omega_{r,k}=-ik\theta_{r,k}, \\
\pa_t \theta_{r,k}-\nu(\pa_y^2-k^2)\theta_{r,k}+iky\theta_{r,k}=\widehat{tu_e^{(2)}\pa_x\theta_L}, \\
\omega_{r,k}(t,k,\pm 1)=\theta_{r,k}(t,k,\pm 1)=0, \\
\omega_{r,k}|_{t=T_0}=\theta_{r,k}|_{t=T_0}=0.
\end{cases}
\end{align}
It follows from the definition of $X$-norm and $Y$-norm that
\begin{align*}
\|\omega_r\|_{X}=&\|\omega_r\|_{Y}+\nu^{\frac{1}{2}}\|\pa_x\omega_r\|_{L^1L^2}\lesssim\|\omega_{r,k}\|_{L^2_k\mathcal{Y}}+\nu^{\frac{1}{2}}\|\pa_x\omega_r\|_{L^1L^2}\nonumber\\
\lesssim&\|\textbf{1}_{\nu k^2\geq 1}\omega_{r,k}\|_{L^2_k\mathcal{Y}}
+\|\textbf{1}_{\nu k^2\leq 1}\omega_{r,k}\|_{L^2_k\mathcal{Y}}
+\nu^{\frac{1}{2}}\|\pa_x\omega_r\|_{L^1L^2}\nonumber\\
:=&K_1+K_2+K_3,
\end{align*}
and
\begin{align*}
\nu^{-\frac{1}{3}}\||\pa_x|^{\frac{1}{3}}\theta_r\|_{X}\lesssim&\nu^{-\frac{1}{3}}\|\textbf{1}_{\nu k^2\geq 1}
|k|^{\frac{1}{3}}\theta_{r,k}\|_{L^2_k\mathcal{Y}}+\nu^{-\frac{1}{3}}\|\textbf{1}_{\nu k^2\leq 1}|k|^{\frac{1}{3}}\theta_{r,k}\|_{L^2_k\mathcal{Y}}+\nu^{\frac{1}{6}}\||\pa_x|^{\frac{4}{3}}\theta_r\|_{L^1L^2}\nonumber\\
:=&K_4+K_5+K_6.
\end{align*}
It suffices to prove
\begin{align*}
K_1+K_2+K_3+K_4+K_5+K_6\lesssim \nu^{-\frac{2}{3}}\mathcal{E}_2F_0.
\end{align*}

For $K_1$, using \eqref{fk-h1-h2-es} with $f=\omega_{r,k}$, and $\textbf{1}_{\nu k^2\geq 1} \nu^{-\frac{1}{2}}\leq \textbf{1}_{\nu k^2\geq 1}\nu^{\frac{1}{3}}|k|^{\frac{5}{3}}$, we obtain
\begin{align}\label{J-1-es}
K_1\lesssim \|\textbf{1}_{\nu k^2\geq 1}\nu^{\frac{1}{3}}|k|^{\frac{5}{3}}\|\theta_{r,k}\|_{L^2_tL^2_y}\|_{L^2_k}
\lesssim \nu^{\frac{1}{3}}\||\pa_x|^{\frac{5}{3}}\theta_r\|_{L^2_tL^2_{x,y}}.
\end{align}

For $K_3$, we utilize \eqref{pxfX-es}  with $f=|\pa_x|^{\frac{1}{3}}\omega_r$ to derive
\begin{align}\label{J-5-es}
K_3\lesssim \nu^{\frac{1}{6}}\||\pa_x|^{\frac{4}{3}}\theta_r\|_{L^1_tL^2_{x,y}}.
\end{align}
Thanks to Lemma \ref{lem-3.2} with $f=|\pa_x|^{\frac{2}{3}}\theta_r$, we achieve
\begin{align*}
\nu^{\frac{1}{3}}\||\pa_x|^{\frac{5}{3}}\theta_r\|_{L^2_tL^2_{x,y}}
+\nu^{\frac{1}{6}}\||\pa_x|^{\frac{4}{3}}\theta_r\|_{L^1_tL^2_{x,y}}
\lesssim\nu^{-\frac{1}{6}}\||\pa_x|^{\frac{2}{3}}(tu_e^{(2)}\pa_x\theta_L)\|_{L^1_tL^2_{x,y}}.
\end{align*}
This estimate, together with \eqref{J-1-es} and \eqref{J-5-es}, yields
\begin{align}\label{J-1-5-6-es}
K_1+K_3+K_6
\lesssim&\nu^{\frac{1}{3}}\||\pa_x|^{\frac{5}{3}}\theta_r\|_{L^2_tL^2_{x,y}}
+\nu^{\frac{1}{6}}\||\pa_x|^{\frac{4}{3}}\theta_r\|_{L^1_tL^2_{x,y}}\nonumber\\
\lesssim&\nu^{-\frac{1}{6}}\||\pa_x|^{\frac{2}{3}}(tu_e^{(2)}\pa_x\theta_L)\|_{L^1_tL^2_{x,y}}.
\end{align}

For $K_4$, using \eqref{fX-es-k} with $f=\theta_{r,k}$ and $\textbf{1}_{\nu k^2\geq 1}\nu^{-\frac{1}{3}}|k|^{\frac{1}{3}}\leq \textbf{1}_{\nu k^2\geq 1}
\nu^{-\frac{1}{6}}|k|^{\frac{2}{3}}$,  we have
\begin{align}
K_4\lesssim\nu^{-\frac{1}{6}}\||k|^{\frac{2}{3}}\theta_{r,k}\|_{L^2_k\mathcal{Y}}
\lesssim&\nu^{-\frac{1}{6}}\||k|^{\frac{2}{3}}(\widehat{tu_e^{(2)}\pa_x\theta_L})\|_{L^2_kL^1_tL^2_y} \nonumber\\
\lesssim&\nu^{-\frac{1}{6}}\||\pa_x|^{\frac{2}{3}}(tu_e^{(2)}\pa_x\theta_L)\|_{L^1_tL^2_{x,y}}.\label{J-3-es}
\end{align}
In view of \eqref{es-u-2-p-x-u-2} and \eqref{px-py-tL-es-inf}, we obtain
\begin{align}\label{J-1-3-5-6-es-0}
\||\pa_x|^{\frac{2}{3}}(tu_e^{(2)}\pa_x\theta_L)\|_{L^1_tL^2_{x,y}}
\lesssim&\|t|\pa_x|^{\frac{2}{3}}u_e^{(2)}\pa_x\theta_L\|_{L^1_tL^2_{x,y}}
+\|tu_e^{(2)}|\pa_x|^{\frac{5}{3}}\theta_L\|_{L^1_tL^2_{x,y}}\nonumber\\
\lesssim&\|\pa_x u_e\|_{L^2_tL^2_{x,y}}\big(\|t\pa_x\theta_L\|_{L^2_tL^\infty_{x,y}}
+\|t|\pa_x|^{\frac{5}{3}}\theta_L\|_{L^2_tL^\infty_{x,y}}\big)\nonumber\\
\lesssim&\|\omega_e\|_{X}\|t(1+\nu t^3)^{-1}\|_{L^2_t}F_0\lesssim\nu^{-\frac{1}{2}}\|\omega_e\|_{X}F_0.
\end{align}
Combining \eqref{J-1-5-6-es}--\eqref{J-1-3-5-6-es-0}, we deduce that
\begin{align}\label{J-1-3-5-6-es}
K_1+K_3+K_4+K_6\lesssim\nu^{-\frac{2}{3}}\|\omega_e\|_{X}F_0.
\end{align}

It remains to estimate $K_2+K_5$. Let $T_{j}=2^j\nu^{-\frac{1}{3}}$ with $j\geq 1$, and
\begin{align*}
K_{2,k}=\mathbf{1}_{\nu k^2\leq 1}\|\omega_{r,k}\|_{\mathcal{Y}}, \qquad  K_{5,k}=\nu^{-\frac{1}{3}}\mathbf{1}_{\nu k^2\leq 1}\||k|^{\frac{1}{3}}\theta_{r,k}\|_{\mathcal{Y}}.
\end{align*}
We decompose $\omega_{r,k}$ and $\theta_{r,k}$ as
$\omega_{r,k}=\sum_{j\geq 1}\omega_{r,k,j}$, $\theta_{r,k}=\sum_{j\geq 1}\theta_{r,k,j}$,
where
 \begin{align*}
\begin{cases}
\pa_t \omega_{r,k,j}-\nu(\pa_y^2-k^2)\omega_{r,k,j}+iky\omega_{r,k,j}=-ik\theta_{r,k,j}, t>T_{j-1}\\
\pa_t \theta_{r,k,j}-\nu(\pa_y^2-k^2)\theta_{r,k,j}+iky\theta_{r,k,j}=\widehat{tu_e^{(2)}\pa_x\theta_L}\chi_j, \\
\omega_{r,k,j}(t,k,\pm 1)=\theta_{r,k,j}(t,k,\pm 1)=0, \\
\omega_{r,k,j}|_{t=T_{j-1}}=\theta_{r,k,j}|_{t=T_{j-1}}=0,
\end{cases}
\end{align*}
with $\chi_j$ defined in \eqref{chi-t-j}. Using Lemma \ref{lem-3.3} with $f=\theta_{r,k,j}$, we obtain
\begin{align}\label{t-2-k-j-0}
\|\theta_{r,k,j}\|_{\mathcal{Y}(T_j,T)}+\nu^{\frac{1}{6}}|k|^{\frac{1}{3}}\|\theta_{r,k,j}\|_{L^2_t(T_j,T;L^2_y)}
\lesssim\|\theta_{r,k,j}(T_j)\|_{L^2_y}.
\end{align}
Applying Lemma \ref{lem-3.4} and equation $\eqref{equ: om2,the2k}_1$ with $f=\omega_{r,k}$, we have
\begin{align}\label{omega21k}
K_{2,k}=\mathbf{1}_{\nu k^2\leq 1}\|\omega_{r,k}\|_{\mathcal{Y}}\lesssim & \textbf{1}_{\nu k^2\leq 1}\nu^{-\frac{1}{6}}|k|^{\frac{2}{3}}\|\theta_{r,k}\|_{L^2_t(T_0,T;L^2_y)}.
\end{align}
According to the decomposition of $\theta_{r,k}$  and  \eqref{t-2-k-j-0}, we deduce
\begin{align}\label{theta2k}
\|\theta_{r,k}\|_{L^2_t(T_0,T;L^2_y)}
\lesssim&\Big\|\sum_{j\geq 1}\theta_{r,k,j}\Big\|_{L^2_t(T_0,T;L^2_y)}
\lesssim \sum_{j\geq 1}\|\theta_{r,k,j}\|_{L^2_t(T_{j-1},T;L^2_y)}\nonumber\\
\lesssim &\sum_{j\geq 1}\big(\|\theta_{r,k,j}\|_{L^2_t(T_{j-1},T_j;L^2_y)}+\|\theta_{r,k,j}\|_{L^2_t(T_{j},T;L^2_y)}\big)\nonumber\\
\lesssim &\sum_{j\geq 1}\big(\|\theta_{r,k,j}\|_{L^2_t(T_{j-1},T_j;L^2_y)}+\|\theta_{r,k,j}(T_j)\|_{L^2_y}\big).
\end{align}
Inserting \eqref{theta2k} into \eqref{omega21k}, we have
\begin{align}\label{k-geq-1-om-2-es}
K_{2,k}\lesssim &\textbf{1}_{\nu k^2\leq 1}\sum_{j\geq 1}\big(\nu^{-\frac{1}{6}}|k|^{\frac{2}{3}}\|\theta_{r,k,j}\|_{L^2_t(T_{j-1},T_j;L^2_y)}
+\nu^{-\frac{1}{3}}|k|^{\frac{1}{3}}\|\theta_{r,k,j}(T_j)\|_{L^2_y}\big)\nonumber\\
\lesssim &\textbf{1}_{\nu k^2\leq 1}\sum_{j\geq 1}\nu^{-\frac{1}{3}}|k|^{\frac{1}{3}}\Big(\nu^{\frac{1}{6}}|k|^{\frac{1}{3}}\|\theta_{r,k,j}\|_{L^2_t(T_{j-1},T_j;L^2_y)}
+\|\theta_{r,k,j}\|_{L^\infty_t(T_{j-1},T_j;L^2_y)}\Big).
\end{align}
By \eqref{t-2-k-j-0}, we derive
\begin{align*}
\nu^{-\frac{1}{3}}|k|^{\frac{1}{3}}\|\theta_{r,k}\|_{\mathcal{Y}}
\lesssim &\sum_{j\geq 1}\big(\nu^{-\frac{1}{3}}|k|^{\frac{1}{3}}\|\theta_{r,k,j}\|_{\mathcal{Y}(T_{j-1},T_j)}
+\nu^{-\frac{1}{3}}|k|^{\frac{1}{3}}\|\theta_{r,k,j}\|_{\mathcal{Y}(T_{j},T)}\big)\nonumber\\
\lesssim &\sum_{j\geq 1}\big(\nu^{-\frac{1}{3}}|k|^{\frac{1}{3}}\|\theta_{r,k,j}\|_{\mathcal{Y}(T_{j-1},T_j)}
+\nu^{-\frac{1}{3}}|k|^{\frac{1}{3}}\|\theta_{r,k,j}(T_j)\|_{L^2_y}\big)\nonumber\\
\lesssim &\sum_{j\geq 1}\nu^{-\frac{1}{3}}|k|^{\frac{1}{3}}\|\theta_{r,k,j}\|_{\mathcal{Y}(T_{j-1},T_j)}.
\end{align*}
This inequality implies
\begin{align}\label{k-geq-1-t-2-es}
K_{5,k}=\textbf{1}_{\nu k^2\leq 1}\nu^{-\frac{1}{3}}|k|^{\frac{1}{3}}\|\theta_{r,k}\|_{\mathcal{Y}}
\lesssim \textbf{1}_{\nu k^2\leq 1}\sum_{j\geq 1}\nu^{-\frac{1}{3}}|k|^{\frac{1}{3}}\|\theta_{r,k,j}\|_{\mathcal{Y}(T_{j-1},T_j)}.
\end{align}
Define
\begin{align*}
\widehat{tu_e^{(2)}\pa_x\theta_L}=\int_{\mathbb{R}}iltu_{e,k-l}^{(2)}\theta_{L,l}dl:=\int_{\mathbb{R}} F_{k,l} dl.
\end{align*}
We decompose $\theta_{r,k,j}$ as $\theta_{r,k,j}=\int_{\mathbb{R}}\theta_{r,k,j,l}\,dl$, where
 \begin{align*}
\begin{cases}
\pa_t \theta_{r,k,j,l}-\nu(\pa_y^2-k^2)\theta_{r,k,j,l}+iky\theta_{r,k,j,l}=F_{k,l}\chi_j, \\
\theta_{r,k,j,l}(\pm 1)=0, \qquad \theta_{r,k,j,l}|_{t=T_{j-1}}=0.
\end{cases}
\end{align*}
Using Lemma~\ref{lem-3.6} and Lemma~\ref{lem-3.7} with $f=\theta_{r,k,j,l}$, we derive
\begin{align}
&\nu^{\frac{1}{6}}|k|^{\frac{1}{3}}\|\theta_{r,k,j,l}\|_{L^2_t(T_{j-1},T_j;L^2_y)}
+\|\theta_{r,k,j,l}\|_{\mathcal{Y}(T_{j-1},T_j)}\label{1-6-1-3-es}\\
\lesssim &\Big(2^{\frac{3}{2}j}(|l|+1) |k-l|^{-\frac{1}{2}}\langle k-l\rangle^{-\frac{1}{2}}+\textbf{1}_{|k|\geq 1}\nu^{\frac{1}{6}}|k|^{-\frac{2}{3}}\Big)
\|(\pa_y,\langle k-l\rangle)(e^{ilyt}F_{k,l})\|_{L^2_t(T_{j-1},T_j;L^2_y)}.\nonumber
\end{align}
Noting that
\begin{align}\label{K-2-k-K-4-k-0}
&\nu^{\frac{1}{6}}|k|^{\frac{1}{3}}\|\theta_{r,k,j}\|_{L^2_t(T_{j-1},T_j;L^2_y)}
+\|\theta_{r,k,j}\|_{\mathcal{Y}(T_{j-1},T_j)}\nonumber\\
\lesssim&\int_{\mathbb{R}}(\nu^{\frac{1}{6}}|k|^{\frac{1}{3}}\|\theta_{r,k,j,l}\|_{L^2_t(T_{j-1},T_j;L^2_y)}
+\|\theta_{r,k,j,l}\|_{\mathcal{Y}(T_{j-1},T_j)})dl
\end{align}
Combining \eqref{k-geq-1-om-2-es}--\eqref{K-2-k-K-4-k-0} produces
\begin{align}\label{K-2-k-K-4-k}
K_{2,k}+K_{5,k}
\lesssim &\textbf{1}_{\nu k^2\leq 1}\sum_{j\geq 1}\nu^{-\frac{1}{3}}|k|^{\frac{1}{3}}\int_{\mathbb{R}}\Big(2^{\frac{3}{2}j}(|l|+1) |k-l|^{-\frac{1}{2}}\langle k-l\rangle^{-\frac{1}{2}}+\textbf{1}_{|k|\geq 1}\nu^{\frac{1}{6}}|k|^{-\frac{2}{3}}\Big)\times\nonumber\\
&\qquad \qquad\qquad \qquad\|(\pa_y,\langle k-l\rangle)(e^{ilyt}F_{k,l})\|_{L^2_t(T_{j-1},T_j;L^2_y)}dl.
\end{align}
By H\"older's inequality, we have
\begin{align}
&\|(\pa_y,\langle k-l\rangle)(e^{ilyt}F_{k,l})\|_{L^2_t(T_{j-1},T_j;L^2_y)} \nonumber \\
\leq & 2^{j}\nu^{-\frac{1}{3}}\|(\pa_y, \langle k-l\rangle)u_{e,k-l}^{(2)}\|_{L^2_t(T_{j-1},T_j)L^2_y}\|(\pa_y, 1)l(e^{ilyt}\theta_{L,l})\|_{L^\infty_t(T_{j-1},T_j)L^\infty_y}.\label{esti: Fkl}
\end{align}
We use \eqref{tL-es-inf-k} with $s=1$ to obtain
\begin{align}\label{thetaL1}
\|(\pa_y, 1)l(e^{ilyt}\theta_{L,l})\|_{L^\infty_t(T_{j-1},T_j)L^\infty_y}\lesssim \min\{|l|,|l|^{\frac{1}{2}}\}e^{-c2^{j}|l|^{\frac{2}{3}}}B_{2}\lesssim \min\{|l|,|l|^{-\frac{1}{2}}\}e^{-c2^{j}|l|^{\frac{2}{3}}}B_{3}.
\end{align}
For $|k-l|\geq 1$, inserting \eqref{thetaL1} into \eqref{esti: Fkl}, we deduce
\begin{align}\label{py-k-l-F-k-l-es-1}
&\|(\pa_y, k-l)(e^{ilyt}F_{k,l})\|_{L^2_t(T_{j-1},T_j)L^2_y}\nonumber\\
\lesssim& 2^{j}\nu^{-\frac{1}{3}}\|(\pa_y, k-l)u_{e,k-l}^{(2)}\|_{L^2_t(T_{j-1},T_j)L^2_y}\min\{|l|,|l|^{-\frac{1}{2}}\}e^{-c2^{j}|l|^{\frac{2}{3}}}B_{3}\nonumber\\
\lesssim& 2^{-\frac{15}{8}j}\nu^{-\frac{1}{3}}\|(\pa_y, k-l)u_{e,k-l}^{(2)}\|_{L^2_t(T_{j-1},T_j)L^2_y}\min\{|l|^{-\frac{11}{12}},|l|^{-\frac{29}{12}}\}e^{-\frac{c}{2} 2^{j}|l|^{\frac{2}{3}}}B_{3},
\end{align}
where we used that
$$ 2^{\frac{23}{8}j}|l|^{\frac{23}{12}} e^{-\frac{c}{2} 2^{j}|l|^{\frac{2}{3}}}\leq C. $$
Due to $|k-l|\geq 1$, we have
\begin{align}\label{py-k-l-F-k-l-es-2}
&\nu^{-\frac{1}{3}}|k|^{\frac{1}{3}}\Big(2^{\frac{3}{2}j}(|l|+1) |k-l|^{-\frac{1}{2}}\langle k-l\rangle^{-\frac{1}{2}}+\textbf{1}_{|k|\geq 1}\nu^{\frac{1}{6}}|k|^{-\frac{2}{3}}\Big)\nonumber\\
\lesssim&\nu^{-\frac{1}{3}}2^{\frac{3}{2}j}(|l|^{\frac{1}{3}}+|k-l|^{\frac{1}{3}})(|l|+1)|k-l|^{-1} +\nu^{-\frac{1}{6}}
\lesssim \nu^{-\frac{1}{3}}2^{\frac{3}{2}j}(|l|^{\frac{4}{3}}+1).
\end{align}
For $|k-l|\leq 1$, by the Poincar\'e inequality $$\|u_{e,k-l}^{(2)}\|_{L^2_t(T_{j-1},T_j)L^2_y}\lesssim
\|\pa_y u_{e,k-l}^{(2)}\|_{L^2_t(T_{j-1},T_j)L^2_y}$$ and the divergence-free condition, we get
\begin{align*}
\|(\pa_y, 1)u_{e,k-l}^{(2)}\|_{L^2_y}\lesssim \|\pa_y u_{e,k-l}^{(2)}\|_{L^2_y}=\|(k-l) u_{e,k-l}^{(1)}\|_{L^2_y}.
\end{align*}
Putting this estimate and \eqref{thetaL1} into \eqref{esti: Fkl}, we yield
\begin{align}\label{py-1-F-k-l-es-1}
&\|(\pa_y, 1)(e^{ilyt}F_{k,l})\|_{L^2_t(T_{j-1},T_j)L^2_y}\nonumber\\
\lesssim& 2^{j}\nu^{-\frac{1}{3}}|k-l|^{\frac{1}{2}}\||k-l|^{\frac{1}{2}}u_{e,k-l}^{(1)}\|_{L^2_t(T_{j-1},T_j)L^2_y}
\min\{|l|,|l|^{-\frac{1}{2}}\}e^{-c2^{j}|l|^{\frac{2}{3}}}B_{3}\nonumber\\
\lesssim& 2^{-\frac{15}{8}j}\nu^{-\frac{1}{3}}|k-l|^{\frac{1}{2}}\||k-l|^{\frac{1}{2}}u_{e,k-l}^{(1)}\|_{L^2_t(T_{j-1},T_j)L^2_y}
\min\{|l|^{-\frac{11}{12}},|l|^{-\frac{29}{12}}\}e^{-\frac{c}{2} 2^{j}|l|^{\frac{2}{3}}}B_{3}.
\end{align}
Thanks to $|k-l|\leq 1$, we also get
\begin{align}\label{py-1-F-k-l-es-2}
&|k|^{\frac{1}{3}}\Big(2^{\frac{3}{2}j}(|l|+1) |k-l|^{-\frac{1}{2}}\langle k-l\rangle^{-\frac{1}{2}}+\textbf{1}_{|k|\geq 1}\nu^{\frac{1}{6}}|k|^{-\frac{2}{3}}\Big)\nonumber\\
\lesssim&2^{\frac{3}{2}j}(|l|^{\frac{1}{3}}+|k-l|^{\frac{1}{3}})(|l|+1) |k-l|^{-\frac{1}{2}}+\nu^{\frac{1}{6}}\nonumber\\
\lesssim&2^{\frac{3}{2}j}(|l|^{\frac{1}{3}}+1)(|l|+1)|k-l|^{-\frac{1}{2}}+\nu^{\frac{1}{6}}\nonumber\\
\lesssim&2^{\frac{3}{2}j}(|l|^{\frac{4}{3}}+1)|k-l|^{-\frac{1}{2}}.
\end{align}
Define $$D_{k-l}:=\|(|k-l|^{\frac{1}{2}}+|k-l|)u_{e,k-l}\|_{L^2_t(T_0,T)L^2_y}.$$
Inserting \eqref{py-k-l-F-k-l-es-1}--\eqref{py-1-F-k-l-es-2} into \eqref{K-2-k-K-4-k} and using $$(|l|^{\frac{4}{3}}+1)\min\{|l|^{-\frac{11}{12}},|l|^{-\frac{29}{12}}\}\lesssim\min\{|l|^{-\frac{11}{12}},|l|^{-\frac{13}{12}}\},$$
we derive
\begin{align}\label{K-2-k-K-5-k}
K_{2,k}+K_{5,k}
\lesssim&\textbf{1}_{\nu k^2\leq 1}\sum_{j\geq 1}\nu^{-\frac{2}{3}}2^{-\frac{3}{8}j}
\int_{\mathbb{R}}D_{k-l}(|l|^{\frac{4}{3}}+1)\min\{|l|^{-\frac{11}{12}},|l|^{-\frac{29}{12}}\}B_{3}dl\nonumber\\
\lesssim&\textbf{1}_{\nu k^2\leq 1}\sum_{j\geq 1}\nu^{-\frac{2}{3}}2^{-\frac{3}{8}j}\int_{\mathbb{R}}D_{k-l}\min\{|l|^{-\frac{11}{12}},|l|^{-\frac{13}{12}}\}B_{3}dl.
\end{align}
Integrating in $k$ on both side of \eqref{K-2-k-K-5-k},  we obtain 
\begin{align}\label{J-2-4-es}
K_2+K_5\lesssim&\nu^{-\frac{2}{3}}\|D_k\|_{L^2_k}\Big(\int_{|l|\leq 1}|l|^{-\frac{11}{12}}B_{3}dl+\int_{|l|\geq 1}|l|^{-\frac{13}{12}}B_{3}dl\Big)\nonumber\\
\lesssim&\nu^{-\frac{2}{3}}\|D_k\|_{L^2_k}F_0\lesssim \nu^{-\frac{2}{3}}\|\omega_e\|_{X}F_0.
\end{align}
Combining \eqref{J-1-3-5-6-es} and \eqref{J-2-4-es}, we complete the proof of Proposition~\ref{prop-5.5}.
\end{proof}

\subsection{Nonlinear estimates of \texorpdfstring{$(\omega_{e}, \theta_{e})$}{omega\_e, theta\_e}}
In what follows, we use the bootstrap argument to establish the estimates for $\omega_e$ and $\theta_e$, which, combined with Proposition~\ref{prop-2.4}, completes the proof of Theorem~\ref{th-1.1}.

\begin{proposition}\label{prop-5.6}
Let $(\omega_{e}, \theta_{e})$ be the solution of \eqref{ote-eq}. For $t> T_0=\nu^{-\frac{1}{6}}$, we have
\begin{align*}
\|\omega_{e}\|_{X}\lesssim \nu^{\frac{2}{3}}|\ln\nu|^{3}, \qquad \|\theta_{e}\|_{X}\lesssim \nu|\ln\nu|^{-1}.
\end{align*}
\end{proposition}

\begin{proof}
By Propositions \ref{prop-5.2}--\ref{prop-5.5}, we obtain
\begin{align}\label{E-2-j-es}
\mathcal{E}_{2}\lesssim&\|\omega_\text{b}\|_{X}+\nu^{-\frac{1}{3}}|\ln\nu|^4\|\theta_\text{b}\|_{X}
+\nu^{-\frac{1}{3}}\||\pa_x|^{\frac{1}{3}}\theta_\text{b}\|_{X}\nonumber\\
&+\|\omega_r\|_{X}+\|\omega_{\mathfrak{o}}\|_{X}+\nu^{-\frac{1}{3}}|\ln\nu|^4\|\theta_r\|_{X}
+\nu^{-\frac{1}{3}}\||\pa_x|^{\frac{1}{3}}\theta_r\|_{X}\nonumber\\
\leq&C\Big(\mathcal{E}_{2,0}+\nu^{-\frac{1}{3}}\mathcal{E}_{2}(E_0+\nu^{-\frac{1}{3}}|\ln\nu|^4F_0)
+\nu^{-\frac{5}{9}}\mathcal{E}_{2}^2\Big).
\end{align}
\textbf{Bootstrap argument:}
 Let $T$ be the maximal time such that
\begin{align}\label{es-E2}
\mathcal{E}_{2}=\|\omega_{e}\|_{X(T_0,T)}+\nu^{-\frac{1}{3}}|\ln\nu|^4\|\theta_{e}\|_{X(T_0,T)}
+\nu^{-\frac{1}{3}}\||\pa_x|^{\frac{1}{3}}\theta_{e}\|_{X(T_0,T)}\leq c^{\frac{1}{2}}\nu^{\frac{2}{3}}|\ln\nu|^{3},
\end{align}
where  $0<c<1$ satisfies $Cc^{\frac{1}{2}}\leq 10^{-4}$.
In light of \eqref{E-2-j-es}, the definition of $\mathcal{E}_{2,0}$ and  
$$E_0+\nu^{-\frac{1}{3}}|\ln\nu|^4F_0\leq c\nu^{\frac{1}{3}},$$ we obtain
\begin{align*}
 \mathcal{E}_{2}\leq& C\nu^{-\frac{1}{3}}\mathcal{E}_{2}(E_0+\nu^{-\frac{1}{3}}|\ln\nu|^4F_0)+C\nu^{-\frac{5}{9}}\mathcal{E}_{2}^2
 +C\nu^{\frac{1}{3}}|\ln\nu|^{3}(E_0+\nu^{-\frac{1}{3}}|\ln\nu|^4F_0)\nonumber\\
 \leq& Cc\mathcal{E}_{2}+Cc^{\frac{1}{2}}\nu^{\frac{1}{9}}|\ln\nu|^{3}\mathcal{E}_{2}+Cc\nu^{\frac{2}{3}}|\ln\nu|^{3}.
\end{align*}
Inserting \eqref{es-E2} into the above estimate, we obtain
\begin{align*}
 \mathcal{E}_{2} \leq& \frac{1}{2}c^{\frac{1}{2}}\nu^{\frac{2}{3}}|\ln\nu|^{3}.
\end{align*}
This contradicts the definition of $T$. Hence $T=+\infty$, which completes the proof of Proposition~\ref{prop-5.6}.
\end{proof}

Combining  Proposition~\ref{prop-2.4}, Proposition~\ref{prop-4.1} with Proposition~\ref{prop-5.6}, we conclude Theorem \ref{th-1.1}.

\begin{Remark}\label{rem-5.1}
For the domain $\mathbb{T}\times[-1,1]$, let
\begin{align*}
\tilde{\mathcal{E}}_{2}=\|\omega_{e}\|_{X(T_0,T)}+\nu^{-\frac{1}{3}}\|\theta_{e}\|_{X(T_0,T)}
+\nu^{-\frac{1}{3}}\||\pa_x|^{\frac{1}{3}}\theta_{e}\|_{X(T_0,T)}.
\end{align*}
Analogue to the proof of Propositions \ref{prop-5.2}--\ref{prop-5.6}, we obtain
\begin{align*}
 \tilde{\mathcal{E}}_{2} \lesssim& \nu^{\frac{2}{3}}|\ln\nu|.
\end{align*}
Combining with Remark \ref{rem-4.1}, we conclude
\begin{align*}
  \|\omega\|_{L^2}+\|\pa_x\nabla\Delta^{-1}\omega\|_{L^2L^2}+\nu^{-\frac13}\|\theta\|_{L^2} \leq& C\nu^{\frac{1}{3}}.
\end{align*}
\end{Remark}

\appendix
\section{elliptic estimate}
\begin{lemma}\label{lem-A.2}
Let $u_k=(\pa_y,-ik)(\pa_y^2-k^2)^{-1}\omega_k$ with $\omega_k|_{\pm 1} = 0$, then
\begin{align}\label{lem-A.2-es}
\|u_k\|_{L^\infty_y}+\|u_k\|_{H^1_y}\lesssim\|\omega_k\|_{L^2_y}.
\end{align}
\end{lemma}

\begin{proof}
Define $\phi_k=(\pa_y^2 - k^2)^{-1}\omega_k$ with $\phi_k|_{\pm 1}=0$. Then, integrating by parts, we obtain
\begin{align}\label{py-phi-k-2-y-es}
\|\pa_y\phi_k\|_{L^2_y}^2 + k^2\|\phi_k\|_{L^2_y}^2 = \langle \omega_k, \phi_k \rangle \leq \|\omega_k\|_{L^2_y}\|\phi_k\|_{L^2_y}.
\end{align}
The Poincar\'e inequality $\|\phi_k\|_{L^2_y} \lesssim \|\pa_y\phi_k\|_{L^2_y}$ yields
\begin{align*}
\|\pa_y\phi_k\|_{L^2_y}^2 \leq \|\omega_k\|_{L^2_y}\|\pa_y\phi_k\|_{L^2_y}.
\end{align*}
Together with \eqref{py-phi-k-2-y-es}, it implies
\begin{align}\label{A-2-1}
\|u_k\|_{L^2_y} \lesssim \|\pa_y\phi_k\|_{L^2_y} + |k|\|\phi_k\|_{L^2_y} \lesssim \|\omega_k\|_{L^2_y}.
\end{align}
From the identity
\begin{align*}
\|\omega_k\|_{L^2_y}^2 = \|\pa_y^2\phi_k\|_{L^2_y}^2 + 2k^2\|\pa_y\phi_k\|_{L^2_y}^2 + k^4\|\phi_k\|_{L^2_y}^2,
\end{align*}
we deduce
\begin{align}\label{A-2-2}
|k|^2\|u_k\|_{L^2_y}^2 \lesssim |k|^2\|\pa_y\phi_k\|_{L^2_y}^2 + |k|^4\|\phi_k\|_{L^2_y}^2 \lesssim \|\omega_k\|_{L^2_y}^2.
\end{align}
Owing to $iku_k^{(1)}+\pa_yu_k^{(2)}=0$, we have
\begin{align}\label{A-2-3}
\|\pa_y u_k^{(1)}\|_{L^2_y} + \|\pa_y u_k^{(2)}\|_{L^2_y} \lesssim \|\omega_k\|_{L^2_y} + |k|\|u_k\|_{L^2_y} \lesssim \|\omega_k\|_{L^2_y}.
\end{align}
The  Sobolev embedding $H^1_y \hookrightarrow L^\infty_y$ yields
\begin{align}\label{A-2-4}
\|u_k\|_{L^\infty_y} \lesssim \|u_k\|_{H^1_y} \lesssim \|\omega_k\|_{L^2_y}.
\end{align}
Combining \eqref{A-2-1}--\eqref{A-2-4}, one completes the proof of Lemma \ref{lem-A.2}.
\end{proof}

\begin{lemma}\label{lem-A.1}
Assume that $u = (\pa_y, -\pa_x)\Delta^{-1} \omega$ with $\omega|_{\pm 1} = 0$. Then for $s \in [0,1]$, 
\begin{align}
&\| |\pa_x|^s u \|_{L^2_{x,y}} \lesssim \|\omega\|_{L^2_{x,y}}, \label{es-u-p-x-u-2}\\
&\| |\pa_x|^s u^{(2)} \|_{L^2_{x,y}} \lesssim \|\pa_x u\|_{L^2_{x,y}}. \label{es-u-2-p-x-u-2}
\end{align}
\end{lemma}

\begin{proof}
From Plancherel's theorem and \eqref{lem-A.2-es}, we obtain
\begin{align*}
\|u\|_{L^2_{x,y}} + \|\pa_x u\|_{L^2_{x,y}} \lesssim \|u_k\|_{L^2_k L^2_y} + \|k u_k\|_{L^2_k L^2_y}
 \lesssim \|\omega_k\|_{L^2_k L^2_y} \lesssim \|\omega\|_{L^2_{x,y}}.
\end{align*}
This inequality implies
\begin{align}\label{p-x-u-L-2-x-y}
\| |\pa_x|^s u \|_{L^2_{x,y}} \lesssim \|u\|_{L^2_{x,y}}^{1-s} \|\pa_x u\|_{L^2_{x,y}}^{s} \lesssim \|\omega\|_{L^2_{x,y}}.
\end{align}
Using the Poincar\'e inequality and $\div u=0$, we find
\begin{align}\label{p-x-u-2-L-2-x-y}
\| |\pa_x|^s u^{(2)} \|_{L^2_{x,y}} &\lesssim \|u^{(2)}\|_{L^2_{x,y}}^{1-s} \|\pa_x u^{(2)}\|_{L^2_{x,y}}^{s}
\lesssim \|\pa_y u^{(2)}\|_{L^2_{x,y}}^{1-s} \|\pa_x u^{(2)}\|_{L^2_{x,y}}^{s} \nonumber\\
&= \|\pa_x u^{(1)}\|_{L^2_{x,y}}^{1-s} \|\pa_x u^{(2)}\|_{L^2_{x,y}}^{s} \lesssim \|\pa_x u\|_{L^2_{x,y}}.
\end{align}
Combining \eqref{p-x-u-L-2-x-y} and \eqref{p-x-u-2-L-2-x-y}, we completes the proof of Lemma~\ref{lem-A.1}.
\end{proof}

\medskip
\noindent
\textbf{Acknowledgements.}  Q. Chen was partially supported by National Natural Science Foundation of China (No: 12471149). Z. Li was supported by National Natural Science Foundation of China (No: 12601421). C. Miao was supported
by National Natural Science Foundation of China (No: 12371095 and 12531005).

\noindent\textbf{Data Availability Statement:}
Data sharing is not applicable to this article as no datasets were generated or analysed during the current study.

\noindent\textbf{Conflict of Interest:}
The authors declare that they have no conflict of interest.


\begin{thebibliography}{99}

\bibitem{A-25} R. Arbon, {\it Enhanced dissipation, Taylor dispersion, and inviscid damping of Couette flow in the Boussinesq system on the plane}, Nonlinearity 38(2025), 105011.

\bibitem{BBCD-23} J. Bedrossian, R. Bianchini, M. Coti Zelati, and M. Dolce, {\it Nonlinear inviscid damping and shear-buoyancy instability in the two-dimensional Boussinesq equations}, Commun. Pure Appl. Math. 76 (2023), 3685-3768.

\bibitem{BGM-19} J. Bedrossian, P. Germain, and N. Masmoudi, {\it Stability of the Couette flow at high Reynolds numbers in two dimensions and three dimensions}, Bull. Amer. Math. Soc. 56 (2019), 373-414.

\bibitem{CLWZ-20} Q. Chen, T. Li, D. Wei, and Z. Zhang, {\it Transition threshold for the 2-D Couette flow in a finite channel}, Arch.
Ration. Mech. Anal. 238(2020), 125-183.

\bibitem{CWZ-24} Q. Chen, D. Wei, and Z. Zhang, {\it Transition threshold for the 3D Couette flow in a finite channel}, Mem. Amer. Math. Soc. 296(2024), 1478.

\bibitem{CLM-26} Q. Chen, Z. Li, and C. Miao, {\it The optimal transition threshold for the 2D Couette flow in the infinite channel}, Commun. Math. Phys. 407(2026), 182.

\bibitem{CWY-25} Y. Chen, W. Wang, and G. Yang, {\it Stability threshold of Couette flow for Boussinesq equations in $\mathbb{R}^2$},
arXiv:2508.11908.

\bibitem{DR-04} P. G. Drazin and W. H. Reid, {\it Hydrodynamic Stability}, Cambridge Mathematical Library. Cambridge University
Press, 2 edition, 2004.

\bibitem{DWZ-21} W. Deng, J. Wu, and P. Zhang, {\it Stability of Couette flow for 2D Boussinesq system with vertical dissipation}, J. Funct. Anal. 281(2021), 109255.

\bibitem{LWZ-25} T. Liang, J. Wu, and X. Zhai, {\it Stability threshold of the 2D Boussinesq system near Couette flow in an infinite
channel}, arXiv: 2512.01159.

\bibitem{LWZ-26} T. Liang, J. Wu, and X. Zhai, {\it Stability threshold analysis of the Boussinesq system with general viscosity near the Couette flow}, Sci. China Math. (2026).

\bibitem{MZZ-23} N. Masmoudi, C. Zhai, and W. Zhao, {\it Asymptotic stability for two-dimensional Boussinesq systems around the Couette flow in a finite channel}, J. Funct. Anal. 284(2023), 109736.

\bibitem{NZ-25} B. Niu and W. Zhao, {\it Improved stability threshold of the two-dimensional Couette flow for Navier-Stokes
Boussinesq systems via quasi-linearization}, Comm. Pure Appl. Anal. 31(2026), 38-64.

\bibitem{R-83} O. Reynolds, {\it An experimental investigation of the circumstances which determine whether the motion of water shall be direct or sinuous and of the law of resistance in parallel channels}, Phil. Trans. R. Soc. 174(1883), 935-982.

\bibitem{RW-26-1} X. Ren and D. Wei, {\it On the stability threshold of Couette flow for 2D Boussinesq equations}, Nonlinear Anal. Real World Appl. 87(2026), 104421.

\bibitem{RW-26-2} X. Ren and D. Wei, {\it Transition threshold of Couette flow for 2D Boussinesq equations}, J. Funct. Anal. 290(2026), 111383.

\bibitem{TTRD-93} L. Trefethen, A. Trefethen, S. Reddy, and T. Driscoll, {\it Hydrodynamic stability without eigenvalues}, Science, 261(1993), 578-584.

\bibitem{WW-25} G. Wang and L. Wang, {\it On asymptotic stability of Couette flow for 2-D Boussinesq system in whole space via
Green's function}, Discrete Contin. Dyn. Syst. Ser. 30(2025), 3014-3041.

\bibitem{WZ-21} D. Wei and Z. Zhang, {\it Transition Threshold for the 3D Couette Flow
in Sobolev Space}, Commun. Pure Appl. Math. 74(2021),  2398-2479.

\bibitem{WZ-26} D. Wei and Z. Zhang, {\it Asymptotic stability threshold of the 2D Couette flow in a finite channel},  Ann. Inst. H.
Poincar\'e C Anal. Non Lin\'eaire. 43(2026), 501-540.

\bibitem{ZZ-23} Z. Zhang and R. Zi, {\it Stability threshold of Couette flow for 2D Boussinesq equations in Sobolev spaces}, J. Math.
Pures Appl. 179(2023), 123-182.

\bibitem{ZZ-23-2} C. Zhai and W. Zhao, {\it Stability threshold of the Couette flow for Navier-Stokes Boussinesq system with large Richardson number $\gamma^2>\frac{1}{4}$}, SIAM J. Math. Anal. 55(2023), 1284-1318.

\bibitem{Z-23} C. Zillinger, {\it On echo chains in the linearized Boussinesq equations around traveling waves}, SIAM J. Math. Anal. 55 (2023), 5127-5188.


\end{thebibliography}
\end{document}